\documentclass[a4paper,10pt,reqno]{amsart}
\usepackage{amsmath,amsthm,amssymb,amsfonts,mathrsfs,color}
\usepackage{graphics} 
\usepackage{enumerate}
\usepackage{ulem}

\usepackage{graphicx,pgf,tikz}
\usetikzlibrary{decorations}
\tikzstyle xyax=[thin]
\tikzstyle mlin=[thick]
\tikzstyle slin=[]

\graphicspath{{ }}

\usepackage{ulem}
\usepackage[hyperindex]{hyperref}
\usepackage{epsfig}
\usepackage{color}

\numberwithin{equation}{section}

\newtheorem{theorem}{Theorem}[section]
\newtheorem{lemma}[theorem]{Lemma}
\newtheorem{proposition}[theorem]{Proposition}

\theoremstyle{definition}
\newtheorem{definition}[theorem]{Definition}
\newtheorem{remark}[theorem]{Remark}

\newcommand{\hel} {
\hskip2.5pt{\vrule height7pt width.5pt depth0pt}
\hskip-.2pt\vbox{\hrule height.5pt width7pt depth0pt}
\, }
\newcommand{\restr}{\hel}
\newcommand{\R}{\mathbf{R}}
\newcommand{\N}{\mathbf{N}}

\renewcommand{\S}{\mathbf{S}}

\newcommand\F{\mathcal{F}}
\def\HH{\mathcal{H}}
\def\LL{\mathcal{L}}
\newcommand\G{\mathcal{G}}

\newcommand{\eps}{\varepsilon}

\newcommand{\wto}{\rightharpoonup}

\newcommand{\Om}{\Omega}

\renewcommand{\phi}{\varphi}

\newcommand\lt{\left}
\newcommand\rt{\right}

\def\sdist{\mathbf{sd}}

\usepackage{xcolor}

\title[Curvature penalization of strongly anisotropic interfaces models]{Curvature penalization of  strongly anisotropic interfaces models and their phase-field approximation}
\author{J.-F. Babadjian, B. Buet and M. Goldman}

\address[J.-F. Babadjian]{Universit\'e Paris Saclay, CNRS,  Laboratoire de math\'ematiques d'Orsay, 91405, Orsay, France}
\email{jean-francois.babadjian@universite-paris-saclay.fr}

\address[B. Buet]{Universit\'e Paris Saclay, INRIA, CNRS,  Laboratoire de math\'ematiques d'Orsay, 91405, Orsay, France}
\email{blanche.buet@universite-paris-saclay.fr}

\address[M. Goldman]{CMAP, CNRS, \'Ecole polytechnique, Institut Polytechnique de Paris, 91120 Palaiseau, France}
\email{michael.goldman@cnrs.fr}

\begin{document}

\begin{abstract}
This paper studies the effect of anisotropy on sharp or diffuse interfaces models. When the surface tension is a convex function of the normal to the interface, the anisotropy is said to be weak. This usually ensures the lower semicontinuity of the associated energy. If, however, the surface tension depends on the normal in a nonconvex way, this so-called strong anisotropy may lead to instabilities related to the lack of lower semicontinuity of the functional. We investigate the regularizing effects of adding a  higher order term of Willmore type to the energy. We consider  two types of problems. The first one is an anisotropic nonconvex generalization of the perimeter, and the second one is an anisotropic nonconvex Mumford-Shah functional. In both cases,  lower semicontinuity properties of the energies with respect to a natural mode of convergence are established, as well as $\Gamma$-convergence type results by means of a phase field approximation. In comparison with related results for curvature dependent energies, one of the original aspects of our work is that, in the  context of free discontinuity problems, we are able to consider singular structures such as crack-tips or multiple junctions.
\end{abstract}
\maketitle

\tableofcontents

\section{Introduction}

Anisotropic sharp interface models involve a surface tension density depending on the orientation of the unknown interface. It thus contains preferable directions,  leading to crystalline surfaces due to the formation of facets, corner or wrinkling. When the surface tension is a convex function of the normal to the interface, the anisotropy is said to be weak, and the associated variational problem or the underlying PDE is in general well-behaved. If however, the anisotropy is strong, which means that the surface tension is a non convex function of the normal, the problem becomes unstable since several directions might be forbidden. From a mathematical standpoint, it corresponds to a loss of ellipticity of the underlying set of PDEs, or a lack of lower semicontinuity of the associated energy. It naturally arises in many applications  such as phase transitions or epitaxial growth  \cite{GurtJabb,fonseca2015motion,Lowen,philippe2022regularized}. In the context of fracture mechanics, tearing experiments on brittle thin sheets with strongly anisotropic surface energy lead to observed crack path which happen to be inconsistent with those analytically predicted (see e.g. \cite{BinMau} and references therein).

From a variational point of view, which is the one adopted in the present work, a generic and formal way to formulate strongly anisotropic interfacial problems consists in minimizing an energy functional of the form
$$\Gamma \mapsto \mathcal E(\Gamma):=\int_\Gamma \phi(\nu_\Gamma)\, d\HH^{d-1},$$
among all possible $(d-1)$-dimensional hypersurfaces $\Gamma \subset \Om$ (say without boundary) subjected to suitable external constraints. In the previous expression, $\Om \subset \R^d$ is a bounded open set with $d=2$, $3$, $\HH^{d-1}$ stands for the $(d-1)$-dimensional Hausdorff measure, $\nu_\Gamma$ is a normal vector to $\Gamma$ (which is well defined provided $\Gamma$ is smooth enough) and $\phi:\R^d \to \R^+$ is a continuous, $1$-homogeneous, even but not necessarily convex surface tension. Minimizing this kind of interfacial energies might lead to the formation of microstructures due to fast oscillations of minimizing sequences $\{\Gamma_n\}_{n \in \N}$. Indeed, it might be energetically convenient to allow the normals $\nu_{\Gamma_n}$ to highly oscillate  between finitely many fixed directions, in such a way that the sequence of sets $\{\Gamma_n\}_{n \in \N}$  `weakly converges' to a limit (generalized) set $\Gamma$ whose energy is strictly lower than the value of the infimum. This is a typical behavior of variational problems lacking lower semicontinuity.

A possible remedy consists in relaxing the original energy $\mathcal E$, i.e., replacing it by its lower semicontinuous envelope $\overline{\mathcal E}$ with respect to a natural topology for which the energy is coercive. In the formal framework described above, it consists in replacing $\phi$ by its convex envelope $\phi^{**}$ (see e.g. \cite{Lussardi} when $\Gamma$ is the boundary of a set) so that  
$$\overline{\mathcal E}(\Gamma)=\int_\Gamma \phi^{**}(\nu_\Gamma)\, d\HH^{d-1}.$$
The so-called relaxed problem now becomes well-posed in a suitable mathematical framework, in the sense that minimizing sequences for $\mathcal E$ converges to (generalized) solutions of $\overline{\mathcal E}$ and $\inf \mathcal E=\min \overline{\mathcal E}$. From a mathematical point of view, the relaxation procedure is perfectly satisfactory because high oscillations of minimizing sequences are ruled out. However, it has the drawback of modifying the underlying physics by suppressing the possible
nucleation of new facets orientated within the nonconvex region.

Another possibility is to add a higher order term of curvature type (the so-called Willmore energy). This penalizes spatial oscillations and regularizes the problem. It leads to a regularized energy functional as in \cite{GurtJabb}
$$\Gamma \mapsto \mathcal E^{(\lambda)}(\Gamma):=\int_\Gamma \phi(\nu_\Gamma)\, d\HH^{d-1} +\lambda \int_\Gamma |H_\Gamma|^2\, d\HH^{d-1},$$
where $\lambda>0$ is a  (fixed) weight parameter, and $H_\Gamma=({\rm div^\Gamma}\nu_\Gamma) \nu_{\Gamma}$ is the mean curvature vector of $\Gamma$. Although not  obvious, it is  expected that the presence of the higher order term $H_\Gamma$ improves the compactness properties of minimizing sequences $\{\Gamma_n\}_{n \in \N}$ for $ \mathcal E^{(\lambda)}$, so that $\Gamma_n$ should now converge in a stronger way to some (generalized) set $\Gamma$,  making it possible to pass to the limit in the nonconvex term, without appealing to any type of relaxation. Let us point out that the behavior of (volume-constrained) minimizers for small $\lambda$ has been investigated in \cite{spencer2004asymptotic,BraiMal,moser2012towards,moser2015singular}. See also \cite{GNRWill} for the application of this idea in a different context.

\medskip

The first objective of the present work is to investigate the lower semicontinuity properties of this type of problems when $\Gamma=\partial E$ is the boundary of a set $E$. Set
\begin{equation}\label{eq:F(E)}
\F(E):= \mathcal E^{(1)}(\partial^* E)=\int_{\partial^* E} \phi(\nu_E)\, d\HH^{d-1}+\int_{\partial^* E}|H_E|^2\, d\HH^{d-1}.
\end{equation}
Postponing precise definitions to the next sections, our first main result is the following lower semi-continuity result,
see Theorem \ref{BBG0}. 
\begin{theorem}\label{BBG0}
Let $\{E_n\}_{n \in \N}$ be a sequence of sets of finite perimeter in $\Om$ converging  in $L^1$ to a set $E$ of finite perimeter in $\Om$. If  $V_E$ has bounded first variation in $L^2_{\mu_E}(\Om;\R^d)$, then
\[
 \liminf_{n\to \infty} \F(E_n)\ge \F(E).
 \]
\end{theorem}
\begin{remark}
 Let us point out that this result, just like Theorem \ref{BBG1} and Theorem \ref{BBG2}, comes with a compactness statement which is transparent from the proof. However since we need to assume a priori some regularity of the limit set we do not include it.
\end{remark}

Another important question for such sharp interface variational problems is the approximation by means of phase field models. This has a long history, going back to  \cite{ModMort} where a so-called Modica-Mortola (or Allen-Cahn) approximation is proposed to approximate the classical perimeter functional in the sense of $\Gamma$-convergence. Regarding the  variational approximation of the sum of the isotropic perimeter and Willmore functionals, a phase field approximation result has been established in \cite{RogSchat,Nagase-Tonegawa}  answering a long standing conjecture of De Giorgi (see also \cite{Lussardi} in the case of weakly anistropic surface energies). Following  \cite{Lowen,Lussardi}, we investigate here the case of strongly anisotropic surface energies. Let $\phi_\eps$ be a suitable smooth approximation of $\phi$, see \eqref{defphieps}. We introduce the phase field energy functional $\F_\eps:H^2(\Om) \to \R$ by
\begin{equation}\label{Fepsintro} \F_\eps(v)=\frac{1}{c_0} \int_{\Om} \phi_\eps\lt(\nabla v\rt)\lt( \frac{\eps}{2}|\nabla v|^2 + \frac{1}{\eps} W(v)\rt) dx + \frac{1}{c_0 } \int_{\Omega} \frac1\eps\lt( -\eps \Delta v+\frac{1}{\eps} W'(v)\rt)^2 dx.\end{equation}
Here $W(z)=(1-z^2)^2$ is the standard double-well potential and $c_0=\int_{-1}^1 \sqrt{2W(z)}\, dz$. In the absence of the second term, and if the surface tension $\phi \equiv 1$, the first contribution of the energy corresponds exactly to the classical Modica-Mortola  $\Gamma$-convergence approximation of the perimeter functional. 
 It is established in \cite{Lussardi} that the first term alone gives rise in the limit to the relaxed anisotropic perimeter
$$\int_{\partial^* E} \phi^{**}(\nu_E)\, d\HH^{d-1},$$
where $\varphi^{**}$ is the convex hull of $\varphi$. The second contribution in the energy functional $\F_\eps$ is exactly the same as in \cite{RogSchat}. It corresponds to a phase field approximation of the Willmore energy
$$\int_{\partial^* E}|H_E|^2\, d\HH^{d-1}.$$
 Again in \cite{Lussardi}, it is established that if $\varphi$ is convex, $\mathcal F_\eps$ $\Gamma$-converges to $\F$ provided $\partial E$ is $\mathcal C^2$.

 Let us point out that following \cite{Lowen}, in the first term of $\F_\eps$, the anisotropy is carried out by multiplying the Modica-Mortola energy density by $\phi_\eps(\nabla v)$. It is argued in  \cite{Lowen} that this is crucial in order to have  an interfacial layer independent of the orientation which is then compatible with the optimal profile of the curvature term. This is in contrast with other phase field models such as \cite{ratz2006higher,wise2007solving}. See also  \cite{bretin2015phase} for a thorough review of phase field approximations of curvature energies.

 Our second main result, Theorem \ref{BBG1}, corresponds to a generalization of the result \cite{Lussardi} to general anisotropies $\varphi$.
 \begin{theorem}\label{BBG1}
Let $d=2,$ or $d=3$, let $\Omega$ be a bounded open subset of $\R^d$ and let $E$ be a set of finite perimeter in $\Om$ such that $V_E$ has bounded first variation in $L^2_{\mu_E}(\Om;\R^d)$.

(i) Lower bound: If $\{v_\eps\}_{\eps>0} \subset H^2(\Om)$ is such that $v_\eps \to {\bf 1}_E-{\bf 1}_{\Om \setminus E}$ in $L^1(\Om)$, then
$$\F(E) \leq \liminf_{\eps \to 0} \F_\eps(v_\eps).$$

(ii) Upper bound: If additionally $\partial E$ is a $\mathcal C^2$-hypersurface, there exists a sequence $\{\bar v_\eps\}_{\eps>0}$ in $H^2(\Om)$ such that $\bar v_\eps \to {\bf 1}_E-{\bf 1}_{\Om \setminus E}$ in $L^1(\Om)$ and
$$\F(E)= \lim_{\eps \to 0} \F_\eps(\bar v_\eps).$$
\end{theorem}

For both Theorem \ref{BBG0} and Theorem \ref{BBG1}, we strongly rely on the existing results on lower semi-continuity and phase field approximations \cite{SchatzLower,RogSchat} which are naturally expressed in the language of varifolds. With these results at hand, the proofs of both theorems are quite short (and might be seen as warm ups for the free-discontinuity counterparts, see below). The main observation we are making and which seems to have remained unnoticed so far,  see e.g. \cite{Lussardi}, is that varifold convergence together with rectifiability of the limit imply convergence of the anisotropic perimeter term. This is the  analog of the property that for sequences of functions, strong convergence is equivalent to the fact that the generated Young measures are Dirac masses, see \cite{rindler2018calculus}. The lower bound in Theorem \ref{BBG1} rests on a locality property for the mean curvature of varifolds in order to compare the curvature of $\partial E$ with the curvature of the limit varifold (which has natural semi-continuity properties), see \cite{AmbMas,SchatzLower,leonardi2009locality}. Moreover, as can already be seen for $d=2$ in \cite{BellMugnai}, the relaxation of $\F$ must be non-local and a good understanding of its structure for $d=3$ is still missing. The restriction on the regularity on $\partial E$ in the upper bound of  Theorem \ref{BBG1} comes from the difficulty of approximating varifolds in energy  by smooth sets. This is a common limitation in this setting, see \cite{RogSchat,merlet2015highly}. See however \cite{rupp2024global} for some recent progress on this question.

\bigskip

The second type of sharp interface models with curvature penalization we are interested in, are motivated by applications to fracture mechanics involving brittle materials with strongly anisotropic surface energies, see \cite{BinMau}.

In a two-dimensional antiplane setting ($d=2$), the variational approach to fracture introduced in \cite{FranMar,BFM} is based on the global minimization of a Mumford-Shah type functional
\begin{equation}\label{MSu}u \mapsto MS(u):=\int_{\Om \setminus J_u} |\nabla u|^2\, dx + \int_{J_u} \phi(\nu_{J_u})\, d\HH^{1},\end{equation}
where $u\in SBV^2(\Om)$ is a scalar function corresponding to an antiplane displacement, and the crack is assimilated to the jump set $J_u$ of $u$. In the isotropic case, $\phi \equiv 1$, we recover the classical Mumford-Shah energy. Under external loadings and/or boundary conditions, the direct method in the calculus of variations easily leads to the existence of minimizers for $MS$ in the case of weak anisotropy, i.e. if $\phi$ is a convex function. In addition, a phase field approximation result due to Ambrosio-Tortorelli has been established in \cite{Ambrosio-Tortorelli,Foc}.

If however, $\phi$ is nonconvex, the problem might develop spatial oscillations as explained above. Therefore, the existence of minimizers for $MS$ and the variational approximation by means of an Ambrosio-Tortorelli type functional is not guaranteed anymore because of a lack of lower semicontinuity of $MS$. Once more, a possible remedy is to penalize spatial oscillations by introducing a curvature term inside the free energy. Contrary to the previously described problem, the crack $J_u$ has no reason to be a boundaryless $1$-dimensional set so that one has to account both for possible endpoints (the so-called crack-tips) as well as multiple junctions in $J_u$.  The following curvature dependent Mumford-Shah functionals has been introduced in \cite{Coscia,BraiMar} in the context of image reconstruction (in the  isotropic case),

\begin{equation}\label{eq:Gset}
\widehat{\mathcal G}^{(\gamma)}(u,C,P):=\int_{\Om \setminus C } |\nabla u|^2\, dx +  \int_{C}(\varphi(\nu_{C}) + |H_{C}|^2)\, d\HH^1 + \gamma \HH^0(P).
\end{equation}
Here $\gamma>0$ is a weight parameter, $C$ is a finite union of curves with endpoints $P$ and  $u\in H^1(\Om\setminus C)$. Notice in particular that this implies that if we consider $u\in SBV^2(\Om)$, then $J_u\subset C$. This formulation is closer in spirit to the orginal strong formulation of the Mumford-Shah functional in \cite{mumford1989optimal} than its weak counterpart \eqref{MSu}. Existence of minimizers for $\widehat{\mathcal G}^{(\gamma)}$ (with an extra fidelity term) has been obtained in   \cite{Coscia} and a phase field approximation has been proposed in \cite{BraiMar}. However, as opposed to the phase field approximation in \eqref{Fepsintro}, the curvature was approximated by the more nonlinear term ${\rm div}\big(\frac{\nabla v}{|\nabla v|}\big)$. While this greatly simplifies the analysis, the question of an approximation through a phase field model more similar to \eqref{Fepsintro} was raised in \cite[Section 7]{BraiMar}. One of the main objectives of this paper is to answer positively this question. Since we will build upon \cite{SchatzLower,RogSchat}, it is more natural for us to work in the $SBV$ framework. We thus introduce the sharp interface energy,

\begin{equation*}
\mathcal G^{(\gamma)}(u):=\int_{\Om \setminus J_u } |\nabla u|^2\, dx +  \int_{J_u}(\varphi(\nu_{J_u}) + |H_{J_u}|^2)\, d\HH^1 + \gamma \HH^0(\mathcal P_{J_u}),
\end{equation*}
We will restrict ourselves to functions $u\in SBV^2(\Om)$ with a jump set $J_u$  made of a finite  family of curves. Here  $\mathcal P_{J_u}$ denotes the finite set of endpoints of the curves in $J_u$ (counting both multiple junctions and true endpoints).   We first prove in Theorem \ref{thm:lsc} below a lower semicontinuity result for $\mathcal G^{(\gamma)}$. To some extent it  can be  seen as an alternative proof, but  under a more restrictive set of assumptions, of the results from \cite{Coscia,BraiMar}. While we include it mostly for presentation purposes, as it illustrates in a simpler setting some of the ideas behind the proof of the lower bound for the phase field approximation, let us make the following observation.  Since, in the notation of \eqref{eq:Gset} we assume that $J_u=C$, such a lower semicontinuity property can only hold provided the coefficient $\gamma$ is small enough (as in the statement of Theorem \ref{thm:lsc}). Here smallness depends only on the smoothness of $J_u\backslash \mathcal P_{J_u}$,  the mutual distances between the points in  $\mathcal P_{J_u}$ as well as the distance of $\mathcal P_{J_u}$ from $\partial \Omega$. Indeed, if two points in $\mathcal P_{J_u}$ are very close to each other, it can be energetically favorable to add a small jump set connecting these points (and thus paying the distance between them) rather than having two crack-tips (which come with a fixed cost of $2$). In the language of \eqref{eq:Gset} this amounts to $C\neq J_u$. This is somewhat related to the relaxation phenomenon for the elastica studied in \cite{BellMugnai} and already mentioned above. Besides \cite{SchatzLower} the main ingredient for this lower semicontinuity results is Lemma \ref{lem:pointX} which states that if a varifold with curvature in $L^2$ contains a set with singular curvature, then either its length or its Willmore energy should be bounded from below. Interestingly, the proof relies mostly on the monotonicity formula even though the energy we consider is anisotropic.

We present now the phase field approximation of $\mathcal G^{(\gamma)}$ that we consider. It is an anisotropic variant of the one proposed in \cite[Section 7]{BraiMar}.  The idea is to have a first phase field variable $v$ as in the Ambrosio-Tortorelli model and a second phase field variable $w$ to approximate the point energy $\HH^0(\mathcal P_{J_u})$. We first point out that as observed in \cite{BraiMar}, since the optimal profile of the classical Ambrosio-Tortorelli functional is not $H^2$, it is not suited for the approximation of curvature energies. Following \cite{BraiMar}, we thus consider a Modica-Mortola approximation with a penalization of the size of the set $\{v=-1\}$. In the language of \eqref{eq:F(E)} this
amounts to require that $|E|\to 0$. Another possible solution would be to detect the jump set of $u$ through the derivative of the optimal profile for the Modica-Mortola energy as very recently suggested in \cite{BreChamMa}. Let us now discuss the point energy. For this we repeat the heuristics from \cite{BraiMal,BraiMar}. For every $\beta>0$, let us consider the functional
\[
 \mathcal{F}^\beta(E)=\int_{\partial E} (\beta^{-1} + \beta |H_E|^2) d\HH^1.
\]
By the Cauchy-Schwarz inequality and Gauss-Bonnet formula, we have
\begin{equation}\label{F4pi}
 \mathcal{F}^\beta(E)\ge 2\int_{\partial E} |H_E|d\HH^1\ge 4\pi
\end{equation}
with equality if and only if $E$ is a ball of radius $\beta$. We thus consider for $\beta_\eps\to 0$  a phase field approximation of $\mathcal{F}^{\beta_\eps}$ as in \eqref{Fepsintro} and set
\[
 G_{\eps,\beta}(w)=\frac{1}{c_0 \beta}\int_{\Omega}\lt(\frac{\eps}{2}|\nabla w|^2 + \frac{1}{\eps} W(w)\rt)dx + \frac{\beta}{c_0 \eps}\int_{\Omega} \left(-\eps \Delta w+\frac{1}{\eps} W'(w)\right)^2dx.
\]
For all $(u,v,w) \in H^1(\Om) \times [H^2(\Om)]^2$, we then introduce the following energy functional
\begin{eqnarray}
 \G^{(\gamma)}_\eps(u,v,w) & = & \frac14\int_\Om (1+v)^2 |\nabla u|^2 \, dx+  \frac{1}{2c_0}\int_\Om  \phi_\eps\lt(\nabla v\rt)\lt(\frac{\eps}{2}|\nabla v|^2+ \frac{1}{\eps} W(v)\rt)dx\nonumber\\
&&  + \frac{1}{8c_0}\int_\Om \frac{1}{\eps} \left( -\eps \Delta v+\frac{1}{\eps} W'(v)\right)^2 (1+w)^2\,dx \label{eq:Geps}\\
&& +\frac{\gamma}{4\pi}G_{\eps,\beta_\eps}(w)+\frac{1}{\eta_\eps}\int_\Omega (1-v)^2\, dx + \frac{1}{\eta_\eps}\int_\Omega (1-w)^2\, dx,\nonumber
\end{eqnarray}
where $\eps>0$, $\eta_\eps>0$ and $\beta_\eps>0$ are  infinitesimal parameters which satisfy the following scaling laws as in \cite{BraiMar}
$$\lim_{\eps \to 0} \frac{\eps|\log \eps|}{\beta_\eps}=\lim_{\eps \to 0} \frac{\beta_\eps}{\eta_\eps}= 0.$$
Notice that just like $\nabla u$ is not penalized in the regions where  $v=-1$ (as in the classical Ambrosio-Tortorelli approximation), the curvature of $J_u$, which is approximated by the term on the second line of \eqref{eq:Geps}, is not penalized in the regions where $w=-1$. This is essential in order to approximate jump sets with crack tips or multiple junctions.

\medskip

As in the case of the perimeter, we need to assume some regularity of the admissible jump sets $J_u$.  We assume here that $J_u$ is the finite disjoint union of $\mathcal C^2$ curves possibly intersecting at their endpoints $\mathcal P_{J_u}$. At each of these end points, $J_u$ has an infinite curvature. We further assume that at multiple junctions the  sum of the tangent vectors vanishes. This means that  the generalized curvature (in the sense of varifolds) has a Dirac mass at all points of $\mathcal P_{J_u}$. Unfortunately, the case of multiple junctions with zero curvature is more delicate and remains out of reach at the moment, see Remark \ref{rem:point-zero-curv}. Our main result is then (see Theorem \ref{BBG2} and \eqref{defA} for the definition of $\mathcal A(\Om)$)
\begin{theorem}\label{BBG2}
Let $u \in \mathcal A(\Om)$.

(i) {\bf Lower bound.} There exists $\gamma_0=\gamma_0(u)>0$ such that for all $\gamma \in (0,\gamma_0)$ the following property is satisfied:
for all $\{(u_\eps,v_\eps,w_\eps)\}_{\eps>0} \subset H^1(\Om) \times [H^2(\Om)]^2$ such that $(u_\eps,v_\eps,w_\eps) \to (u,1,1)$ in $[L^1(\Om)]^3$ we have
$$\mathcal G^{(\gamma)}(u) \leq \liminf_{\eps \to 0} \mathcal G^{(\gamma)}_\eps(u_\eps,v_\eps,w_\eps).$$

(ii) {\bf Upper bound.} There exists a sequence $\{(\bar u_\eps,\bar v_\eps,\bar w_\eps)\}_{\eps>0}$ in $H^1(\Om) \times [H^2(\Om)]^2$ such that  $(\bar u_\eps,\bar v_\eps,\bar w_\eps) \to (u,1,1)$ in $[L^1(\Om)]^3$,
$$\mathcal G^{(\gamma)}(u)= \lim_{\eps \to 0} \mathcal G^{(\gamma)}_\eps(\bar u_\eps,\bar v_\eps,\bar w_\eps).$$
\end{theorem}
The restriction on $\gamma$ in Theorem \ref{BBG2} has its root in the use of Lemma \ref{lem:pointX} just as in the derivation of Theorem \ref{thm:lsc}, see the discussion above. The main idea in the proof of Theorem \ref{BBG2} is contained in Proposition \ref{prop:dichotomie}.  It states  that if the phase field variable $w_\eps$ is not uniformly close to $1$ (say in a ball $B$), then the phase field energy involving $w_\eps$ inside $B$
$$G_{\eps,\beta_\eps}(w_\eps,B)=\frac{1}{c_0 \beta_\eps}\int_{B}\lt(\frac{\eps}{2}|\nabla w_\eps|^2 + \frac{1}{\eps} W(w_\eps)\rt)dx + \frac{\beta_\eps}{c_0 \eps}\int_{B} \left(-\eps \Delta w_\eps+\frac{1}{\eps} W'(w_\eps)\right)^2dx$$
must be bounded from below by $4\pi$.  This is sharp as seen from \eqref{F4pi}. The difficulty actually lies in extending the geometric argument leading to \eqref{F4pi} to the phase field energy $ G_{\eps,\beta}$. The argument rests on a blow-up argument around a point $x_0 \in \Om$ such that $w_\eps(x_0)$ is far away from $1$. Using the results of \cite{RogSchat}, one shows that
\begin{equation}\label{lowerGepsintro}\liminf_{\eps \to 0} G_{\eps,\beta_\eps}(w_\eps,B) \geq 2 \int_{\R^2} |H_V|\, d\mu_V,\end{equation}
for some $1$-rectifiable integral varifold $V$ in $\R^2$, with square-integrable first variation. Using the monotonicity formula as  in \cite{pozzetta}, it is not hard to bound the right-hand side from below by $4$ which is unfortunately not sharp as already noticed in \cite{pozzetta}. This is related to derivation of  Li-Yau inequalities for varifolds through monotonicity formula which are sharp for $d=3$ but not for $d=2$, see \cite{kuwert2004removability,pozzetta,muller2023li}. Let us point out that in order to prove that the limit varifold $V$ is non-trivial, we use in {\sf Step 2} of the proof of Proposition \ref{prop:dichotomie} a variant of the argument from \cite{pozzetta} but at the $\eps-$level.  Going back to the optimal lower bound in \eqref{lowerGepsintro} the question is thus to generalize the  Gauss-Bonnet formula for   $1$-rectifiable integral varifolds in $\R^2$ with square-integrable first variation. Although such objects are expected to enjoy good geometric properties (see e.g. \cite{AllardAlmgren} where it is proved that $1$-rectifiable, integral stationary varifolds are a locally finite sum of segments with constant multiplicities inside each of the segments), the generalization of the Gauss-Bonnet formula to varifolds seems to be a serious issue. In the 2014 preprint \cite{menne2023sharp}, a sharp lower bound on the mean curvature for $m$-rectifiable integral varifolds in $\R^n$ with $p$-integrable first variation is obtained under the critical assumption $2 \leq p=m=n-1$ (see \cite[Theorem 24.1]{Menne}). Unfortunately, our case does not fit within this set of assumptions since $p=2$, $n=2$ and $m=1$. For that reason, we dedicate an Appendix in which we derive the Gauss-Bonnet formula (Theorem \ref{thm:Gauss-Bonnet}) needed for our purpose by closely following and adapting the arguments from \cite{menne2023sharp} to our specific situation. The main point is to prove a  Lusin type property of the Gauss map, see Proposition \ref{prop:lusin} and  \cite{Santilli,santilli2021second} for related results. Since our goal is much more modest than in \cite{menne2023sharp} (restricting in particular to $1-$dimensional varifolds), our proofs are often shorter and the notation lighter than those from  \cite{menne2023sharp}. We hope that our Appendix could also help shed light on  \cite{menne2023sharp} which seems to have remain relatively unnoticed, see e.g. \cite{pozzetta}.

\bigskip

The paper is organized as follows. In Section 2, we introduce notation and the mathematical background needed for our subsequent analysis. Section 3 is devoted to the study of strongly anisotropic perimeters with curvature penalization. Theorem \ref{BBG0} provides a lower semicontinuity property of the functional $\mathcal F$, while Theorem \ref{BBG1} shows a $\Gamma$-convergence type approximation result by means of the phase field energy $\mathcal F_\eps$. In Section 4, we carry out the analysis of the strongly anisotropic Mumford-Shah energy with curvature penalization. Theorem \ref{thm:lsc} shows a lower semicontinuity result of $\mathcal G^{(\gamma)}$ for small enough weight parameter $\gamma>0$.  Theorem \ref{BBG2} gives a $\Gamma$-convergence type results of $\mathcal G^{(\gamma)}$ by means of the phase field energy $\mathcal G^{(\gamma)}_\eps$. Eventually, the Appendix is devoted to show the generalized Gauss-Bonnet formula, Theorem \ref{thm:Gauss-Bonnet}, for varifolds.

\section{Preliminaries}

\subsection{Radon measures}

The Lebesgue measure in $\R^d$ is denoted by $\LL^d$, and the $k$-dimensional Hausdorff measure by $\HH^k$. We will write $\omega_k$ for the $\LL^k$-measure of the $k$-dimensional unit ball in $\R^k$.

Let $X$ be a locally compact metric space and $Y$ be an Euclidean space. We denote by $\mathcal M(X;Y)$ (or simply $\mathcal M(X)$ if $Y=\R$) the space of {\it bounded Radon measures} in $X$. By the Riesz representation Theorem, $\mathcal M(X;Y)$ can be identified with the topological dual of $\mathcal C_0(X;Y)$, the space of all continuous functions $f :X \to Y$ such that $\{|f|\geq \eps\}$ is compact for all $\eps>0$. Similarly, the space of {\it Radon measures} $\mathcal M_{\rm loc}(X;Y)$ (or simply $\mathcal M_{\rm loc}(X)$ if $Y=\R$) is identified with the topological dual of the space $\mathcal C_c(X;Y)$ of all continuous functions $f:X \to Y$ compactly supported in $X$.

Let $\Om \subset \R^d$ be an open set, $\mu \in \mathcal M_{\rm loc}(\Om)$ be a non negative Radon measure and $1 \leq k \leq d$. The $k$-dimensional density of $\mu$ at $x \in \Om$ is defined by 
$$\Theta^k(\mu,x):=\lim_{\varrho \to 0}\frac{\mu(B_\varrho(x))}{\omega_k \varrho^k}$$
provided the limit exists.

\subsection{Rectifiable sets}

A set $M \subset \R^d$ is {\it countably $\HH^k$-rectifiable} if there exist countably many Lipschitz functions $f_i :\R^k \to \R^d$, $i \in \N$, such that
$$\HH^k\left(M \setminus\left( \bigcup_{i \in \N} f_i(\R^k)\right)\right)=0.$$
If $\HH^k(M)<\infty$, then $M$ admits an approximate tangent space, denoted by $T_x M$, at $\HH^k$-a.e. $x \in M$, see \cite[Theorem 2.83]{AFP}. We denote by $\Theta^k(M,x):=\Theta^k(\HH^k\restr M,x)$ the $k$-dimensional density of $M$ at $x$ and recall the Besicovitch-Mastrand-Mattila Theorem which states that $\Theta^k(M,\cdot)=1$ $\HH^k$-a.e. in $M$, while $\Theta^k(M,\cdot)=0$ $\HH^k$-a.e. in $\R^d \setminus M$, see  \cite[Theorems 2.56  and 2.63]{AFP}.

It is known that the previous definition of rectifiability is equivalent to require that $\HH^k$-almost every $M$ can be covered by a countable union of $k$-dimensional $\mathcal C^1$-submanifolds. If one further assume the submanifolds to be of class $\mathcal C^2$, then the set $M$ is said to be $\mathcal C^2$-rectifiable.

\subsection{Sets of finite perimeter}

Let $\Om \subset \R^d$ be an open set. A Lebesgue measurable set $E \subset \R^d$ is said to have {\it finite perimeter} in $\Om$ if $D{\bf 1}_E \in \mathcal M(\Om;\R^d)$. In that case, we have
$$D{\bf 1}_E=\nu_E \HH^{d-1}\restr \partial^* E,$$
    where $\partial^* E$ is the reduced boundary of $E$ (a countably $\HH^{d-1}$-rectifiable set), and $\nu_E$ is the approximate inner unit normal to $E$ (see \cite[Theorem 3.59]{AFP}). We denote by $P(E,\Om)=|D{\bf 1}_E|(\Om)=\HH^{d-1}(\partial^* E \cap \Om)$ the perimeter of $E$ in $\Om$.

\subsection{Varifolds}

Let us recall several basic ingredients of the theory of (co-dimension $1$) varifolds (see \cite{Simon_GMT} for a detailed description). We denote by $\mathbf G_{d-1}$ the Grassmannian manifold of all $(d-1)$-dimensional linear subspaces of $\R^d$. The set $\mathbf G_{d-1}$ is identified with the set of all orthogonal projection matrices onto $(d-1)$-dimensional linear subspaces of $\R^d$, i.e., $d \times d$ symmetric matrices $A$ such that $A^2=A$ and ${\rm tr}(A)=d-1$, in other words, matrices of the form $A={\rm Id}-e\otimes e$
for some $e \in \S^{d-1}$.

A {\it $(d-1)$-varifold} in $\Om$ is a Radon measure on $\mathbf G_{d-1}(\Om):=\Om\times \mathbf G_{d-1}$. The class of $(d-1)$-varifolds in $\Om$ is denoted by $\mathbf V_{d-1}(\Om)$. The weight measure of $V\in \mathbf V_{d-1}(\Om)$ is  the nonnegative measure  $\mu_V \in \mathcal M_{\rm loc}(\Om)$ defined by $\mu_V(B)=V(B \times \mathbf G_{d-1})$ for all Borel sets $B \subset \Om$.  We define the first variation of a \((d-1)\)-varifold in $V$ in \(\Om\) as the first order distribution
$$\delta V(\zeta)=\int_{\mathbf{G}_{d-1}(\Om)} D \zeta(x)\cdot S \,d V(x,S) \quad \text{ for } \zeta\in \mathcal C^1_c(\Om;\R^d)\,.$$
The varifold $V$ is said to have bounded first variation in $\Om$, if the distribution $\delta V$ extends to a bounded Radon measure in $\Om$. In that case, the Besicovitch differentiation Theorem shows that $\delta V$ splits into
$$\delta V= -H_V \mu_V + \sigma,$$
where $H_V \in L^1_{\mu_V}(\Om;\R^d)$ is the generalized mean curvature of $V$ and $\sigma \in \mathcal M(\Om;\R^d)$ is singular with respect to $\mu_V$. We say that $V$ has bounded first variation in $L^p_{\mu_V}(\Om;\R^d)$, $p\geq 1$, if $\sigma=0$ and  $H_V \in L^p_{\mu_V}(\Om;\R^d)$.

\medskip

Let $M \subset \Om$ be a countably $\HH^{d-1}$-rectifiable set, and $\theta:M \to \R^+$ be a locally $\HH^{d-1}$-integrable (multiplicity) function. A particular class of varifolds is that of $(d-1)$-rectifiable varifolds $V=\mathbf v(M,\theta)$ which can be written as
$$\int_{ G_{d-1}(\Om)}\Phi(x,S)\, dV(x,S):=\int_{M} \Phi(x,T_x M)\theta(x)\, d\HH^{d-1}(x) \quad \text{ for }\Phi \in \mathcal C_c(\mathbf G_{d-1}(\Om)).$$
In that case, the weight measure is given by $\mu_{V}=\theta \HH^{d-1}\restr M$. We will denote by $\nu_V$ an approximate unit normal to the rectifiable set $M$. Let us point out that since we mostly work in the framework of (unoriented) varifolds, all quantities involved will not depend on the choice of the normal. If further $\theta \in \mathbf N$ $\HH^{d-1}$-a.e. in $M$, then $V$ is called $(d-1)$-rectifiable integral varifold.  We will sometimes denote by $V_M=\mathbf v(M,1)$, $\mu_M=\mu_{V_M}$ and, if $V_M$ has bounded first variation in $\Om$, by $H_M:=H_{V_M}$ the generalized mean curvature of $V_M$.

If $M=\partial^* E$ for some set $E$ of finite perimeter in $\Om$, we write $V_E=V_{\partial^* E}$, $\mu_E=\mu_{\partial^* E}$, and $H_E=H_{\partial^* E}$.

\medskip

Let us recall the following monotonicity formula, see \cite[Theorem 2.2]{pozzetta}, which will  be used throughout the paper.

\begin{lemma}\label{lem:monoton}
Let $\Om \subset \R^2$ be an open set and $V \in \mathbf V_1(\Om)$ be a (non zero) $1$-rectifiable integral varifold with bounded first variation in $L^p_{\mu_V}(\Om;\R^2)$, with $p>1$. Then, for every $x_0 \in \Om$ and $r>0$ with $\overline B_r(x_0) \subset \Om$, $\Theta^1(\mu_V,x_0) $ exists and
\begin{equation}\label{eq:monotonicity-form}
\Theta^1(\mu_V,x_0) \leq  \frac{\mu_V(B_r(x_0))}{2r} +\frac12\int_{B_r(x_0)}|H_V|\,d\mu_V.
\end{equation}
Moreover, if $\Om=\R^2$ and $\mu_V(\R^2)<\infty$, then
\begin{equation}\label{eq:monotonicity-formBis}
\frac{\mu_V(B_r(x_0))}{r} \leq \int_{\R^2} |H_V|\,d\mu_V.
\end{equation}
\end{lemma}

\begin{proof}
The fact that $\Theta^1(\mu_V,x_0) $ exists  for any $x_0$ follows from \cite[Corollary 17.8]{Simon_GMT}. Formulas  \eqref{eq:monotonicity-form} and \eqref{eq:monotonicity-formBis} can  be inferred from the proof of \cite[Theorem 2.2]{pozzetta} (though not stated under this form). For the reader's convenience we include some details. By \cite[(18)]{pozzetta}, for every $x_0 \in \Om$ and $0<r<R$ with $\overline B_R(x_0)\subset \Om$,
\begin{multline*}
 \frac{\mu_V(B_r(x_0))}{r} + \frac{1}{r} \int_{B_r(x_0)} H_V(x) \cdot (x-x_0) \, d \mu_V (x)  \leq \frac{\mu_V(B_R(x_0))}{R} \\ + \int_{B_R(x_0)} H_V(x) \cdot \left( \frac{x-x_0}{R} - \frac{x-x_0}{|x-x_0|} {\bf 1}_{B_R(x_0) \setminus B_r(x_0)}  \right) \, d \mu_V (x).
\end{multline*}
 On the one hand, keeping $R>0$ fixed and letting $r \to 0$, since  $H_V \in L^p_{\mu_V}(\Om;\R^2)$ with $p>1$ yields  \eqref{eq:monotonicity-form}.

On the other hand, if $\Om=\R^2$ and $\mu_V(\R^2)<\infty$, we can now keep $r>0$ fixed and let $R \to \infty$. Since $H_V \in L^1_{\mu_V}(\R^2;\R^2)$, we obtain by dominated convergence
\begin{eqnarray*}
 \frac{\mu_V(B_r(x_0))}{r} & \leq & - \frac{1}{r} \int_{B_r(x_0)} H_V(x) \cdot (x-x_0) \, d \mu_V (x) - \int_{\R^2 \setminus B_r (x_0)} H_V(x) \cdot  \frac{x-x_0}{|x-x_0|} \, d \mu_V (x) \\
 & \leq & \int_{\R^2} |H_V|\,d\mu_V,
\end{eqnarray*}
leading to \eqref{eq:monotonicity-formBis}.
\end{proof}

Let us next state the following version of the Sobolev-Michael-Simon inequality (see \cite[Theorems 7.1 and 7.3]{allardfirst}).

\begin{theorem}\label{thm:SMS}
There exists a universal constant $c_1>0$ such that the following property is satisfied: let $\Omega \subset \R^2$ be an open set and $V \in \mathbf V_1(\R^2)$ be a $1$-rectifiable integral varifold with bounded first variation in $L^1_{\mu_V}(\R^2;\R^2)$. Setting
$$\gamma:=\int_{\Omega}|H_V|\, d\mu_V,$$
then, for every compactly supported Lipschitz function $\zeta :\Omega \to \R$.
$$(1-c_1\gamma) \sup_{{\rm Supp}(V)} |\zeta| \leq c_1 \int_{\mathbf G_1(\Omega)} |S \cdot \nabla \zeta(x)|\, dV(x,S).$$
\end{theorem}

\medskip
We finally recall the main result of  \cite{RogSchat} which will be at the basis of our forthcoming phase-field approximation analysis. Although not stated in this form, a careful reading of \cite[Theorem 4.1 \& Theorem 5.1]{RogSchat} yields the following statement. Let us consider the standard double well potential $W:\R \to \R$ defined by
$$W(z)=(1-z^2)^2 \quad \text{ for } z \in \R,$$ and set
\begin{equation}\label{defco}
 c_0=\int_{-1}^1 \sqrt{2W(z)}\, dz=\frac{4\sqrt 2}{3}.
\end{equation}

\begin{theorem}\label{theo_RogSchat}
Let $d=2, 3$ and $\Om$ be a bounded open subset of $\R^d$. Given $v_\eps \in H^2(\Om)$, we define the measures
$$\mu_\eps=\frac{1}{c_0}\left(\frac{\eps}{2}|\nabla v_\eps|^2 +\frac{1}{\eps} W(v_\eps)\right) \LL^d, \quad \alpha_\eps= \frac{1}{c_0 \eps} \left( -\eps \Delta v_\eps+\frac{1}{\eps} W'(v_\eps)\right)^2 \LL^d,$$
{the function
$$\xi_\eps =\frac{\eps}{2}|\nabla v_\eps|^2 - \frac{1}{\eps} W(v_\eps)$$}
and the varifold $V_\eps$ as
\begin{equation*}
 \int_{\mathbf G_{d-1}(\Om)} \Phi(x,S) \, dV_\eps(x,S)= \int_{\Omega} \Phi\lt(x,{\rm Id}- \nu_\eps \otimes  \nu_\eps \rt) d\mu_\eps, \quad \Phi \in \mathcal C_c(\mathbf G_{d-1}(\Om)).
\end{equation*}
where $\nu_\eps : \Omega \rightarrow \S^{d-1}$ are Borel extensions of $\frac{\nabla v_\eps}{|\nabla v_\eps|}$ on $\{\nabla v_\eps=0\}$.
Assume that 
 $$\sup_{\eps>0}\{\mu_\eps(\Omega)+ \alpha_\eps(\Omega)\}<\infty.$$
 Then, up to a subsequence, there exist $\alpha \in \mathcal M^+(\Om)$  and $V \in \mathbf V_{d-1}(\Om)$ such that  $V_\eps\wto V$ weakly* in $\mathcal M(\mathbf G_{d-1}(\Om))$, $\mu_\eps\wto \mu_V$, $\alpha_\eps\wto \alpha$ weakly* in $\mathcal M(\Om)$ and $\xi_\eps \to 0$ in $L^1_{\rm loc}(\Om)$. Moreover, $V$ is a $(d-1)$-rectifiable integral varifold with bounded first variation in $L^2_{\mu_V}(\Om;\R^d)$ and
 \begin{equation}\label{loweralpha}
|H_V|^2 \mu_V \le \alpha.
 \end{equation}
\end{theorem}

\subsection{Curves}

Let $\Omega$ be an open subset of $\R^2$. We consider the class of $1$-dimensional sets $\Gamma \subset \Omega$ which are made of a finite union of curves of class $\mathcal C^2$ without self-intersections and that intersect only at their endpoints. Specifically, there exist $\mathcal C^2$ mappings $\gamma_1,\ldots,\gamma_N :[0,1] \to \Omega$ such that $\dot \gamma_i \neq 0$ and, setting $\Gamma_i=\gamma_i([0,1])$,
\begin{equation}\label{eq:curves}
\Gamma=\bigcup_{i=1}^N \Gamma_i.
\end{equation}
 We denote by $\mathcal P_\Gamma=\bigcup_{i=1}^N \{\gamma_i(0),\gamma_i(1)\}$ the set of all end points of the curves $\Gamma_i$. Given $p \in \mathcal P_\Gamma$, there is $I \subset \{1,\ldots,N\}$ with $\#(I) \geq 1$ such that $p \in \bigcap_{i \in I} \Gamma_i$. Up to a change of orientation, we can assume that
$$p=\gamma_i(0) \quad \text{ for }i \in I.$$
For $i \in I$, let $\dot \gamma_i(0)$ be the tangent vector to $\Gamma_i$ at $p$. We assume that (see Remark \ref{rem:point-zero-curv})\footnote{If, for some $i \in I$, $\Gamma_i$ is a loop, i.e. $\gamma_i(0)=\gamma_i(1)$, then $\dot \gamma_i(0)$ should be replaced by $\dot \gamma_i(0)-\dot \gamma_i(1)$ in \eqref{eq:curvature}.}
\begin{equation}\label{eq:curvature}
v_p:=\sum_{i \in I}  \dot \gamma_i(0) \neq 0.
\end{equation}
The class of all sets $\Gamma$ satisfying \eqref{eq:curves} and \eqref{eq:curvature} is denoted by $\mathscr C(\Omega)$.

\medskip

Let $V_\Gamma$ be the varifold associated to such a set $\Gamma \in \mathscr C(\Om)$ (which is a rectifiable set). A classical computation shows that the first variation of $V_\Gamma$ is given by
$$\delta V_{\Gamma}(\zeta)= -\sum_{i =1}^N \int_{\Gamma_i} H_{\Gamma_i} \cdot \zeta\, d\HH^1 - \sum_{p \in \mathcal P_\Gamma} \zeta(p) \cdot v_p\quad \text{ for }\zeta \in \mathcal C^1_c(\Omega;\R^2),$$
where $H_{\Gamma_i}(x)=\frac{d}{dt}\left(\frac{\dot \gamma_i(t)}{|\dot \gamma_i(t)|}\right)$ is the classical curvature of $\Gamma_i$ at the point $x=\gamma_i(t) \in \gamma_i(]0,1[)$. In other words, $V_\Gamma$ has bounded first variation which can be represented by the measure
$$\delta V_\Gamma= -\sum_{i =1}^N H_{\Gamma_i}\HH^1\restr \Gamma_i- \sum_{p \in \mathcal P_\Gamma} v_p \delta_p.$$
Condition \eqref{eq:curvature} can thus be interpreted as a nonzero curvature in the sense of varifolds at all intersection points of the curves $\Gamma_i$'s.

\medskip

The following elementary lemma describes the tangent measures to $\HH^1\restr \Gamma$ for   $\Gamma \in \mathscr C(\Om)$. We omit the proof.

\begin{lemma}\label{lem:blow-up-Gamma}
Let $\Gamma \in \mathscr C(\Om)$. For all $x_0 \in \Omega$ and $\varrho>0$, we define the blow-up measure $\lambda_{x_0,\varrho}$ of $\HH^1\restr \Gamma$ by
$$\int_{\R^2} \zeta\, d\lambda_{x_0,\varrho}:=\frac1\varrho \int_{\Gamma}\zeta\left(\frac{y-x_0}{\varrho}\right)d \HH^1(y) \quad \text{ for }\zeta \in \mathcal C_c(\R^2)$$
and 
$$\Theta(\Gamma,x_0):=\lim_{\varrho \to 0} \frac{\HH^1(\Gamma \cap B_\varrho(x_0))}{2\varrho}.$$
Then, $\lambda_{x_0,\varrho} \wto \lambda_0$ weakly* in $\mathcal M_{\rm loc}(\R^2)$ as $\varrho \to 0$, where
\begin{itemize}
\item[(i)] if $x_0 \in \Om \setminus \Gamma$, then $\lambda_0=0$ and $\Theta(\Gamma,x_0)=0$;
\item[(ii)] if $x_0 \in \gamma_i(]0,1[)$ for some (unique) $i \in \{1,\ldots,N\}$, then $\lambda_0 =\HH^1\restr T$ where $T=\R \dot \gamma(t_0)$ is the tangent line to $\gamma_i(]0,1[)$ at $x_0=\gamma_i(t_0)$, and $\Theta(\Gamma,x_0)=1$;
\item[(iii)]  if $x_0 \in \mathcal P_\Gamma$, then $\lambda_0=\sum_{i \in I} \HH^1\restr L_i$ and $L_i=\R^+ \dot \gamma_i(0)$ is the half-line tangent to $\Gamma_i$ at the point $x_0=\gamma_i(0)$ and $\Theta(\Gamma,x_0)=\frac{\#(I)}{2}$.
\end{itemize}
\end{lemma}

The following result is a consequence of Taylor's expansion of the length function.

\begin{lemma}\label{lem:taylor}
For every $\Gamma \in \mathscr C(\Om)$, there exists $\varrho_0=\varrho_0(\Gamma)>0$ such that for every $x_0 \in \mathcal P_\Gamma$ and $0<\varrho \leq \varrho_0$, \begin{equation}\label{eq:C_*}
B_\varrho(x_0) \cap \mathcal P_\Gamma=\{x_0\}, \quad  \HH^1(\Gamma \cap B_\varrho(x_0))\le 2 \varrho \Theta(\Gamma,x_0)+ \varrho^2  \Theta(\Gamma,x_0) \sup_{\Gamma \setminus \mathcal P_\Gamma}|H_{\Gamma}|.
\end{equation}
\end{lemma}

\begin{proof}
Since both sides of \eqref{eq:C_*} are additive in the number of curves meeting at $x_0$, it is enough to consider the case where $\Gamma$ is a single curve. We may further assume   that $x_0=0$ and that $\gamma$ is an arc-length parametrization of $\Gamma$, i.e. setting $L=\HH^1(\Gamma)$, $\gamma \in \mathcal C^2([0,L];\R^2)$ satisfies $|\dot \gamma|=1$ on $[0,L]$ and $\gamma(0)=0$. We thus need to show the existence of $\varrho_0>0$ such that for every $\varrho \in (0,\varrho_0)$,
\begin{equation}\label{eq:red}
\HH^1(\Gamma \cap B_\varrho)\le  \varrho+\varrho^2   \sup_{t \in (0,L)}|\ddot \gamma(t)|.
\end{equation}

Let $\varrho_0>0$ be small enough so that $\Gamma \cap B_{\varrho_0}$ has one connected component. In particular, for every $0 <\varrho < \varrho_0$, denoting by $\ell(\varrho)= \HH^1(\Gamma \cap B_\varrho)$, we have

$$\Gamma \cap B_\varrho = \gamma([0,\ell(\varrho)), \qquad \Gamma \cap \partial B_\varrho =\{\gamma(\ell(\varrho))\}.$$
Moreover, since  $\gamma \in \mathcal C^2([0,L];\R^2)$, then $\ell$ is an increasing function of class $\mathcal C^2$ on $(0,\varrho_0)$. Deriving the equation $|\gamma(\ell(\varrho))|^2=\varrho^2$ with respect to $\varrho$ yields
\begin{equation}\label{eq:dotell}
\dot \gamma(\ell(\varrho)) \cdot \gamma(\ell(\varrho))\dot \ell(\varrho) =\varrho \quad \text{ for }\varrho \in (0,\varrho_0).
\end{equation}
Recalling that $\gamma$ is an arclength parametrization of $\Gamma$ and that $\ell$ is increasing, we get
$$\varrho \leq |\gamma(\ell(\varrho))| |\dot \gamma(\ell(\varrho))| \dot \ell (\varrho)=\varrho \dot \ell(\varrho),$$
hence $\dot \ell(\varrho) \geq 1$. Moreover, passing to the limit in  \eqref{eq:dotell} as $\varrho \to 0$ yields $|\dot \gamma(0)|^2 \dot \ell(0)^2=1$, hence $\dot \ell(0)=1$. As a consequence, using Taylor expansion, we get that for $\varrho \in (0,\varrho_0)$,
\begin{equation}\label{eq:taylor}
\ell(\varrho)=\ell(0)+\varrho \dot \ell(0)+\int_0^\varrho (\varrho -s)\ddot \ell(s)\, ds = \varrho+\int_0^\varrho (\varrho -s)\ddot \ell(s)\, ds.
\end{equation}
In order to estimate $\ddot \ell$, we derive \eqref{eq:dotell} which leads to
$$1=\dot \ell^2\left( 1+\gamma(\ell) \cdot  \ddot \gamma(\ell)\right) +\gamma(\ell) \cdot \dot \gamma(\ell) \ddot \ell \quad \text{ in }(0,\varrho_0).$$
Using \eqref{eq:dotell} and recalling that $\dot \ell \geq 1$, we obtain for every $s \in (0,\varrho_0)$
\begin{eqnarray*}
\frac{s}{\dot \ell(s)}\ddot \ell(s) & = &\gamma(\ell(s)) \cdot \dot \gamma(\ell(s)) \ddot \ell(s) \leq - \dot \ell(s)^2 \gamma(\ell(s)) \cdot  \ddot \gamma(\ell(s))\\
& \leq & \dot \ell(s)^2 |\gamma(\ell(s)) | | \ddot \gamma(\ell(s))| = \dot \ell(s)^2s | \ddot \gamma(\ell(s))|.
\end{eqnarray*}
Since $\dot \ell(0)=1$, we can further reduce $\varrho_0$ in such a way that $\dot \ell(s)^3 \leq 2$ for every $s \in (0,\varrho_0)$, which implies that
$$\ddot \ell (s) \leq 2| \ddot \gamma(\ell(s))| \quad \text{ for }s \in (0,\varrho_0).$$
Inserting inside \eqref{eq:taylor} yields
$$\ell(\varrho) \leq \varrho + \varrho^2 \sup_{s \in (0,\varrho_0)} |\ddot \gamma(\ell(s))|,$$
which completes the proof of \eqref{eq:red}, hence of the lemma.
 \end{proof}

\subsection{Functions of bounded variation}

Given an open set $\Om \subset \R^d$, the space of {\it functions of bounded variation} is defined by
$$BV(\Om)=\{u \in L^1(\Om) : \; Du \in \mathcal M(\Om;\R^d)\}.$$
We shall also consider the subspace $SBV(\Om)$ of special functions of bounded variation made of functions $u \in BV(\Om)$ whose distributional derivative can be decomposed as
$Du=\nabla u \LL^d + (u^+-u^-)\nu_u \HH^{d-1} \restr J_u$. 
In the previous expression, $\nabla u$ is the Radon-Nikod\'ym derivative of $Du$ with respect to $\LL^d$, and it is called the approximate gradient of $u$. The Borel set $J_u$ is the (approximate) jump set of $u$. It is a countably $\HH^{d-1}$-rectifiable subset of $\Om$ oriented by the (approximate) normal direction of jump $\nu_u :J_u \to \mathbf S^{d-1}$, and $u^\pm$ are the one-sided approximate limits of $u$ on $J_u$ according to $\nu_u$. Finally we define
$$SBV^2(\Om)=\{u \in SBV(\Om) : \; \nabla u \in L^2(\Om;\R^d) \text{ and } \HH^{d-1}(J_u)<\infty\}.$$

In the sequel, we will be interested only in the two-dimensional case $d=2$, and in the smaller class of functions
\begin{equation}\label{defA}\mathcal A(\Om):=\{u \in SBV^2(\Om) : \; J_u \in \mathscr C(\Omega)\}.\end{equation}

\section{Strongly anisotropic perimeters with curvature}\label{sec:percur}

Let $d=2$ or $d=3$ and $\Om \subset \R^d$ be bounded open set. Let us consider an even and continuous anisotropy function $\phi: \S^{d-1}\to \R^+$ satisfying
\begin{equation}\label{eq:bound-phi}
\frac{1}{C} \leq \phi(z) \leq C \quad \text{ for }z \in \S^{d-1},
\end{equation}
for some $C\geq 1$. Since $\phi$ is an even function, the function $\tilde \phi : \mathbf G_{d-1} \to \R^+$ given, for every $S={\rm Id} -e \otimes e  \in \mathbf G_{d-1}$ (with $e \in \S^{d-1}$), by
\begin{equation}\label{tildephiS}\tilde \phi(S):=\phi(e)\end{equation}
is well defined and continuous.

\medskip

For every Lebesgue measurable set $E$ with finite perimeter in $\Om$, we consider the following energy functional
\[
 \F(E)=
 \begin{cases}
 \displaystyle \int_{\Om \cap \partial^* E} \left(\phi(\nu_E) + |H_E|^2\right)  d\HH^{d-1} & \text{if $V_E$ has bounded first variation in $L^2_{\mu_E}(\Om;\R^d)$,}\\
 =\infty & \text{otherwise.}
 \end{cases}
 \]

\subsection{Lower semi-continuity}

We start with the  question of lower semicontinuity for $\F$ with respect to the following (natural) convergence.

\begin{definition}
A sequence $\{E_n\}_{n \in \N}$ of Lebesgue measurable subsets of $\Om$ converges to the Lebesgue measurable set $E \subset \Om$ if ${\bf 1}_{E_n} \to {\bf 1}_E$ strongly in $L^1(\Om)$, or equivalently, if $\LL^d(E_n \triangle E) \to 0$. 
\end{definition}

We now address the

\begin{proof}[Proof of Theorem \ref{BBG0}]

There is no loss of generality to suppose that 
$$\liminf_{n\to \infty} \F(E_n)<\infty.$$
At the expense of extracting a subsequence, we can thus assume that the $\liminf$ above is a limit and that $\sup_n\F(E_n) <\infty$. In particular, for each $n \in \N$, denoting by $V_n:=V_{E_n}$ the varifold associated to $E_n$ and by $\mu_n:=\HH^{d-1}\restr\partial^* E_n$ the weight measure of $V_n$, then $\{V_n\}_{n \in \N}$ is a sequence of  $(d-1)$-rectifiable  integral varifolds in $\Om$ with bounded first variation in $\Om$ satisfying $\delta V_n= -H_{E_n} \mu_n$ for some $H_{E_n} \in L^2_{\mu_{n}}(\Om;\R^d)$. In addition, according to \eqref{eq:bound-phi},
$$\sup_{n \in \N} \int_{\Om \cap \partial^* E_n} \left(1+|H_{E_n}|^2\right)\, d\HH^{d-1} <\infty.$$

By Allard's integral compactness theorem, see e.g. \cite[Theorem 42.7 \& Remark 42.8]{Simon_GMT}, we infer that, up to a subsequence,  $V_n\wto V$ weakly* in $\mathcal M(G_{d-1}(\Om))$ for some $(d-1)$-rectifiable integral varifold $V \in \mathbf V_{d-1}(\Om)$. Thus, there exist a countably $\HH^{d-1}$-rectifiable set $M$ and a positive integer valued $\HH^{d-1}$-summable function $\theta$ such that for every $\Phi \in \mathcal C_c(\mathbf G_{d-1}(\Om))$,
$$\int_{\mathbf G_{d-1}(\Om)} \Phi(x,S)\, dV(x,S)=\int_{\Om \cap M} \theta(x) \Phi(x,T_x M)\, d\HH^{d-1}(x),$$
 where $T_x M$ is the approximate tangent space to $M$ at $x$.  Moreover, as $\mu_n \wto \mu_V$ weakly* in $\mathcal M(\Om)$ and $\delta V_n= -H_{E_n} \mu_n \wto \delta V$ weakly* in $\mathcal M(\Om;\R^d)$, using a standard lower semicontinuity result, see e.g. \cite[Example 2.36]{AFP}, we get that $\delta V=-H_V \mu_V$ for some $H_V \in L^2_{\mu_V}(\Om;\R^d)$ and
$$ \int_\Om |H_{\mu_V}|^2 \, d\mu_V\le \liminf_{n\to \infty} \int_\Om |H_{E_n}|^2 \, d\mu_n.$$
Since ${\bf 1}_{E_n} \to {\bf 1}_E$ strongly in $L^1(\Om)$ and $\mu_n=\HH^{d-1}\restr \partial^* E_n = \left| D {\bf 1}_{E_n} \right|$ we have $\sup_n \left| D {\bf 1}_{E_n} \right|(\Omega) < \infty$ and consequently $D {\bf 1}_{E_n} \wto D {\bf 1}_{E}$ weakly* in $\mathcal M(\Om ; \R^d)$. Moreover, $\mu_n = \left| D {\bf 1}_{E_n} \right| \wto \mu_V$ weakly* in $\mathcal M(\Om)$ and therefore $$\mu_E = \HH^{d-1}\restr \partial^* E = \left| D {\bf 1}_{E} \right| \leq \mu_V.$$
Since in addition $V_E$ has bounded first variation in $L^2_{\mu_E}(\Om;\R^d)$, according to \cite[Theorem 1]{Menne} (see also \cite[Theorem 5.1]{schatzle2004quadratic} and \cite[Theorem 3.1]{SchatzLower} in dimension $d=2$), we infer that  $\partial^* E$ is $\mathcal C^2$-rectifiable. Then \cite[Corollary 4.3]{SchatzLower} implies that, $\HH^{d-1}$-a.e. in $\partial^* E$, we have  $H_V=H_{E}$, hence
\begin{equation}\label{eq:ashscouaired}
\int_{\Om \cap \partial^* E} |H_{E}|^2 \, d\HH^{d-1} \le \int_{\Om} |H_V|^2 d\mu_V\le \liminf_{n\to \infty} \int_{\Om \cap \partial^* E_n} |H_{E_n}|^2 \, d\HH^{d-1}.
\end{equation}

We now consider  the anisotropic perimeter term. Notice first that by \cite[Proposition 2.85]{AFP} we have  the locality of the approximate tangent spaces: $T_x M=T_x (\partial^* E)$ for $\HH^{d-1}$-a.e. $x \in \partial^* E$.  Let now $\zeta \in \mathcal C_c(\Om)$ be such that $0 \leq \zeta \leq 1$ in $\Om$. Then, the function $(x,S) \in \mathbf G_{d-1}(\Om) \mapsto \zeta(x) \tilde \phi(S)$ belongs to $\mathcal C_c(\mathbf G_{d-1}(\Om))$ so that by varifold convergence,
\begin{multline}\label{eq:ineqphi}
 \liminf_{n\to \infty } \int_{\Om \cap \partial^* E_n} \phi(\nu_{E_n}) \, d\HH^{d-1}\geq  \liminf_{n\to \infty} \int_{\mathbf G_{d-1}(\Om)} \zeta(x)\tilde \phi (S)\, dV_n (x,S)\\
 =\int_{\mathbf G_{d-1}(\Om)} \zeta(x)\tilde \phi (S)\, dV (x,S)= \int_{\Om} \zeta(x) \tilde \phi(T_xM) \, d\mu_V(x) \\\ge \int_{\Om \cap \partial^* E} \zeta\, \phi(\nu_E)\, d\HH^{d-1}.
\end{multline}
Passing to the supremum with respect to all $\zeta$ as above and gathering \eqref{eq:ashscouaired} together with \eqref {eq:ineqphi} yields
$$\int_{\Om \cap \partial^* E} \left( \phi(\nu_E)+ |H_{E}|^2\right) d\HH^{d-1} \leq \liminf_{n \to \infty}\int_{\Om \cap \partial^* E_n} \left( \phi(\nu_{E_n})+ |H_{E_n}|^2\right)d\HH^{d-1}.$$
This proves the desired semicontinuity property of $\F$.
\end{proof}

\begin{remark}\label{rem:simplelowercompact}
In dimension $d=2$, the following  elementary argument avoiding geometric measure theory can be used (see e.g. \cite{BellMugnai} or  \cite[Lemma 2.1]{BraiMar}).

We first recall that if $\gamma \in H^2([0,1];\R^2)$ is a closed curve and $\Gamma=\gamma([0,1])$, denoting by $\kappa:=|H|$ the mean curvature of $\Gamma$, we have by the Gauss-Bonnet Theorem,  Cauchy-Schwarz and Young's inequalities
\[
 2\pi \leq \int_\Gamma \kappa\, d\HH^1 \le \HH^1(\Gamma)^{\frac{1}{2}} \lt(\int_\Gamma \kappa^2\, d\HH^1\rt)^{\frac{1}{2}}\leq \frac{1}{2} \int_\Gamma (1 +\kappa^2) \, d\HH^1\stackrel{\eqref{eq:bound-phi}}{\leq}  \frac{C}{2} \int_\Gamma (\phi(\nu) +\kappa^2) \, d\HH^1,
\]
hence
\begin{equation}\label{Gausslower2d}
 \int_\Gamma (\phi(\nu) +\kappa^2) \, d\HH^1\geq \frac{4\pi}{C} \qquad \textrm{and } \qquad \int_\Gamma \kappa^2\, d\HH^1 \geq \frac{ 4\pi^2}{\HH^{1}(\Gamma)}.
\end{equation}

In order to have compactness and lower semicontinuity for the natural topology, we need to extend  a bit the definition of $\F$ as follows. For $(E,\Gamma)$ such that $\partial E\subset \Gamma$ with $\Gamma=\bigcup_{i \in I} \Gamma^i$ and $\Gamma^i$ are  closed loops which can only intersect tangentially, we set
\[
 \F(E,\Gamma)= \sum_{i\in I} \int_{\Gamma^i} (\phi(\nu) +\kappa^2) \, d\HH^1.
\]
Let now $(E_n,\Gamma_n)$ be such that $\sup_n\F(E_n,\Gamma_n)<\infty$. By the first inequality in \eqref{Gausslower2d},
\[
\#(I_n) \frac{4\pi}{C} \leq \sum_{i \in I_n}\int_{\Gamma_n^i} (\phi(\nu) +\kappa^2) \, d\HH^1=\F(E_n,\Gamma_n),
\]
so that $\#(I_n)$ is uniformly bounded with respect to $n$. By the second inequality in \eqref{Gausslower2d} we see that  $\HH^1(\Gamma_n^i)$ are bounded from above and below. Up to translations, we thus find that the parametrizations $\gamma_n^i$ are bounded in $H^2([0,1];\R^2)$ and thus converging weakly in that space (and also strongly in $H^1([0,1];\R^2)$) to some curves $\gamma^i$. We then find
\[
 \liminf_{n\to \infty} \F(E_n,\Gamma_n)\ge \sum_{i \in I} \int_{\Gamma^i} (\phi(\nu) +\kappa^2)\,  d\HH^1.
\]
Moreover, the sequence  $\{{\bf 1}_{E_n}\}_{n \in \N}$ is bounded in $BV(\Om)$ and thus, up to a further subsequence, $E_n$ converges to some set $E$. Since $\HH^1\restr \Gamma^n\geq \HH^1\restr \partial E_n$, by lower semicontinuity of the perimeter we find,
\[
 \HH^1\restr \Gamma \geq \HH^1\restr \partial E.
\]
This proves that $\partial E \subset \Gamma$ and concludes the proof of the compactness and lower semicontinuity.
\end{remark}

\subsection{Phase field approximation}

Given $v \in H^2(\Om)$, we consider the following phase-field approximation of $\F$ based on \cite{Lowen,Lussardi}: for $\eps > 0$ and $\sqrt{\eps} r_\eps\to 0$,
\begin{equation*}
 \F_\eps(v)=\frac{1}{c_0} \int_{\Om}\phi_\eps \lt( \nabla v \rt)\lt( \frac{\eps}{2}|\nabla v|^2 + \frac{1}{\eps} W(v)\rt) dx + \frac{1}{c_0 } \int_{\Omega} \frac1\eps\lt( -\eps \Delta v+\frac{1}{\eps} W'(v)\rt)^2 dx \: ,
\end{equation*}
where $\phi_\eps$ extends $\phi$ from the sphere to $\R^d$ as proposed in \cite[Section 2]{Lussardi} to handle regions of $\Omega$ where $\nabla v = 0$ and $\phi \big( \frac{\nabla v}{| \nabla v|} \big)$ cannot be defined.
More precisely, for $z \in \R^n$,
\begin{equation}\label{defphieps}
\phi_\eps (z) = \overline{\phi} \left( \frac{z}{\sqrt{r_\eps^2 + |z|^2}} \right) \quad \text{and} \quad \overline{\phi} = \xi \phi + (1 - \xi) \frac{\min_{\S^{d-1}} \phi}{4}
\end{equation}
where $0 \leq \xi \leq 1$ is a cut-off function satisfying $\xi = 1$ in $\R^d \setminus B_{1/2}$, $\xi = 0$ in $B_{1/4}$ and $| \nabla \xi | \leq 4$. As proved in \cite{Lussardi}, such definitions ensure that $\overline{\phi}$ is Lipschitz continuous (with Lipschitz constant $L>0$) with
\begin{equation}\label{boundphieps}
 \frac{1}{4C} \le \phi_\eps(z)\le C \qquad \text{ for all } z\in \R^d,
\end{equation}
hence
\begin{equation} \label{eq:muepsBounded}
\int_{\Om}\phi_\eps \lt( \nabla v \rt)\lt( \frac{\eps}{2}|\nabla v|^2 + \frac{1}{\eps} W(v)\rt) dx \geq \frac{1}{4C} \int_{\Om}\lt( \frac{\eps}{2}|\nabla v|^2 + \frac{1}{\eps} W(v)\rt) dx.
\end{equation}
Furthermore, 
\begin{equation} \label{eq:LussardiPhirVsPhi}
 \left|\phi_\eps (z) - \phi \left( \frac{z}{|z|} \right) \right| \leq \frac{L r_\eps}{2 | z |} \quad \text{ for all }z \neq 0.
\end{equation}

\begin{remark}\label{improvreps}
 Let us point out that in \cite{Lussardi} the stronger hypothesis $r_\eps=O(\eps)$ is assumed compared to ours $\sqrt{\eps}r_\eps=o(1)$. From the upper bound construction one could actually conjecture that the sharpest condition should be $\eps r_\eps=o(1)$.
\end{remark}

\subsubsection{Proof Theorem \ref{BBG1}: the lower bound inequality}
We first prove the following intermediate statement.
\begin{lemma}\label{lemma:lowerphasefield}
 Assume that $\sup_{\eps} \F_\eps(v_\eps)<\infty$. Then, in the notation of Theorem \ref{theo_RogSchat}, up to a subsequence, there exist $V \in \mathbf V_{d-1}(\Om)$ such that  $V_\eps\wto V$ weakly* in $\mathcal M(\mathbf G_{d-1}(\Om))$ and  $\mu_\eps\wto \mu_V$. Moreover, $V$ is a $(d-1)$-rectifiable integral varifold with bounded first variation in $L^2_{\mu_V}(\Om;\R^d)$ and
 \begin{equation}\label{lowerphasefield}
  \liminf_{\eps\to 0} \F_\eps(v_\eps)\ge \int_{\Omega} (\phi(\nu_V) +|H_V|^2) \, d\mu_V.
 \end{equation}

\end{lemma}
\begin{proof}
 Let $$K = \sup_{\eps >0} \F_\eps(v_\eps)<\infty.$$
We rewrite the energy as (recall the definitions of $\mu_\eps$ and $\alpha_\eps$ from Theorem \ref{theo_RogSchat})
\[
 \F_\eps(v_\eps)= \int_{\Omega}\phi_\eps (\nabla v_\eps ) \,d\mu_\eps +\alpha_\eps(\Om).
\]
According to \eqref{eq:muepsBounded}, we get that the sequence $\{\mu_\eps(\Om)+\alpha_\eps(\Om)\}_{\eps>0}$ is bounded so that Theorem \ref{theo_RogSchat} applies. In particular, up to a subsequence, $V_\eps \wto V$, where $V$ is a $(d-1)$-rectifiable integral varifold, and there exist a countably $\HH^{d-1}$-rectifiable set $M$ and a positive integer valued $\HH^{d-1}$-summable function $\theta$ such that $\mu_V=\theta \HH^{d-1}\restr M$ and
$$\int_{\mathbf G_{d-1}(\Om)} \Phi(x,S)\, dV(x,S)= \int_M \theta(x) \Phi(x,T_x M)\, d\HH^{d-1}(x), \quad \Phi \in \mathcal C_c(\mathbf G_{d-1}(\Om)).$$ 
Moreover,
\begin{equation} \label{eq:DiscrepancyTo0}
 \xi_\eps =\frac{\eps}{2}|\nabla v_\eps|^2 - \frac{1}{\eps} W(v_\eps)\to 0 \quad \text{in } L_{\rm loc}^1(\Om).
\end{equation}
We infer from \eqref{eq:muepsBounded} that for all $\eps > 0$,
\[
K \geq \int_{\Omega }\phi_\eps ( \nabla v_\eps) \lt( \frac{\eps}{2}|\nabla v_\eps|^2 + \frac{1}{\eps} W(v_\eps)\rt) dx \geq \frac{1}{4C \eps}  \int_\Omega W(v_\eps) \: dx \]
hence,  by Hölder inequality, 
\[\int_\Omega \sqrt {W(v_\eps) } \, dx
 \leq 2 \sqrt{ K C |\Omega| \eps }.
\]

Combining \eqref{eq:LussardiPhirVsPhi} and Young inequality we then get
\begin{align*}
\int_{\Omega }\phi_\eps ( \nabla v_\eps) \lt( \frac{\eps}{2}|\nabla v_\eps|^2 + \frac{1}{\eps} W(v_\eps)\rt)dx
& \geq \int_{\Omega}\phi_\eps ( \nabla v_\eps) |\nabla v_\eps| \sqrt{2 W(v_\eps)} \, dx \\
& \geq \int_{\{\nabla v_\eps \neq 0\}} \phi \lt( \frac{\nabla v_\eps}{| \nabla v_\eps |} \rt) |\nabla v_\eps| \sqrt{2 W(v_\eps)} \, dx - \frac{L r_\eps}{2} \int_\Omega \sqrt{2 W(v_\eps)} \, dx,
\end{align*}
and thus
\begin{multline} \label{eq:liminf1}
 \liminf_{\eps\to 0} \int_{\Omega }\phi_\eps (\nabla v_\eps) \lt( \frac{\eps}{2}|\nabla v_\eps|^2 + \frac{1}{\eps} W(v_\eps)\rt)dx \\
 \geq \liminf_{\eps\to 0}  \int_{\{\nabla v_\eps\neq 0\}} \phi \lt( \frac{\nabla v_\eps}{| \nabla v_\eps |} \rt) |\nabla v_\eps| \sqrt{2 W(v_\eps)} \: dx \: .
\end{multline}
Note that $\nu_\eps = \frac{\nabla v_\eps}{| \nabla v_\eps |}$ on $\{|\nabla v_\eps| \neq 0\}$, while $\phi(\nu_\eps) \mu_\eps= \frac{\phi(\nu_\eps)}{\eps c_0} W(v_\eps) \LL^d$ on $\{\nabla v_\eps=0\}$. Recalling \eqref{eq:bound-phi}, we get by \eqref{eq:DiscrepancyTo0} {that for all $\zeta \in \mathcal C_c(\Omega)$ with $0 \leq \zeta \leq 1$,
\begin{multline} \label{eq:liminf2}
\left| \frac{1}{c_0} \int_{\{\nabla v_\eps\neq 0\}} \phi \lt( \frac{\nabla v_\eps}{| \nabla v_\eps |} \rt) |\nabla v_\eps| \sqrt{2 W(v_\eps)}\zeta \: dx - \int_\Omega \phi \lt( \nu_\eps \rt) \zeta\, d \mu_\eps \right|\\
 \leq \frac{C}{c_0} \int_{\Omega} \lt| |\nabla v_\eps| \sqrt{2 W(v_\eps)} -\lt(\frac{\eps}{2}|\nabla v_\eps|^2 + \frac{1}{\eps} W(v_\eps)\rt)\rt|\zeta\, dx \\
= \frac{C}{c_0} \int_{\Omega} \lt( \sqrt{\frac{\eps}{2}} | \nabla v_\eps| - \sqrt{\frac{W(v_\eps)}{\eps}} \rt)^2\zeta\,  dx \leq \frac{C}{c_0} \int_{{{\rm Supp}(\zeta)}} | \xi_\eps | \: dx\to 0 \: ,
\end{multline}}
where we used, in the last inequality, that  for $a$, $b \geq 0$, $|a-b| \leq a + b$ and thus $(a-b)^2 \leq |a-b|(a+b) = |a^2 - b^2|$.
Thanks to \eqref{tildephiS}, \eqref{eq:liminf1}, \eqref{eq:liminf2}, and by convergence of varifolds,
we have 
\begin{multline*}
 \liminf_{\eps\to 0} \frac{1}{c_0}  \int_{\Omega }\phi_\eps (\nabla v_\eps) \lt( \frac{\eps}{2}|\nabla v_\eps|^2 + \frac{1}{\eps} W(v_\eps)\rt)dx \\
 \geq \liminf_{\eps\to 0}  \int_{\Omega} \phi \lt( \nu_\eps \rt){\zeta} \, d \mu_\eps
\geq \lim_{\eps\to 0} \int_{\mathbf G_{d-1}(\Omega)} \zeta(x)\tilde \phi(S) \, dV_\eps(x,S)
 = \int_{\mathbf G_{d-1}(\Omega)} \zeta(x)\tilde \phi(S)\,  dV(x,S) \: .
\end{multline*}
Passing to the supremum with respect to all $\zeta$  yields
\begin{equation*}
\liminf_{\eps\to 0} \frac{1}{c_0} \int_{\Omega }\phi_\eps (\nabla v_\eps) \lt( \frac{\eps}{2}|\nabla v_\eps|^2 + \frac{1}{\eps} W(v_\eps)\rt)dx  \geq \int_{\Om} \tilde \phi(T_x M) \, d \mu_V=\int_{\Om} \phi(\nu_V)\, d\mu_V.
\end{equation*}
The corresponding lower bound on the curvature term is a direct consequence of \eqref{loweralpha}. This concludes the proof of \eqref{lowerphasefield}.
\end{proof}
We now complete the proof of the lower bound in Theorem \ref{BBG1}.
There is no loss of generality in assuming that
$$\liminf_{\eps\to0} \F_\eps(v_\eps)<\infty.$$
Thus, up to a subsequence (not relabeled), we can assume that the $\liminf$ is actually a limit. We may thus apply Lemma \ref{lemma:lowerphasefield}. By the lower bound estimate in the Modica-Mortola Theorem (see \cite{ModMort}), we have
\begin{equation}\label{lowerPFE}
\mu_V = \theta \HH^{d-1}\restr M \ge \HH^{d-1}\restr \partial^* E.
\end{equation}

By the locality of the approximate tangent spaces together with \eqref{lowerPFE} yields
\begin{equation}\label{eq:phiterm1}
\int_{\Om}  \phi(\nu_V) \, d \mu_V\geq  \int_{\Om \cap \partial^* E}  \phi(\nu_E) \, d\HH^{d-1}.
\end{equation}
This shows the lower bound inequality for the anisotropic interfacial term. Since $V_E$ has bounded first variation in $L^2_{\mu_E}(\Om;\R^d)$, according to \cite[Theorem 1]{Menne} (see also \cite[Theorem 5.1]{schatzle2004quadratic} and \cite[Theorem 3.1]{SchatzLower} in dimension $d=2$), we infer that  $\partial^* E$ is $\mathcal C^2$-rectifiable. Thus, \eqref{lowerPFE} and \cite[Corollary 4.3]{SchatzLower} imply that $H_V=H_{E}$ $\HH^{d-1}$-a.e. in $\partial^* E$, hence by \eqref{lowerPFE}, we infer that
\begin{equation}\label{eq:otherterm}
 \int_{\Om \cap \partial^* E} |H_E|^2 \, d\HH^{d-1}=  \int_{\Om \cap \partial^* E} |H_{V}|^2 \, d\HH^{d-1}\leq \int_\Om |H_V|^2 \, d\mu_V.
\end{equation}
Combining \eqref{lowerphasefield}, \eqref{eq:phiterm1} and \eqref{eq:otherterm} leads to
\[
 \liminf_{\eps\to 0} \F_\eps(v_\eps)\ge \int_{\Om \cap \partial^* E} \left(\phi(\nu_E)+ |H_E|^2\right)d\HH^{d-1} = \F(E).
\]
This corresponds to the desired lower bound.
\hfill$\Box$

\subsubsection{Proof of Theorem \ref{BBG1}: the upper bound inequality}\label{sec:ubdMM}

Let $q:t \in \R \mapsto \tanh(t\sqrt{2})$ be the optimal profile, i.e. the unique solution of the ODE
\begin{equation}\label{eq:edo2}
\left\lbrace
\begin{array}{rcl}
 -q''+ W'(q) & = & 0  \text{ in }\R \\
 q(\pm\infty) & = & \pm 1 ,
 \end{array}
 \right.
\end{equation}
which also satisfies
\begin{equation}\label{eq:edo}
q'=\sqrt{2W(q)} \qquad \text{ in }\R.
\end{equation}

 As in \cite[Lemma 4.1]{DoMuRo} (see also \cite{DoMuMa2}) we fix a cut-off function $\zeta\in \mathcal C^\infty_c(\R;[0,1])$ such that
\[
 \zeta(t)=1  \quad \textrm{for } t\in[-1,1], \qquad \zeta(t)=0,  \quad \textrm{for } |t|\ge 2.
\]
Given $\lambda>1$ to be fixed later, we introduce the scaling parameter
\begin{equation}\label{defdeps}\delta_\eps:=\lambda \eps |\log \eps|.\end{equation}
and  set
\begin{equation}\label{defqeps}
 q_\eps(t)=\zeta\left(\frac{2 t}{\delta_\eps}\right) q\left(\frac{t}{\eps}\right) + ({\rm sign \ } t)\left(1-\zeta\left(\frac{2 t}{\delta_\eps}\right)\right), \quad t \in \R.
\end{equation}

We first prove two elementary properties of $q_\eps$ which will be useful in the forthcoming recovery sequence construction.

\begin{lemma}\label{lem:c_0}
For every $\lambda>1$, we have
\begin{equation}\label{eqc0}\lim_{\eps \to 0} \int_{-\delta_\eps}^{\delta_\eps} \frac{\eps}{2}|q_\eps'|^2  dt=\lim_{\eps \to 0} \int_{-\delta_\eps}^{\delta_\eps}\frac{1}{\eps} W(q_\eps)dt=  \frac{c_0}{2}.\end{equation}
Moreover,
\begin{equation}\label{eq:term3.42}
\lim_{\eps\to0}\int_{|t|\ge \delta_\eps/2} \left(\frac{\eps}{2}|q'_\eps|^2 + \frac{W(q_\eps)}{\eps}\right)  dt= 0.
\end{equation}
\end{lemma}

\begin{proof}
Since
\[
 q_\eps(t)= \zeta\Big(\frac{2 t}{\delta_\eps}\Big)\left(q\Big(\frac{t}{\eps}\Big) -1\right) +1 \quad \text{ in }\R^+
\]
we get by  \eqref{eq:edo}
\begin{eqnarray*}
q'_\eps(t) & =& \frac{2}{\delta_\eps} \zeta'\Big(\frac{2 t}{\delta_\eps}\Big) \left(q\Big(\frac{t}{\eps}\Big)-1\right) +\frac{1}{\eps} \zeta\Big(\frac{2 t}{\delta_\eps}\Big) q'\Big(\frac{t}{\eps}\Big)\\
 & = & \left( q\Big(\frac{t}{\eps}\Big) -1\right)\left[\frac{2}{\delta_\eps} \zeta'\Big(\frac{2 t}{\delta_\eps}\Big)-\frac{\sqrt2}{\eps}  \zeta\Big(\frac{2 t}{\delta_\eps}\Big)\left(1+ q\Big(\frac{t}{\eps}\Big) \right)\right].
 \end{eqnarray*}
Using that for $t\in [\delta_\eps/2,\delta_\eps]$,
\begin{equation}\label{eq:exp}
\left|1- q\Big(\frac{t}{\eps}\Big)\right|\leq  2 e^{-\frac{\sqrt{2}}{2} \frac{ \delta_\eps}{\eps}}=2 e^{-\frac{\sqrt{2}}{2}  \lambda|\log \eps|}\leq \eps^{\frac{\sqrt{2}}{2} \lambda},
\end{equation}
we obtain (because $\lambda>1$)
\begin{equation*}
\sup_{\frac{\delta_\eps}{2} \leq t \leq \delta_\eps}  \left(\frac{\eps}{2}|q'_\eps(t)|^2 + \frac{W(q_\eps(t))}{\eps}\right) \leq C \left(\eps \left(\frac{1}{\delta_\eps^2}+\frac{1}{\eps^2}\right) + \frac1\eps\right) \eps^{ \lambda},
\end{equation*}
for some constant $C>0$ independent of $\eps$. A similar estimate holds in $[-\delta_\eps,-\frac{\delta_\eps}{2}]$, and it leads to \eqref{eq:term3.42}.\\

On the other hand, by \eqref{eq:edo}
\begin{eqnarray}
\int_{-\delta_\eps/2}^{\delta_\eps/2}\frac{\eps}{2}|q'_\eps(t)|^2  dt & = & \int_{-\delta_\eps/2}^{\delta_\eps/2}  \frac{W(q_\eps(t))}{\eps} dt\nonumber \\
&=&\frac{1}{2}\int_{-\delta_\eps/2\eps}^{\delta_\eps/2\eps} \sqrt{2W(q(s))}q'(s) ds.\label{eq:term3.3}
\end{eqnarray}
Since $\delta_\eps/2\eps=\lambda |\log \eps|/2\to \infty$, the right-hand side of \eqref{eq:term3.3} tends to $c_0/2$ (recall \eqref{defco}) as $\eps\to 0$. In combination with  \eqref{eq:term3.42} this completes the proof of the lemma.
\end{proof}

\begin{lemma}\label{lem:epsp}
There exists $\eps_0 \in (0,1)$ satisfying: for all $p>1$, there exists $\lambda>1$ such that
$$\int_{-\delta_\eps}^{\delta_\eps}  \left|-\eps q_\eps'' +\frac{1}{\eps} W'(q_\eps)\right| ^2dt \leq \eps^p, \quad \forall \eps \in (0,\eps_0).$$
\end{lemma}

\begin{proof}
For all  $t \in\R^+$, we compute thanks to \eqref{eq:edo2}
\begin{eqnarray*}
 q_\eps''(t) & = & \frac{4}{\delta_\eps^2} \zeta''\Big(\frac{2 t}{\delta_\eps}\Big) \left( q\Big(\frac{t}{\eps}\Big)-1\right) +\frac{4}{\delta_\eps \eps}  \zeta' \Big(\frac{2 t}{\delta_\eps}\Big) q'\Big(\frac{t}{\eps}\Big) +\frac{1}{\eps^2}  \zeta\Big(\frac{2 t}{\delta_\eps}\Big)  q''\Big(\frac{t}{\eps}\Big)\\
  & = & \frac{4}{\delta_\eps^2}  \zeta''\Big(\frac{2 t}{\delta_\eps}\Big) \left( q\Big(\frac{t}{\eps}\Big)-1\right) +\frac{4}{\delta_\eps \eps}  \zeta' \Big(\frac{2 t}{\delta_\eps}\Big) q'\Big(\frac{t}{\eps}\Big) +\frac{1}{\eps^2}  \zeta \Big(\frac{2 t}{\delta_\eps}\Big) W'\left( q\Big(\frac{t}{\eps}\Big)\right).
\end{eqnarray*}
Using \eqref{eq:edo}, we find
\begin{multline*}
 -\eps q_\eps''(t) +\frac{1}{\eps} W'(q_\eps(t))= \frac{4\eps}{\delta_\eps^2}  \zeta'' \Big(\frac{2 t}{\delta_\eps}\Big)\left(1- q \Big(\frac{t}{\eps}\Big)\right) -\frac{4\sqrt{2}}{\delta_\eps}  \zeta' \Big(\frac{2 t}{\delta_\eps}\Big)\left(1- q\Big(\frac{t}{\eps}\Big)^2\right) \\
 -\frac{1}{\eps} \lt[ \zeta\Big(\frac{2 t}{\delta_\eps}\Big) W'\left( q\Big(\frac{t}{\eps}\Big)\right)-W'\left( \zeta\Big(\frac{2 t}{\delta_\eps}\Big)\left( q\Big(\frac{t}{\eps}\Big)-1\right) +1\right)\rt].
\end{multline*}
Using the explicit expression of $W'$, we now rewrite the term inside the brackets on the right-hand-side of the previous equality as
\begin{multline*}
 \zeta\Big(\frac{2 t}{\delta_\eps}\Big) W'\left( q\Big(\frac{t}{\eps}\Big)\right)-W'\left( \zeta\Big(\frac{2 t}{\delta_\eps}\Big)\left(q\Big(\frac{t}{\eps}\Big)-1\right) +1\right) \\
=4  \zeta\Big(\frac{2 t}{\delta_\eps}\Big)\left(q\Big(\frac{t}{\eps}\Big)-1\right) \left\{  q\Big(\frac{t}{\eps}\Big)\left(q\Big(\frac{t}{\eps}\Big)+1\right)\right.\\
\left.- \left[ \zeta\Big(\frac{2 t}{\delta_\eps}\Big)\left(q\Big(\frac{t}{\eps}\Big)-1\right)+1\right]\left[  \zeta\Big(\frac{2 t}{\delta_\eps}\Big)\left(q\Big(\frac{t}{\eps}\Big)-1\right)+2\right]\right\}
\end{multline*}
so that
\begin{multline*}
 -\eps q_\eps''(t) +\frac{1}{\eps} W'(q_\eps(t)) =\left(1-q\Big(\frac{t}{\eps}\Big)\right)\lt\{\frac{4\eps}{\delta_\eps^2}  \zeta''\Big(\frac{2 t}{\delta_\eps}\Big) -\frac{4\sqrt{2}}{\delta_\eps}  \zeta'\Big(\frac{2 t}{\delta_\eps}\Big) \left(1+q\Big(\frac{t}{\eps}\Big)\right) \right.\\
 \left.+\frac{4}{\eps} \zeta\Big(\frac{2 t}{\delta_\eps}\Big)\lt[  q\Big(\frac{t}{\eps}\Big)\left(q\Big(\frac{t}{\eps}\Big)+1\right)-\left(  \zeta\Big(\frac{2 t}{\delta_\eps}\Big)\left(q\Big(\frac{t}{\eps}\Big)-1\right)+1\right)\left(  \zeta\Big(\frac{2 t}{\delta_\eps}\Big)\left(q\Big(\frac{t}{\eps}\Big)-1\right)+2\right)\rt] \rt\}.
\end{multline*}
A similar computation holds in $\R^-$. In any case, we see that
\begin{equation}\label{eq:zeroEDP}
 -\eps q_\eps''(t) +\frac{1}{\eps} W'(q_\eps(t))=0 \quad \text{if }|t|\le \delta_\eps/2\text{ or }|t|\ge \delta_\eps,
 \end{equation}
while, recalling \eqref{eq:exp},
$$\sup_{\frac{\delta_\eps}{2} \leq |t| \leq \delta_\eps} \left|-\eps q_\eps''(t) +\frac{1}{\eps} W'(q_\eps(t))\right| \leq C  \eps^{\lambda \frac{\sqrt{2}}{2} } \left(\frac{1}{\lambda^2\eps |\log \eps|^2} +\frac{1}{\lambda\eps|\log \eps|} +\frac{1}{\eps}\right),$$
for some constant $C>0$ independent of $\eps$ and $\lambda$. Hence, for all $p>1$, there exists $\lambda>1$ independent of $\eps$ such that
\begin{equation}\label{smallEDP}
\sup_{\frac{\delta_\eps}{2} \leq |t| \leq \delta_\eps} \left|-\eps q_\eps''(t) +\frac{1}{\eps} W'(q_\eps(t))\right| \leq   \eps^{p/2}
\end{equation}
and the conclusion follows from \eqref{eq:zeroEDP} and \eqref{smallEDP}, after integration over the interval $[-\delta_\eps,\delta_\eps]$.
\end{proof}
Let  $\sdist:={\rm sdist}(\cdot,\partial E)$ be the signed distance to $\partial E$  with the convention that $\sdist\leq 0$ in $E$ and $\sdist\geq 0$ in $\R^d \setminus E$. Since $\partial E$ is of class $\mathcal C^2$, there exists $\delta>0$ such that $\sdist$  is of class $\mathcal C^2$ in $\overline{ U_\delta}$, where $U_\delta = \{x \in \R^d : {\rm dist}(x,\partial E)<\delta\}$. We finally set
\[
\bar v_\eps(x)=q_\eps(\sdist(x)) \quad \text{and} \quad \bar\mu_\eps =\frac{1}{c_0}\left(\frac{\eps}{2}|\nabla \bar v_\eps|^2 + \frac{1}{\eps} W (\bar v_\eps)\right) \LL^d
\]
so that $v_\eps$ is constant equal to $\pm 1$ outside $U_{\delta_\eps}$ and ${\rm Supp} (\bar \mu_\eps) \subset \overline{U_{\delta_\eps}}$.

In particular  we have
\[
 \|\bar v_\eps -({\bf 1}_E-{\bf 1}_{\Om \setminus E})\|_{L^1(\Om)}\le C \LL^d(U_{\delta_\eps})\le C \lambda \eps |\log \eps|
\]
so that $\bar v_\eps \to {\bf 1}_E-{\bf 1}_{\Om \setminus E}$ in $L^1(\Om)$.\\
We now prove the convergence of the energy. Since the  convergence of the curvature type term
$$\frac{1}{ c_0}  \int_{\Omega} \frac1\eps\lt( -\eps \Delta \bar v_\eps+\frac{1}{\eps} W'(\bar v_\eps)\rt)^2dx \to \int_{\Om \cap \partial E}|H_E|^2\, d\HH^{d-1}$$
follows as in e.g. \cite{BellMugnai2,DoMuRo}, we do not include the details. See also {\sf Step 6} of the proof of the upper bound construction for Theorem \ref{BBG2} where similar computations are done. We thus turn to  the proof of
\begin{equation}\label{claimupperBBG1}
 \limsup_{\eps \to 0} \int_{\Omega} \phi_\eps(\nabla \bar{v}_\eps) \, d\bar{\mu}_\eps \le \int_{\partial E} \phi(\nu_E) \, d\HH^{d-1}.
\end{equation}
 For $t\in \R$, let $\Gamma_t= \{\sdist=t\}$ so that $\nabla \sdist=\nu_{\Gamma_t}$ on $\Gamma_t$ for $|t|$ small enough. Since $\nabla \bar v_\eps = q_\eps^\prime(\sdist) \nabla \sdist$ and $|\nabla \sdist | = 1$, we use the co-area formula to infer
 \[
  c_0\int_{\Omega} \phi_\eps(\nabla \bar{v}_\eps) \, d\bar{\mu}_\eps=\int_{-\delta_\eps}^{\delta_\eps} \lt(\frac{\eps}{2}|q'_\eps(t)|^2+\frac{1}{\eps} W(q_\eps(t))\rt)\lt(\int_{\Gamma_t} \phi_\eps(q_\eps'(t)\nu_{\Gamma_t}) \, d\HH^{d-1}\rt) dt.
 \]
Since $\phi_\eps\le C$ and for $|t|\le \delta_\eps$
\begin{equation}\label{H1Gammat}
 \HH^{d-1}(\Gamma_t)\le (1+C|t|) \HH^{d-1}(\partial E)\le C \HH^{d-1}(\partial E)
\end{equation}
we have by \eqref{eq:term3.42}
\begin{multline*}
 \limsup_{\eps\to 0} \int_{|t|\ge \delta_\eps/2} \lt(\frac{\eps}{2}|q'_\eps(t)|^2+\frac{1}{\eps} W(q_\eps(t))\rt)\lt(\int_{\Gamma_t} \phi_\eps(q_\eps'(t)\nu_{\Gamma_t})\,  d\HH^{d-1}\rt) dt\\
 \le C \HH^{d-1}(\partial E)\limsup_{\eps\to 0} \int_{|t| \ge \delta_\eps/2}\left( \frac{\eps}{2}|q'_\eps|^2+\frac{1}{\eps} W(q_\eps)\right) dt=0.
\end{multline*}
Using \eqref{eq:LussardiPhirVsPhi} and \eqref{H1Gammat}, we thus find
\begin{multline*}
 c_0\limsup_{\eps \to 0} \int_{\Omega} \phi_\eps(\nabla \bar{v}_\eps) \, d\bar{\mu}_\eps=\limsup_{\eps \to 0} \int_{-\frac{\delta_\eps}{2}}^{\frac{\delta_\eps}{2}} \lt(\frac{\eps}{2}|q'_\eps(t)|^2+\frac{1}{\eps} W(q_\eps(t))\rt)\lt(\int_{\Gamma_t} \phi_\eps(q_\eps'(t)\nu_{\Gamma_t})\, d\HH^{d-1}\rt) dt\\
  \le \limsup_{\eps \to 0}\int_{-\frac{\delta_\eps}{2}}^{\frac{\delta_\eps}{2}} \lt(\frac{\eps}{2}|q'_\eps(t)|^2+\frac{1}{\eps} W(q_\eps(t))\rt)\lt(\int_{\Gamma_t} \phi(\nu_{\Gamma_t}) \, d\HH^{d-1}\rt) dt\\
  +CL\HH^{d-1}(\partial E)\limsup_{\eps\to 0} r_\eps \int_{-\frac{\delta_\eps}{2}}^{\frac{\delta_\eps}{2}} \lt(\frac{\eps}{2}|q'_\eps|^2+\frac{1}{\eps} W(q_\eps)\rt)\frac{1}{|q_\eps'|} dt.
 \end{multline*}
For the first right-hand side term we have by \eqref{eqc0}
\begin{multline*}
 \limsup_{\eps \to 0}\int_{-\frac{\delta_\eps}{2}}^{\frac{\delta_\eps}{2}} \lt(\frac{\eps}{2}|q'_\eps(t)|^2+\frac{1}{\eps} W(q_\eps(t))\rt)\lt(\int_{\Gamma_t} \phi(\nu_{\Gamma_t}) \, d\HH^{d-1}\rt) dt\\
 \le \limsup_{\eps \to 0} \lt(\sup_{|t|\le \delta_\eps} \int_{\Gamma_t} \phi(\nu_{\Gamma_t}) \, d\HH^{d-1}\rt)\int_{-\frac{\delta_\eps}{2}}^{\frac{\delta_\eps}{2}} \lt(\frac{\eps}{2}|q'_\eps|^2+\frac{1}{\eps} W(q_\eps)\rt) dt 
 =c_0\int_{\partial E} \phi(\nu_E)\, d\HH^{d-1}.
\end{multline*}
For the second right-hand side term, since $|q'_\eps(t)|=\sqrt{2 W(q_\eps(t))}/\eps$  for $|t|\le \delta_\eps/2$,  we find
\[
 r_\eps \int_{-\frac{\delta_\eps}{2}}^{\frac{\delta_\eps}{2}} \lt(\frac{\eps}{2}|q'_\eps|^2+\frac{1}{\eps} W(q_\eps)\rt)\frac{1}{|q_\eps'|} dt=r_\eps \eps \int_{-\frac{\delta_\eps}{2}}^{\frac{\delta_\eps}{2}}  |q'_\eps| \, dt\le 2 \eps r_\eps\to 0.
\]
This concludes the proof of \eqref{claimupperBBG1}.
\hfill$\Box$

\section{Strongly anisotropic Mumford-Shah functional with curvature}

We extend the analysis of the previous section to the approximation of the Mumford-Shah  functional with curvature penalization. In this section we set the dimension $d=2$ and consider $\Om \subset \R^2$ a bounded open set. For $\gamma>0$ and $u \in \mathcal A(\Om)$  we define the anisotropic Mumford-Shah functional with curvature penalization by
\begin{equation}\label{eq:MS}
\mathcal G^{(\gamma)}(u)=\int_{\Om} |\nabla u|^2\, dx + \int_{J_u} (\phi(\nu_u)+ |H_{J_u}|^2)\, d\HH^1+\gamma \HH^0(\mathcal P_{J_u}).
\end{equation}
Here  $\phi: \S^{d-1}\to \R^+$ is a continuous anisotropy function satisfying \eqref{eq:bound-phi} as in Section \ref{sec:percur}. By definition \eqref{eq:MS} of the energy, if $\mathcal G^{(\gamma)}(u)<\infty$ then $u \in SBV^2(\Om)$ and $J_u  \in \mathscr C(\Om)$  is a finite union of $\mathcal C^2$-curves whose intersection points $\mathcal P_{J_u}$ have a nonzero curvature in the sense of varifolds.

\begin{remark}\label{rem:point-zero-curv}
If the jump set $J_u$ is made of three line segments intersecting at a single point, say $0$, with an angle $2\pi/3$ (the so-called triple junction configuration), and $u$ is constant inside each connected components of the complementary of $J_u$ (see Figure \ref{fig:tj}), it is known that the varifold $\mathbf v(J_u,1)$ associated to $J_u$ is stationary, i.e., it has zero curvature in the sense of varifolds. However, it is not clear whether the point $0$ should be counted in the point energy \eqref{eq:MS} as a junction point, or if it should not as a point of zero generalized curvature.

\begin{figure}[htbp]
\begin{tikzpicture}

\draw (0, 0) circle (2cm);

\draw[very thick] (0, 0) -- (0,2) node[above] {$J_u$};
\draw[very thick] (0, 0) -- (-1.732,-1);
\draw[very thick] (0, 0) -- (1.732, -1);

\draw (-0.5,-1) node[right]{$u=0$};
\draw (-1.5,0.7) node[right]{$u=-1$};
\draw (0.5,0.7) node[right]{$u=1$};

\end{tikzpicture}
\caption{\small The triple junction configuration}
\label{fig:tj}

\end{figure}
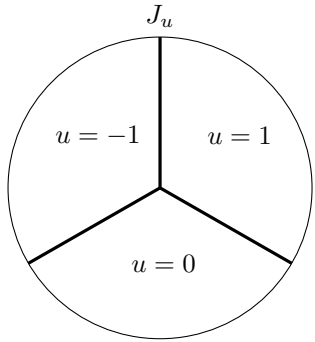

In the context of our phase-field approximation, see Subsection \ref{sec:PFAMS} (with $\Omega=B_1$ and $\varphi\equiv 1$ to simplify), the construction of the upper bound provided in Subsection \ref{sec:upperbound} leads to the existence of  a recovery sequence $\{(\bar u_\eps,\bar v_\eps,\bar w_\eps)\}_{\eps>0}$ such that
$$\mathcal G^{(\gamma)}_\eps(\bar u_\eps,\bar v_\eps,\bar w_\eps) \to 3+\gamma.$$
The question of deciding whether the point $0$ should be counted or not in the point energy is related to the possibility of  improving the upper bound by constructing  a sequence $\{(\hat u_\eps,\hat v_\eps,\hat w_\eps)\}_{\eps>0}$ such that
$$\mathcal G^{(\gamma)}_\eps(\hat u_\eps,\hat v_\eps,\hat w_\eps) \to 3.$$
See \cite[Section 2.3]{bretin2015phase} for a related discussion.
\end{remark}

\medskip

Thanks to the generalized curvature assumption \eqref{eq:curvature}  we can prove the  following density improvement for $1$-rectifiable integral varifolds with square integrable first variation, whose support contains a curve in the class $\mathscr C(\Om)$. Roughly speaking it states that, after a blow-up, if a stationary varifold $V_0$ supported on a  cone $C_0$ with vertex $0$ has a weight measure bounded from below by $\HH^1\restr {\Gamma_0}$, where $\Gamma_0= \bigcup_i L_i \in \mathscr C(\Om)$ is made of finitely many lines  starting at $0$ and with nonzero generalized curvature, then $C_0 \supset \Gamma_0 \cup \ell$ where $\ell$ is an additional half line starting at $0$ or $C_0 = \Gamma_0$ and $V_0$ has multiplicity larger than $2$ on some half line $L_i$. This is an essential ingredient to obtain the lower bound for the energy of points.

\begin{figure}[!h]
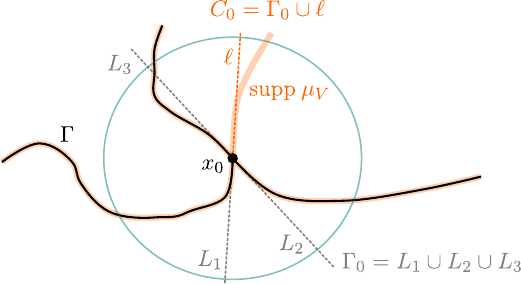
\caption{Blow up at $x_0 \in \mathcal{P}_\Gamma$.}
\end{figure}

\begin{lemma}\label{lemdenssing}
 Let $\Gamma\in \mathscr C(\Omega)$ and $V$ be a $1$-rectifiable integral varifold in $\Omega\subset \R^2$ with first variation in $L^2_{\mu_V}(\Om;\R^2)$ and such that $\mu_V\ge  \HH^1\restr \Gamma$. Then, for every $x_0\in \mathcal P_{\Gamma}$,
 \begin{equation}\label{eq:densing}
  \Theta^1(\mu_V,x_0)\ge \Theta^1(\Gamma,x_0)+\frac{1}{2}.
 \end{equation}

\end{lemma}
\begin{proof}
Without loss of generality we may assume that $x_0=0$. For simplicity we write $\mu:=\mu_V$. We recall that by \cite[Corollary 17.8]{Simon_GMT}, since $V$ is a $1$-rectifiable integral varifold with bounded first variation in $L^2_{\mu}(\Om;\R^2)$, the density function exists and is finite at every $x \in \Omega$ and in particular
$$\Theta(\mu):=\Theta^1(\mu,0)=\lim_{\varrho \to 0} \frac{\mu(B_\varrho)}{2\varrho} < \infty \: .$$
We perform a blow-up at the point $0 \in \mathcal P_\Gamma$. Let $V_{\varrho}$ be the blow-up of $V$ as defined in \cite[Definition 38.1]{Simon_GMT}, i.e.
$$\langle V_{\varrho},\Phi\rangle:=\frac{1}{\varrho}\int_{\mathbf G_1(\R^2)} \Phi\left(\frac{y}{\varrho},S\right)\, dV(y,S) \quad\text{ for }\Phi \in \mathcal C_c(\mathbf G_1(\R^2)).$$
We set $\mu_\varrho:=\mu_{V_\varrho}$. Then $\{V_{\varrho}\}_{\varrho>0}$ is a sequence of $1$-rectifiable integral varifolds with bounded first variation. Since for every $R>0$ and $\varrho>0$;
\[
 \mu_\varrho(B_R)=\frac{\mu(B_{R\varrho})}{\varrho} = 2 R \frac{\mu(B_{R\varrho})}{2 R\varrho},
\]
by definition of $\Theta(\mu)$ we find that  for every $\varrho$ small enough,
\begin{equation}\label{upperbondmuvarho}
 \mu_\varrho(B_R)\le R( 2\Theta(\mu)+1).
\end{equation}
Let us compute the first variation of $V_{\varrho}$. For all $R>0$ and $\zeta \in \mathcal C^1_c(B_R;\R^2)$,
$$\delta V_{\varrho}(\zeta) = \delta V (\zeta_\varrho) = -\int_{B_R} H_V \cdot \zeta_{\varrho}\, d\mu,$$
where $\zeta_{\varrho}(\cdot)=\zeta\left(\cdot/\varrho\right)\in \mathcal C^1_c(B_{R\varrho};\R^2)$. As a consequence of the Cauchy-Schwarz inequality and \eqref{upperbondmuvarho}, we find for $0 < \varrho \leq 1$ small enough,
\begin{eqnarray*}
|\delta V_{\varrho}(\zeta)| & \leq & \|H_V\|_{L^2_\mu(B_R;\R^2)} \|\zeta_{\varrho}\|_{L^2_\mu(B_R;\R^2)}\nonumber\\
& \leq & \|H_V\|_{L^2_\mu(B_R;\R^2)} \|\zeta\|_{L^\infty(B_R;\R^2)} \,\mu(B_{R\varrho})^{1/2} \\
& \leq & C \|H_V\|_{L^2_\mu(B_R;\R^2)} \|\zeta\|_{L^\infty(B_R;\R^2)} \varrho^{1/2},
\end{eqnarray*}
for some constant $C=C(R,\Theta(\mu))>0$.
This yields
\begin{equation}\label{supdelV}
\|\delta V_{\varrho}\|(B_R)  \leq  C \|H_V\|_{L^2_\mu(B_R;\R^2)}\,  \varrho^{1/2}.\end{equation}
Thus, by \eqref{upperbondmuvarho}, \eqref{supdelV} and Allard's compactness Theorem, see \cite[Theorem 42.7 \& Corollary 42.8]{Simon_GMT}, there exists a decreasing subsequence $\{\varrho_j\}_{j \in \N} \searrow 0$ and a $1$-rectifiable integral varifold $V_0$ with locally bounded first variation in $\R^2$ such that $V_{\varrho_j} \wto V_0$ weakly* in $\mathcal M_{\rm loc}(\mathbf G_1(\R^2))$. By \eqref{supdelV}, $\delta V_0=0$, hence $V_0$ is a stationary varifold in $\R^2$. According to \cite[Theorem 19.3]{Simon_GMT}, we deduce that $V_0=\mathbf v(C_0,\psi)$ is a cone (centered at the origin) in the sense that $C_0$ is a countably $1$-rectifiable set invariant under all homotheties $x \mapsto t^{-1}  x$, $t>0$, and $\psi$ is a positive integer valued locally $\HH^1$-integrable function on $C_0$ with $\psi(x)=\psi(t^{-1} x)$, for $x \in C_0$ and $t>0$. Since we are in $\R^2$, this means that $C_0$ is the union of countably many half lines $\{\ell_j\}_{j \in J}$ such that $\psi=\psi_j\in \N \setminus \{0\}$ is constant on $\ell_j$. Notice that since $2\Theta(\mu)=2\Theta^1(C_0,0)=\sum_{j\in J} \psi_j$, then $\#(J) \leq 2\Theta(\mu)<\infty$ and $C_0$ is actually made of finitely many half lines.

Recalling  Lemma \ref{lem:blow-up-Gamma}-(iii) (and notation therein) together with the inequality $\mu \geq \HH^1\restr \Gamma$, we deduce that
$$\sum_{j\in J} \psi_j \HH^1\restr \ell_j=\mu_{V_0} \geq  \HH^1\restr \left(\bigcup_{i \in I} L_i\right).$$
In particular $\bigcup_{i \in I} L_i \subset \bigcup_{j \in J} \ell_j$, hence $I \subset J$ and $L_i=\ell_i$ for $i \in I$. We may write 
\begin{equation}\label{Thetamularge}
 \Theta(\mu)=\Theta^1(\Gamma,0)+\frac{1}{2}\sum_{i\in I} (\psi_i-1) +\frac{1}{2}\sum_{j\in J \setminus I} \psi_j.
\end{equation}
By \eqref{eq:curvature}, the curvature of $\bigcup_{i \in I} L_i$ is the sense of varifolds is not equal to zero and thus we cannot have $\mu=\HH^1\restr \Gamma$. As a consequence, at least one of the two sums at the right-hand side of \eqref{Thetamularge} is not equal to zero. This concludes the proof of \eqref{eq:densing}.
\end{proof}

The following result will be instrumental to derive a lower bound inequality for the point energy of the phase-field approximation of the Mumford-Shah functional with curvature penalization (see Theorem \ref{BBG2}).

\begin{lemma}\label{lem:pointX}
Let $\Om \subset \R^2$ be a bounded open set and $\Gamma \in \mathscr C(\Om)$. There exists $\gamma_0=\gamma_0(\Gamma)>0$ such that the following property holds: if $X \subset \Om$ is a finite set, and $V \in \mathbf V_1(\Om \setminus X$) is a $1$-rectifiable integral varifold with bounded first variation in $L^2_{\mu_V}(\Om \setminus X;\R^2)$ satisfying
\begin{equation}\label{eq:hyp2}
\mu_V \geq \HH^1\restr \Gamma,
\end{equation}
then, for every $\gamma \in (0,\gamma_0)$,
$$\int_{\Omega\backslash X} (\phi(\nu_V)+|H_V|^2) \, d\mu_V\ge \int_{\Gamma} (\phi(\nu_\Gamma)+|H_{\Gamma}|^2)\,  d\HH^1 +\gamma\mathcal H^0(\mathcal P_{\Gamma}\backslash X).$$
\end{lemma}

\begin{proof}

We first notice that for $x_0 \in \Gamma \setminus (X \cup \mathcal P_\Gamma)$ and $\varrho>0$ such that $B_\varrho(x_0)\cap (X \cup \mathcal P_\Gamma)=\emptyset$,   the  associated varifold $V_\Gamma$ is integral with bounded first variation given by $\delta V_\Gamma=-H_\Gamma \HH^1 \restr (\Gamma \cap B_\varrho(x_0))$. Thus, using \eqref{eq:hyp2}, \cite[Corollary 4.3]{SchatzLower} ensures that $H_V=H_\Gamma$ $\HH^1$-a.e. on $\Gamma$, hence
\begin{equation}\label{eq:m-surf22}
\int_{B_\varrho(x_0)}(\phi(\nu_V)+ |H_V|^2) \,d\mu_V\ge  \int_{\Gamma \cap B_\varrho(x_0)}(\phi(\nu_\Gamma)+ |H_\Gamma|^2) \,d\HH^1.
\end{equation}

\medskip

Define 
\begin{equation}\label{eq:barrho}
\bar \varrho=4\gamma (\inf \phi)^{-1},
\end{equation}
and let $\gamma_0=\gamma_0(\Gamma)>0$ be small enough so that for every $\gamma \in (0,\gamma_0)$,
$$\bar\varrho \leq \varrho_0, \quad \overline B_{\bar \varrho}(x)\subset \Omega, \quad B_{\bar \varrho}(x)\cap (\mathcal{ P}_{\Gamma}\backslash \{x\})=\emptyset \qquad  \text{ for } x\in \mathcal{ P}_{\Gamma},$$
where $\varrho_0=\varrho_0(\Gamma)>0$ is given by Lemma \ref{lem:taylor}. We claim that for every $x_0\in \mathcal P_{\Gamma}\backslash X$,
\begin{equation}\label{claimlowerfinal22}
\int_{B_{\bar \varrho}(x_0)} (\phi(\nu_V)+|H_V|^2)\, d\mu_V\ge \int_{\Gamma \cap B_{\bar \varrho}(x_0)} (\phi(\nu_\Gamma)+|H_{\Gamma}|^2) \,d\HH^1 +\gamma.
 \end{equation}
Assuming the validity of \eqref{claimlowerfinal22}, a covering argument together with \eqref{eq:m-surf22} leads to the conclusion. 

\medskip

The rest of the proof is devoted to the proof of  \eqref{claimlowerfinal22}. Without loss of generality we may assume that $x_0=0$ and that
\begin{equation}\label{assumeclaimfinal2}
\int_{B_{\bar \varrho}} (\phi(\nu_V)+|H_V|^2) \, d\mu_V < \int_{\Gamma\cap B_{\bar \varrho}} (\phi(\nu_\Gamma)+|H_{\Gamma}|^2) \, d\HH^1 +\gamma
\end{equation}
since otherwise there is nothing to prove. By \eqref{eq:hyp2} in $B_{\bar \varrho}\backslash \{0\}$ and the locality properties of the first variation and the approximate normal, we have
\[
\int_{B_{\bar \varrho}} |H_V|^2 \, d\mu_V\ge \int_{\Gamma \cap B_{\bar \varrho}} |H_{\Gamma}|^2 \, d\HH^1 \quad \textrm{and}\quad \int_{B_{\bar \varrho}} \phi(\nu_V) \, d\mu_V\ge \int_{\Gamma\cap B_{\bar \varrho}} \phi(\nu_\Gamma) \, d\HH^1
\]
so that \eqref{assumeclaimfinal2} implies
\begin{equation}\label{conseqassumfinal3}
\int_{B_{\bar \varrho}} \phi(\nu_V) \, d\mu_V\le \int_{\Gamma \cap B_{\bar \varrho}} \phi(\nu_\Gamma)\, d\HH^1+\gamma
\end{equation}
and
\begin{equation}\label{conseqassumfinal2}
\int_{B_{\bar \varrho}} |H_V|^2 \, d\mu_V\le \int_{\Gamma \cap B_{\bar \varrho}} |H_{\Gamma}|^2\,  d\HH^1 +\gamma.
\end{equation}
Since $V$ is a $1$-rectifiable integral varifold in $\Om \setminus X$, we have $\mu_V=\theta\HH^1\restr M$ for some $\HH^1$-rectifiable set $M \subset \Om$ and some $\HH^1$-integrable function $\theta:M \to \N$. Moreover, by \eqref{eq:hyp2}, $\Gamma\subset M$ and  $\Theta^1(\mu,\cdot)=\theta \ge 1$ $\HH^1$-a.e. on $\Gamma$. Since $\nu_{\Gamma}=\nu_V$ $\HH^1$-a.e. on $\Gamma$,
\begin{multline*}
\int_{B_{\bar \varrho}} \phi(\nu_V) \, d\mu_V-\int_{\Gamma \cap B_{\bar \varrho}} \phi(\nu_\Gamma) \, d\HH^1\\
 =\int_{\Gamma\cap B_{\bar \varrho}} (\theta-1) \phi(\nu_\Gamma)\, d\HH^1 +\int_{(M\backslash \Gamma)\cap B_{\bar \varrho}} \theta \phi(\nu_V) \, d\HH^1\\
  \ge (\inf \phi )\lt(\int_{\Gamma\cap B_{\bar \varrho}} (\theta-1)  d\HH^1 +\int_{(M\backslash \Gamma)\cap B_{\bar \varrho}}\theta \, d\HH^1\rt)\\
  =(\inf \phi)\lt(\mu_V(B_{\bar \varrho})-\HH^1(\Gamma\cap B_{\bar \varrho})\rt).
\end{multline*}
 Thus, \eqref{conseqassumfinal3} implies
\begin{equation}\label{conseqassumfinalpostpros22}
\mu_V(B_{\bar \varrho})-\HH^1(\Gamma\cap B_{\bar \varrho})\le (\inf \phi)^{-1} \gamma.
\end{equation}
Recalling \eqref{eq:hyp2}, we can apply Lemma \ref{lemdenssing} and find 
\[
 \Theta^1(\mu_V,0)\ge \Theta^1(\Gamma,0)+\frac12.
\]
Thanks to the monotonicity formula \eqref{eq:monotonicity-form}, this leads to
\[
\Theta^1(\Gamma,0)+\frac12\le  \Theta^1(\mu_V,0)\le \frac{\mu_V(B_{\bar \varrho})}{2\bar \varrho} +\frac{1}{2}\int_{B_{\bar \varrho}}|H_V|\, d\mu_V,
\]
which we rewrite as
\[
 1\le \frac{\mu_V(B_{\bar \varrho})-\HH^1(\Gamma\cap B_{\bar \varrho})}{\bar \varrho}+ \frac{\HH^1(\Gamma\cap B_{\bar \varrho})-2\bar \varrho \Theta^1(\Gamma,0)}{\bar \varrho}+\int_{B_{\bar \varrho}}|H_V|\, d\mu_V.
\]

Combining \eqref{eq:C_*} and \eqref{conseqassumfinalpostpros22}, this yields
\[
 1\le \frac{\gamma}{\bar \varrho \inf \phi} + \bar \varrho \,\Theta^1(\Gamma,0)\sup_{\Gamma \setminus \mathcal P_{\Gamma}}|H_{\Gamma}| +\int_{B_{\bar \varrho}}|H_V|\, d\mu_V.
\]
Recalling the definition \eqref{eq:barrho} of $\bar \varrho$, we can restrict $\gamma_0>0$ so that for every $\gamma\in (0,\gamma_0)$,
\begin{equation}\label{eq:1102}
 \bar \varrho\, \Theta^1(\Gamma,0)\sup_{\Gamma \setminus \mathcal P_{\Gamma}}|H_{\Gamma}|  \le \frac{1}{4}.
\end{equation}
After simplification and using the Cauchy-Schwarz inequality, we obtain 
\[
\frac12\le \left(\int_{B_{\bar \varrho}} |H_V|^2 \, d\mu_V\right)^{\frac{1}{2}} \mu_V(B_{\bar \varrho})^{\frac{1}{2}}.
\]
By \eqref{conseqassumfinalpostpros22}, \eqref{eq:C_*}, \eqref{eq:1102} and Lemma \ref{lem:blow-up-Gamma}-(iii), we have
\[
 \mu_V(B_{\bar \varrho})\le C_1 \Theta^1(\Gamma,0) \bar \varrho,
\]
for some universal constant $C_1>0$, so that
\[
 \int_{B_{\bar \varrho}} |H_V|^2 \,d\mu_V\ge \frac{1}{4C_1 \Theta^1(\Gamma,0) \bar \varrho}.
\]
On the one hand, recalling  \eqref{conseqassumfinal2}, we find
\[
 \gamma +\int_{\Gamma \cap B_{\bar \varrho}} |H_{\Gamma}|^2 \,d\HH^1\ge \int_{B_{\bar \varrho}} |H_V|^2\,  d\mu_V\ge \frac{1}{4C_1 \Theta^1(\Gamma,0) \bar \varrho}.
\]
On the other hand, using that $\HH^1 \restr \Gamma \leq \mu_V$ and thus
\[
 \int_{\Gamma\cap B_{\bar \varrho}} |H_{\Gamma}|^2 \, d\HH^1\le C_1 \bar \varrho\, \Theta^1(\Gamma,0) \sup_{\Gamma \setminus \mathcal P_{\Gamma}} |H_{\Gamma}|^2,
\]
we get
\begin{equation}\label{eq:1550}
 \gamma+C_1  \bar \varrho\, \Theta^1(\Gamma,0) \sup_{\Gamma \setminus \mathcal P_{\Gamma}} |H_{\Gamma}|^2\ge \frac{1}{4C_1 \bar \varrho\,\Theta^1(\Gamma,0)}.
\end{equation}
Recalling once more our choice $\bar \varrho= 4\gamma(\inf \phi)^{-1}$, we can further reduce $\gamma_0$ in such a way that \eqref{eq:1550} is violated for all $\gamma<\gamma_0$. This concludes the proof of the lemma.
\end{proof}

\subsection{Lower semicontinuity}

We  now prove  a lower semicontinuity property of the Mumford-Shah functional $\mathcal G^{(\gamma)}$ in the presence of curvature penalization. Since a more general result has been already obtained (with a simpler proof, see Remark \ref{rem:simplelowercompact}) in \cite[Theorem 3.2]{BraiMar}, the motivation for including Theorem  \ref{thm:lsc} rather lies in its proof. Indeed, it can be seen as a streamlined  version of the phase-field counterpart in Theorem \ref{BBG2}.

\begin{theorem}\label{thm:lsc}
 For all $u \in \mathcal A(\Om)$, there exists $\gamma_0=\gamma_0(u)>0$ such that for any $\gamma \in (0,\gamma_0)$, the following property holds: for all sequence $\{u_n\}_{n \in \N}$ in $\mathcal A(\Om)$ such that $u_n \to u$ in $L^1(\Om)$,
 $$\mathcal G^{(\gamma)}(u) \leq \liminf_{n \to \infty} \mathcal G^{(\gamma)}(u_n).$$
 \end{theorem}
 
 \begin{proof}
Assume without loss of generality that $\liminf_n \mathcal G^{(\gamma)}(u_n)<\infty$, and extract a subsequence (not relabeled) such that the liminf is a limit. In particular, according to \eqref{eq:bound-phi},
$$\sup_{n \in \N} \left\{\int_\Om |\nabla u_n|^2\, dx + \HH^1(J_{u_n})\right\} <\infty.$$
By Ambrosio's compactness and lower semi-continuity Theorem, \cite[Theorems 4.7 \& 4.8]{AFP}, $\nabla u_n \wto \nabla u$ weakly in $L^2(\Om;\R^2)$,
\begin{equation}\label{eq:bulk1}
\liminf_{n\to \infty} \int_\Om |\nabla u_n|^2\, dx \geq \int_\Om |\nabla u|^2\, dx
\end{equation}
and 
\begin{equation}\label{eq:Amb}
\liminf_{n\to \infty} \HH^1(J_{u_n}) \geq \HH^1(J_u).
\end{equation}

Let $V_n=\mathbf v(J_{u_n},1)\in \mathbf V_1(\Om)$ be the varifold associated to $J_{u_n}$. Up to a further subsequence, we have $V_n \wto V$ weakly* in $\mathcal M(\mathbf G_1(\Om))$ for some $V \in \mathbf V_1(\Om)$. Since $\mu_{V_n}=\HH^1\restr J_{u_n}$, we get from \eqref{eq:Amb} that
\begin{equation}\label{muVJ}
\mu_V \geq \HH^1\restr J_u.
\end{equation}
Since $\sup_n\HH^0(\mathcal P_{J_{u_n}}) <\infty$, up to a further subsequence, we can suppose that $\HH^0(\mathcal P_{J_{u_n}})=m$ is independent of $n$, and thus
$$\mathcal P_{J_{u_n}}=\{z_1^n,\ldots,z_m^n\}.$$
Up to a new subsequence, still not relabeled, there exist $z_1,\ldots,z_m \in \overline \Om$ (not necessarily pairwise distinct) such that $z_i^n \to z_i$ for all $1 \leq i \leq m$. Let $X=\{z_1,\ldots,z_m\} \cap \Om$ and $\varrho>0$ small enough so that
$$B_\varrho(z) \subset \Om, \quad B_\varrho(z) \cap X =\{z\} \quad \text{ for }z \in X.$$
 Set
$$A_\varrho:=\Om \setminus \bigcup_{z \in X} \overline B_\varrho(z).$$
Note that $\{V_n\}_{n \in \N}$ is a sequence of $1$-rectifiable integral varifolds in $A_\varrho$ with bounded first variation in $L^2_{\mu_{V_n}}(A_\varrho;\R^2)$ satisfying, by \eqref{eq:bound-phi},
$$\mu_{V_n}(A_\varrho) + \int_{A_\varrho} |H_{V_n}|^2\, d\mu_{V_n} \leq \HH^1(J_{u_n}) + \int_{J_{u_n}} |H_{J_{u_n}}|^2\, d\HH^1 \leq \mathcal G^{(\gamma)}(u_n).$$
 We obtain from Allard's compactness Theorem (see \cite[Theorem 42.7 \& Corollary 42.8]{Simon_GMT}) that $V$ is a $1$-rectifiable integral varifold in $A_\varrho$ with bounded first variation in $L^2_{\mu_V}(A_\varrho;\R^2)$. Since $\varrho$ is arbitrary, we deduce that $V$ is a $1$-rectifiable integral varifold in $\Om \setminus X$ with bounded first variation in $L^2_{\mu_V}(\Om \setminus X;\R^2)$. Thus, by varifold convergence, we have for all $\zeta \in \mathcal C_c(A_\varrho)$ with $0 \leq \zeta \leq 1$,
\begin{eqnarray*}
\liminf_{n \to \infty} \int_\Om \varphi(\nu_{u_n})\, d\HH^1& \geq  & \liminf_{n \to \infty} \int_{A_\varrho}\tilde \varphi(S)\zeta(x)\, dV_n(x,S)\\ 
& = &  \int_{A_\varrho}\tilde \varphi(S)\zeta(x)\, dV(x,S)\\
& = &  \int_{A_\varrho} \varphi(\nu_V)\zeta(x)\, d\mu_V(x).
\end{eqnarray*}
Passing to the supremum with respect to all $\zeta \in \mathcal C_c(A_\varrho;[0,1])$ yields
\begin{equation}\label{eq:surf1}
\liminf_{n \to \infty} \int_\Om \varphi(\nu_{u_n})\, d\HH^1 \geq  \int_{A_\varrho} \varphi(\nu_V)\, d\mu_V.
\end{equation}
Concerning the curvature term, since $\mu_{V_n} \wto \mu_V$ weakly* in $\mathcal M(A_\varrho)$ and $\delta V_n= -H_{J_{u_n}} \mu_n \wto \delta V=-H_V \mu_V$ weakly* in $\mathcal M(A_\varrho;\R^2)$, using a standard lower semicontinuity result (see e.g. \cite[Example 2.36]{AFP}), we get that
\begin{equation}\label{eq:curv1}
\int_{A_\varrho} |H_{\mu_V}|^2 \, d\mu_V\le \liminf_{n\to \infty} \int_{A_\varrho} |H_{V_n}|^2 \, d\mu_{V_n}.
\end{equation}
Recalling that $\HH^0(\mathcal P_{J_{u_n}})=m\geq \HH^0(X)$, we deduce from \eqref{eq:bulk1}, \eqref{eq:surf1} and \eqref{eq:curv1} that
$$\liminf_{n\to \infty} \mathcal G^{(\gamma)}(u_n) \geq \int_\Om |\nabla u|^2\, dx + \int_{A_\varrho} \left(\varphi(\nu_V) + |H_V|^2\right) d\mu_V + \gamma \HH^0(X).$$
Letting $\varrho \to 0$,
$$\liminf_{n\to \infty} \mathcal G^{(\gamma)}(u_n) \geq \int_\Om |\nabla u|^2\, dx + \int_{\Om \setminus X} \left(\varphi(\nu_V) + |H_V|^2\right) d\mu_V + \gamma \HH^0(X).$$
The conclusion then follows from \eqref{muVJ} and  Lemma \ref{lem:pointX}.
\end{proof}
 
\subsection{Phase field approximation}\label{sec:PFAMS}

Following \cite{BraiMar}, we consider infinitesimal parameters $\eta_\eps\to 0$ and $\beta_\eps \to 0$ such that 
\begin{equation}\label{eq:scaling}
 \frac{\eps|\log \eps|}{\beta_\eps}\to 0, \qquad \frac{\beta_\eps}{\eta_\eps}\to 0.
\end{equation}

Let $\eps>0$ and $\beta>0$. For $w \in H^2(\Om)$ and Borel sets $A \subset \Om$, we define (recall the definition \eqref{defco} of $c_0$)
\[
 G_{\eps,\beta}(w,A)=\frac{1}{c_0 \beta}\int_{A}\lt(\frac{\eps}{2}|\nabla w|^2 + \frac{1}{\eps} W(w)\rt)dx + \frac{\beta}{c_0 \eps}\int_{A} \left(-\eps \Delta w+\frac{1}{\eps} W'(w)\right)^2dx.
\]
For all Borel set $A \subset \Om$ and all $(u,v,w) \in H^1(\Om) \times [H^2(\Om)]^2$, we define (recall the definition \eqref{defphieps} of $\phi_\eps$)
\begin{multline*}
 \G^{(\gamma)}_\eps(u,v,w,A)=\frac14\int_A (1+v)^2 |\nabla u|^2 \, dx \\
 +\frac{1}{2c_0}\int_A  \phi_\eps\lt(\nabla v\rt)\lt(\frac{\eps}{2}|\nabla v|^2 + \frac{1}{\eps} W(v)\rt)dx+\frac{1}{8c_0}\int_A \frac{1}{\eps} \left( -\eps \Delta v+\frac{1}{\eps} W'(v)\right)^2 (1+w)^2\,dx \\
 + \frac{\gamma}{4\pi}G_{\eps,\beta_\eps}(w,A)+\frac{1}{\eta_\eps}\int_A (1-v)^2\, dx + \frac{1}{\eta_\eps}\int_A (1-w)^2\, dx.
\end{multline*}
and
$$ \G^{(\gamma)}_\eps(u,v,w)= \G^{(\gamma)}_\eps(u,v,w,\Om).$$

The rest of this subsection is devoted to show the $\Gamma$-convergence type result, Theorem \ref{BBG2}.

\subsubsection{Proof of Theorem \ref{BBG2}: the lower bound}

Let $\{(u_\eps,v_\eps,w_\eps)\}_{\eps>0} \subset H^1(\Om) \times [H^2(\Om)]^2$ be a sequence such that  $(u_\eps,v_\eps,w_\eps) \to (u,1,1)$ in $[L^1(\Om)]^3$ where $u \in \mathcal A(\Om)$. We can further assume that
$$\liminf_{\eps\to 0}\mathcal G^{(\gamma)}_\eps(u_\eps,v_\eps,w_\eps)<\infty$$
otherwise, there is nothing to prove. Up to extraction of a (not relabeled) subsequence, we can suppose that the liminf above is a limit and that
\begin{equation}\label{eq:nrj-bd}
\sup_{\eps>0} \mathcal G^{(\gamma)}_\eps(u_\eps,v_\eps,w_\eps) \leq \Lambda,
\end{equation}
for some constant $\Lambda>0$. As in Theorem \ref{theo_RogSchat}, we define the measures
$$\mu_\eps=\frac{1}{c_0}\left(\frac{\eps}{2}|\nabla v_\eps|^2 +\frac{1}{\eps} W(v_\eps)\right) \LL^2, $$
and for $\nu_\eps$ a Borel extension of $\nabla v_\eps/|\nabla v_\eps|$ on $\{\nabla v_\eps=0\}$, the varifolds $V_\eps$ as
$$ \int_{\mathbf G_{1}(\Om)} \Phi(x,S) \, dV_\eps(x,S)= \int_{\Omega} \Phi\lt(x,{\rm Id}- \nu_\eps\otimes  \nu_\eps\rt) d\mu_\eps \qquad \textrm{for } \Phi \in \mathcal C_c(\mathbf G_{1}(\Om)).$$
 We have that $\mu_\eps=\mu_{V_\eps}$ and, by the energy bound \eqref{eq:nrj-bd} together with \eqref{eq:bound-phi},
$$  \mu_\eps(\Om)\leq 2C \Lambda.$$
For later use, we also define 
\begin{equation}\label{eq:def_feps}
f_\eps:=-\eps \Delta v_\eps +\frac1\eps W'(v_\eps) \qquad \textrm{and } \qquad  \alpha_\eps:=\frac{1}{c_0 \eps} f_\eps^2 \LL^2.
\end{equation}
By weak* compactness of Radon measures, up to a further (not relabeled) subsequence, we deduce that $\mathcal G_\eps(u_\eps,v_\eps,w_\eps,\cdot) \wto m$ and  $\mu_\eps \wto \mu$ weakly* in $\mathcal M(\Om)$  some nonnegative Radon measures $m$, $\mu \in \mathcal M^+(\Om)$, and $V_\eps \wto V$ weakly* in $\mathcal M(\mathbf G_1(\Om))$ for some $V \in \mathbf V_1(\Om)$. Moreover, $\mu_V=\mu$ and by lower semicontinuity, we have
\begin{equation}\label{eq:mOmega}
\liminf_{\eps\to 0}\mathcal G^{(\gamma)}_\eps(u_\eps,v_\eps,w_\eps)\geq m(\Om).
\end{equation}
It is thus enough to show that
$$ m(\Om) \geq  \int_{\Om} |\nabla u|^2\, dx + \int_{J_u} (\phi(\nu_u)+ |H_{J_u}|^2)\, d\HH^1+\gamma\HH^0(\mathcal P_{J_u}).$$

\medskip

We first prove the following alternative which states that, either $w_\eps$ is uniformly close to $1$ in a ball $B$, or the energy  $G_{\eps,\beta_\eps}(w_\eps,B)$ is bounded away from zero by a specific positive constant related to the Gauss-Bonnet formula (see Theorem \ref{thm:Gauss-Bonnet} in the Appendix). It will be instrumental in the forthcoming analysis since it will lead to the point energy at all junction and end points of the jump set. Notice that a crucial point is to show that after blow-up the limit varifold is not trivial. This is obtained by combining some of the estimates from \cite{RogSchat} in order to argue as in \cite{pozzetta} but for the diffuse interface model.

\begin{proposition}\label{prop:dichotomie}
Let $x_0 \in \Om$ and $r>0$ be such that $B_r(x_0) \subset \Om$ and  let $\{w_\eps\}_{\eps>0} \subset H^2(\Om)$ be a sequence satisfying $w_\eps\to 1$ in $L^1(\Om)$. Then either
\begin{equation}\label{quantmass}
\liminf_{\eps\to 0} G_{\eps,\beta_\eps}(w_\eps,B_r(x_0))\ge 4\pi
\end{equation}
or 
\begin{equation*}
 \|w_\eps-1\|_{L^\infty(B_{r/2}(x_0))}\to 0.
\end{equation*}
\end{proposition}

\begin{proof}
We assume without loss of generality that $x_0=0$, $r=1$ and we set $B=B_1$. If $\liminf_\eps G_{\eps,\beta_\eps}(w_\eps,B)=\infty$, there is nothing to prove. If $\liminf_\eps G_{\eps,\beta_\eps}(w_\eps,B)<\infty$, we can assume, up to a subsequence (not relabeled), that the liminf is a limit and that the sequence $\{G_{\eps,\beta_\eps}(w_\eps,B)\}_{\eps>0}$ is bounded. We use the notation
$$\bar f_\eps=-\eps \Delta w_\eps+\frac{1}{\eps}W'(w_\eps), \quad  \bar \mu_\eps=\frac{1}{c_0}\lt(\frac{\eps}{2}|\nabla w_\eps|^2 + \frac{1}{\eps} W(w_\eps)\rt)\LL^2\restr B \quad \textrm{and } \quad \bar \alpha_\eps=\frac{1}{c_0 \eps} \bar f_\eps^2\LL^2\restr B$$
so that for some $\bar \Lambda>0$,
$$G_{\eps,\beta_\eps}(w_\eps,B)=\frac{1}{ \beta_\eps}\bar \mu_\eps(B)+ \beta_\eps \bar\alpha_\eps(B)\le \bar \Lambda.$$
In particular we have
\begin{equation}\label{estimfbar}
 \int_B \bar f_\eps^2 dx=c_0\eps \bar\alpha_\eps(B)\le c_0 \frac{\eps}{\beta_\eps} G_{\eps,\beta_\eps}(w_\eps,B) \qquad \textrm{and } \qquad \bar{\mu}_\eps(B)\le \beta_\eps G_{\eps,\beta_\eps}(w_\eps,B).
\end{equation}

Let us suppose that
$$\limsup_{\eps \to 0} \|w_\eps -1\|_{L^\infty(B_{1/2})} >0.$$
Then, there exist $\tau\in (0,1)$, a (not relabeled) subsequence  and points $x_\eps \in B_{1/2}$ such that
$$|w_{\eps}(x_\eps)-1|\ge \tau \quad \text{ for } \eps>0.$$
Let $\delta_{\eps}=\eps/\beta_\eps\to0$ (by \eqref{eq:scaling}) and define, for all $y \in B_{1/2\beta_\eps}$,
$$\hat{w}_{{\eps}}(y)=w_\eps(x_\eps+\beta_\eps y), \quad \hat f_{\eps}(y)=\beta_\eps \bar f_\eps(x_\eps+\beta_\eps y).$$
We also set
\[\hat{\mu}_{\eps}:=\frac{1}{c_0}\lt(\frac{\delta_{\eps}}{2} |\nabla \hat{w}_{{\eps}}|^2 + \frac{1}{\delta_{\eps}}W(\hat{w}_{{\eps}})\rt) \LL^2\restr B_{1/2\beta_\eps},\]
\[ \hat{\alpha}_{{\eps}}:= \frac{1}{\delta_{\eps} c_0} \lt( -\delta_\eps \Delta \hat{w}_{\eps} +\frac{1}{\delta_\eps} W'(\hat{w}_{\eps})\rt)^2 \LL^2\restr B_{1/2\beta_\eps},\]
and
{\[\hat{\xi}_{\hat{\eps}}:=\frac{1}{c_0}\lt(\frac{\delta_{\eps}}{2} |\nabla \hat{w}_{{\eps}}|^2 - \frac{1}{\delta_{\eps}}W(\hat{w}_{{\eps}})\rt).
\]}
Then $\hat f_\eps=-\delta_\eps \Delta \hat{w}_{\eps} +\frac{1}{\delta_\eps} W'(\hat{w}_{\eps})$ and
 \begin{equation}\label{hyplower}
\hat \mu_{ \eps}(\R^2)+\hat \alpha_{ \eps}(\R^2)= G_{\delta_{\eps},1}(\hat{w}_{{\eps}},B_{1/2\beta_\eps})\leq  G_{\eps,\beta_\eps}(w_\eps,B)\leq \bar \Lambda.
 \end{equation}
\medskip
 
{\sf Step 1.} Let us show that there exists a constant $\theta_0 \in (0,1)$ (only depending on $\bar \Lambda$ and $\tau$) such that
\begin{equation}\label{claimmueps}
\hat\mu_\eps(B_{\delta_\eps}) \geq \theta_0 \delta_\eps.
\end{equation}
We may assume that $\hat\mu_\eps(B_{\delta_\eps}) \le \delta_\eps$ otherwise there is nothing to prove.  We perform a second blow-up by setting for all $z \in B_{1/2\eps}$,
\[
\begin{cases}
\displaystyle \tilde{w}_{\eps}(z)= \hat{w}_{\eps}(\delta_{\eps} z)=w_{\eps}(x_\eps+\eps z), \\
\displaystyle \tilde{f}_{\eps}(z)=\delta_{\eps} \hat{f}_{\eps}(\delta_{\eps} z)=\eps \bar f_{\eps}(x_\eps+\eps z).
  \end{cases}
\]
Changing variables and using \eqref{estimfbar}, we get that
\[
 \int_{B_1} \tilde{f}_{\eps}^2\, dz= \int_{B_{\eps}(x_\eps)}\bar{f}_{ \eps}^2\, dx\le c_0\bar\Lambda\delta_\eps \to 0
\]
and, by the  assumption  $\hat\mu_\eps(B_{\delta_\eps}) \le \delta_\eps$,
$$\int_{B_1} \left(\frac12|\nabla \tilde w_{\eps}|^2 + W(\tilde w_{\eps})\right) dz=c_0 \frac{\hat \mu_{\eps}(B_{\delta_{\eps}})}{\delta_{\eps}} \leq c_0.$$
By Sobolev embedding, this entails that the sequence $\{\tilde w_{\eps}\}_{\eps>0}$ is bounded in $L^6(B_1)$. Using that
\[- \Delta \tilde{w}_{\eps}= \tilde{f}_{\eps} - W'(\tilde{w}_{\eps}) \quad \text{ in }B_1,\]
we have by elliptic regularity together with the Sobolev embedding that there is a universal constant $C>0$ such that
\[
\| {\tilde w}_{\eps}\|_{\mathcal C^{0,1/2}(B_{1/2})} \leq C\| \tilde{w}_{\eps}\|_{H^2(B_{1/2})}\leq L,
\]
where $L>0$ is a constant depending on $\bar \Lambda$. Since by hypothesis $| \tilde{w}_{\eps}(0)-1|\ge \tau$ we find that for $c=\min\{\frac{1}{4L^2},\frac12\}$
\[
 |\tilde{w}_{\eps}-1|\ge \tau/2 \qquad \textrm{ in } B_{c \tau^2 }.  
\]
This implies
\[
 \delta_{\eps}^{-1}\hat\mu_{\eps}(B_{\delta_{\eps}})\ge\frac{1}{c_0\delta_{\eps}^2} \int_{B_{c \tau^2 \delta_{\eps}}} W(\hat{w}_{\eps})\, dz =  \frac{1}{c_0} \int_{B_{c \tau^2}} W(\tilde {w}_\eps)\, dy \geq \frac{\pi c^2\tau^{6}}{4c_0},
\]
concluding the proof of \eqref{claimmueps} for the choice $\theta_0 \leq \min\{\frac{\pi c^2\tau^{6}}{4 c_0},1\}$.

\medskip

 {\sf Step 2.} We next show that
\begin{equation}\label{claimmulower}
\limsup_{\eps \to 0}\hat \mu_\eps(B_1)>0.
\end{equation}
We argue  by contradiction by assuming that 
\begin{equation}\label{eq:hypmueps}
\hat\mu_\eps(B_1)\to 0.
\end{equation}
Using the monotonicity formula \cite[Lemma 4.2]{RogSchat} we can argue as in \cite[Theorem 2.2]{pozzetta} but at the $\eps$-level. We have for $0<\varrho<R<1/(2\beta_\eps)$,
\begin{eqnarray*}
 R^{-1}\hat \mu_\eps(B_R)-\varrho^{-1}\hat \mu_\eps(B_\varrho)& =& -\int_{\varrho}^R s^{-2}{\int_{B_s}\hat\xi_\eps\, dy} \, ds +\frac{1}{c_0}\int_{\varrho}^R \frac{1}{s^3} \int_{\partial B_s}
 \delta_\eps (y\cdot \nabla \hat w_\eps(y))^2\, d\HH^1(y) \, ds\\
 && +\frac{1}{c_0}\int_{\varrho}^{R} \frac{1}{s^2} \int_{B_s} \hat f_\eps(y)(y\cdot \nabla \hat w_\eps(y))\, dy\, ds\\
& \ge & -\int_{\varrho}^R s^{-2}{\int_{B_s}\hat\xi_\eps\, dy} \, ds -\frac{1}{c_0}\int_{\varrho}^{R} \frac{1}{s^2} \int_{B_s} |\hat f_\eps(y)| |y| |\nabla \hat w_\eps(y)|\, dy \, ds.
\end{eqnarray*}
Arguing as in \cite[Proposition 4.3, pp. 695]{RogSchat}, we infer that
\begin{eqnarray*}
 \int_0^{R} \frac{1}{s^2} \int_{B_s} |\hat f_\eps(y)| |y \cdot \nabla \hat w_\eps(y)|\, dy\, ds & \le & \int_{B_R} \frac{|\hat f_\eps(y)| |y \cdot \nabla \hat w_\eps(y)|}{|y|}\, dy\\
& \leq & \int_{B_R} |\hat f_\eps| | \nabla \hat w_\eps|\, dy\\
  & \le & \lt(\frac{1}{\delta_\eps}\int_{B_R} |\hat f_\eps|^2 \, dy\rt)^{\frac{1}{2}}\lt(\int_{B_R} \delta_\eps |\nabla \hat w_\eps|^2\, dy\rt)^{\frac{1}{2}}\\
 & \le & \sqrt 2 c_0 \hat \alpha_\eps(B_R)^{\frac{1}{2}}\hat \mu_\eps(B_R)^{\frac{1}{2}}.
 \end{eqnarray*}
This yields 
\[
 R^{-1}\hat \mu_\eps(B_R)-\varrho^{-1}\hat \mu_\eps(B_\varrho)+\sqrt 2  \hat \alpha_\eps(B_R)^{\frac{1}{2}}\hat \mu_\eps(B_R)^{\frac{1}{2}} \ge-\int_{\varrho}^R s^{-2}{\int_{B_s}\hat\xi_\eps\, dy}\, ds.
\]

We may now invoke \cite[Proposition 4.6, $(54)$]{RogSchat} and \cite[Proposition 3.6]{RogSchat} to get the existence of some $M>2$ and $C>0$ such that for all $\gamma \in (0,1/M)$ and all $\delta_\eps\le \varrho \le R\le 1$,
\begin{eqnarray*}
 \int_{\varrho}^R s^{-2}{\int_{B_s}\hat\xi_\eps\, dy}\, ds & \le & \int_{\varrho}^R s^{-2}{\int_{B_s}\hat\xi_{\eps,+}\, dy}\, ds\\
 & \leq & C\left\{\int_\varrho^R s^{-2+\gamma} \hat \mu_\eps(B_{2s}) \, ds+\delta_\eps^2 \hat \alpha_\eps(B_3)\int_\varrho^R s^{-M\gamma-2}\, ds +\delta_\eps\int_\varrho^R s^{\gamma-2}\, ds\right\}\\
&  \leq &C\left\{\int_\varrho^R s^{-2+\gamma} \hat \mu_\eps(B_{2s})\, ds+ \delta_\eps^2 \hat \alpha_\eps(B_3) \varrho^{-M\gamma-1}+\delta_\eps \varrho^{\gamma-1}\right\}\\
& \leq & C\left\{\int_\varrho^R s^{-2+\gamma} \hat \mu_\eps(B_{2s})\, ds +\delta_\eps^{1-M\gamma} \hat \alpha_\eps(B_3)+\delta_\eps^\gamma\right\}.
\end{eqnarray*}
In conclusion, we have for every $0<\delta_\eps\le \varrho \le R\le 1$
\begin{multline}\label{conseqmonotoneupgrade}
 R^{-1}\hat \mu_\eps(B_R)-\varrho^{-1}\hat \mu_\eps(B_\varrho)+ \sqrt{2} \hat \alpha_\eps(B_R)^{\frac{1}{2}}\hat \mu_\eps(B_R)^{\frac{1}{2}}\\
 \geq -C\lt( \int_\varrho^R s^{-2+\gamma} \hat\mu_\eps(B_{2s})\, ds +\delta_\eps^{1-M\gamma} \hat \alpha_\eps(B_3)+\delta_\eps^\gamma\rt).
\end{multline}

We then define
\[
\varrho_\eps=\sup\{ \varrho\in (\delta_\eps,1) \ : \varrho^{-1}\hat \mu_\eps(B_\varrho)\ge \theta_0\}.
\]
By \eqref{claimmueps} and \eqref{eq:hypmueps}, we have $\varrho_\eps\to 0$. Setting $\varrho=\varrho_\eps$ in  \eqref{conseqmonotoneupgrade} and using \eqref{hyplower}, we find that for every $R \in (\varrho_\eps,1]$,
\begin{multline*}
 R^{-1}\hat \mu_\eps(B_R)
 \ge \theta_0-\sqrt{2}\bar\Lambda^{\frac{1}{2}}\hat \mu_\eps(B_R)^{\frac{1}{2}} -C\lt(2 \theta_0 \int_{\varrho_\eps}^R s^{-1+\gamma}\, ds +\bar\Lambda\delta_\eps^{1-M\gamma} +\delta_\eps^\gamma\rt)\\
 \ge \theta_0-\sqrt{2}\bar\Lambda^{\frac{1}{2}}\hat \mu_\eps(B_1)^{\frac{1}{2}} -C\lt( R^\gamma +\bar\Lambda\delta_\eps^{1-M\gamma} +\delta_\eps^\gamma\rt).
\end{multline*}
Recalling \eqref{eq:hypmueps}, we can pass to the limit as $\eps \to 0$ to get that $\theta_0 \leq CR^\gamma$ which leads to a contradiction provided we choose $R=R(\theta_0)>0$ small enough. We have thus reached the desired contradiction which proves that \eqref{claimmulower} holds.

\medskip

{\sf Step 3.} We may now apply Theorem \ref{theo_RogSchat}. We thus find that $\hat{\mu}_{\eps}\rightharpoonup \mu_{\hat V}$  and $ \hat{\alpha}_{{\eps}} \rightharpoonup \hat \alpha$ weakly* in $\mathcal M_{\rm loc}(\R^2)$ where $\hat V \in \mathbf V_1(\R^2)$ is a $1$-rectifiable integral varifold in $\R^2$ with bounded first variation in $L^2_{\mu_{\hat V}}(\R^2;\R^2)$. Moreover, by Step 2 and in particular \eqref{claimmulower}, we have
$$\int_{\R^2} (1+ |H_{\hat V}|^2)\, d\mu_{\hat V} \geq \mu_{\hat V}(\R^2)>0,$$
while, using $|H_{\hat V}|^2 \mu_{\hat V} \le \hat \alpha$ together with \eqref{hyplower}
\begin{multline*}
\int_{\R^2} (1+ |H_{\hat V}|^2)\, d\mu_{\hat V} \leq   \liminf_{\eps\to 0} \left\{\hat{\mu}_{\eps} (B_{1/2\beta_\eps})+\hat{\alpha}_{\eps} (B_{1/2\beta_\eps)} \right\}\\
 \leq \liminf_{\eps\to 0} G_{\delta_{\eps},1}(\hat{w}_{{\eps}},B_{1/2\beta_\eps}) \leq \liminf_{\eps\to 0} G_{\eps,\beta_{\eps}}(w_{{\eps}},B) \le \bar \Lambda.
\end{multline*}
We can thus apply the generalized Gauss-Bonnet formula, Theorem \ref{thm:Gauss-Bonnet}, which yields
\[
\int_{\R^2} (1+ |H_{\hat{V}}|^2) \,d\mu_{\hat V}\ge 2\int_{\R^2} |H_{\hat{V}}| \,d\mu_{\hat V}\ge 4\pi
\]
and proves \eqref{quantmass}.
\end{proof}

Let us denote by $X$ the set of all points $x \in \Om$ such that there exists a sequence $\{r_k\}_{k \in \N} \searrow 0^+$ with 
$$\liminf_{\eps\to 0} G_{\eps,\beta_\eps}(w_\eps,B_{r_k}(x))\ge 4\pi \quad \text{ for }k \in \N.$$
Note that if $x \in X$, then $m(\overline B_{r_k}(x)) \geq \gamma$ for all $k \in \N$, hence letting $k \to \infty$, $m(\{x\})\geq \gamma$. As a consequence, $X$ is a finite (and thus $\HH^1$-negligible) set. 

\medskip

We next  establish that $V$ is  $1$-rectifiable and integral away from $X$.

\begin{lemma}\label{lem:unif_conv}
The varifold $V$ is a  $1$-rectifiable integral varifold in $\Om \setminus X$ with bounded first variation in $L^2_\mu(\Omega\backslash X;\R^2)$ and, for all open set $A \subset \Om$,
\begin{equation}\label{claimmOm}
 m(A \setminus X)\ge \frac{1}{2}\int_{A\backslash X} (\phi(\nu)+|H_V|^2) \, d\mu.
\end{equation}
In particular, $V=\mathbf v(M,\theta)$, where $M={\rm Supp}(\mu)$ is a relatively closed $1$-rectifiable subset in $\Omega \setminus X$, $\theta=\Theta^1(\mu,\cdot)$ and $\mu=\theta\HH^1 \restr M$.
\end{lemma}

\begin{proof}

Let $B$ be an open ball such that $\overline B \subset A \setminus X$.  By definition of $X$ and Proposition \ref{prop:dichotomie}, we have $w_\eps \to 1$ uniformly in $B$ and thus, given $\delta>0$
$$\inf_{B} w_\eps \geq 1-\delta$$
for $\eps>0$ small enough.  We recall that $\alpha_\eps$ is the `curvature' measure defined in \eqref{eq:def_feps}. According to the energy bound \eqref{eq:nrj-bd},  we get
\begin{equation*}
\frac{1}{2}\int_{B} \phi_\eps(\nabla v_\eps)\, d\mu_\eps+\frac{(2-\delta)^2}{8} \alpha_\eps(B) \le 
 \mathcal G^{(\gamma)}_\eps(u_\eps,v_\eps,w_\eps,B) \leq \Lambda.
 \end{equation*}
We are thus in position to apply Lemma \ref{lemma:lowerphasefield} which implies that  $V$ is a $1$-rectifiable integral varifold in $B$ with bounded first variation in $L^2_\mu(B;\R^2)$. Writing $\mu=\theta \HH^1 \restr M$ for some $1$-rectifiable set $M \subset B$ and some integer-valued $\HH^1$-integrable function $\theta$, by \eqref{lowerphasefield},
\[
m(\overline B)  \geq  \liminf_{\eps \to 0} \mathcal G^{(\gamma)}_\eps(u_\eps,v_\eps,w_\eps,B)\\
  \geq  \frac{1}{2}  \int_{B}  \phi(\nu_V)\, d\mu +\frac{(2-\delta)^2}{8}\int_B |H_V|^2\, d\mu.
\]
Letting $\delta \to 0$ leads to
$$m(\overline B) \geq \frac{1}{2}\int_{B} \left(\phi(\nu_V)+|H_V|^2\right) d\mu.$$

According to the Besicovitch covering Theorem, $\mu$-almost all of  the open set $A \setminus X$ can be covered by countably many pairwise disjoint such closed balls $\{\overline B_j\}_{j \in \N}$. We then have
\begin{multline*}
m(A \setminus X) \geq m\left(\bigcup_{j \in \N} \overline B_j\right) = \sum_{j \in \N} m(\overline B_j) \\
\geq \sum_{j \in \N}\frac{1}{2}\int_{B_j} \left(\phi(\nu_V)+|H_V|^2\right) d\mu=\frac12\int_{A \setminus X}\left(\phi(\nu_V)+|H_V|^2\right) d\mu.
\end{multline*}
Taking $A=\Om$, we then obtain that $V$ is a $1$-rectifiable integral varifold in $\Om \setminus X$ with bounded first variation in $L^2_\mu(\Om \setminus X;\R^2)$. We infer from \cite[Remark 17.9]{Simon_GMT} that $V=\mathbf v(M,\theta)$, where $M={\rm Supp}(\mu)$ is a relatively closed $1$-rectifiable subset in $\Omega \setminus X$ and $\theta=\Theta^1(\mu,\cdot)$. In particular we may write (globally) that $\mu=\theta\HH^1 \restr M$.
\end{proof}

In order to establish the lower bound  for the point energy, we need to analyze  more carefully the limit varifold $V$ inside balls where $w_\eps$ is uniformly close to $1$. The following result states that the phase-field variable $v_\eps$ converges locally uniformly to $1$ away from the support of $V$ (see \cite[Proof of Theorem 1]{HutTone} in the setting of the Allen-Cahn equation).

\begin{lemma}\label{lem:unif}
We have $v_\eps \to 1$ locally uniformly in $\Om \setminus (X \cup {\rm Supp}(\mu))$.
\end{lemma}

\begin{proof}
Let $x_0 \in \Om \setminus (X \cup {\rm Supp}(\mu))$ and $r>0$ be such that $\overline B_{r}(x_0) \subset \Om \setminus (X \cup {\rm Supp}(\mu))$. In particular, by definition of $X$ and Proposition \ref{prop:dichotomie}, we have $w_\eps \to 1$ uniformly in $B_{r}(x_0)$. If
$$\limsup_{\eps \to 0} \|v_\eps -1\|_{L^\infty(B_{r/2}(x_0))} >0,$$
arguing as in the proof of Proposition \ref{prop:dichotomie} (with $\eps$ instead of $\delta_\eps$, and $v_\eps$ instead of $\hat w_\eps$), we obtain that 
$$\limsup_{\eps \to 0} \mu_\eps(B_{r/2}(x_0))>0,$$
hence $x_0 \in {\rm Supp}(\mu)$ which is impossible. Therefore $v_\eps \to 1$ uniformly in $B_{r/2}(x_0)$. Now if $K$ is a compact subset of $\Om \setminus (X \cup {\rm Supp}(\mu))$, it can be covered with a finite number of such balls, hence showing that $v_\eps \to 1$ uniformly in $K$.
\end{proof}

As a consequence, we get that the (integer) multiplicity of $V$ is actually even. While this result might be well-known to experts in the field, we could not find a proof in the literature and thus decided to include it. The proof is based on a careful inspection of \cite[Proposition 5.2]{RogSchat}, see also \cite[Theorem 1]{HutTone}.
\begin{proposition}\label{prop:even}
The multiplicity function satisfies $\theta \in 2\N^*$ $\HH^1$-a.e. on ${\rm Supp}(\mu)$.
\end{proposition}

\begin{proof}
{\sf Step 1.} We start by performing a blow-up at all rectifiability points of $M={\rm Supp}(\mu)\backslash X$ as in \cite[Section 5]{RogSchat}. Let $x_0 \in M$ such that $M$ admits an approximate tangent line at $M$, i.e.,
$$\lim_{\varrho \to 0} \frac{1}{\varrho}\int_{\R^2} \Phi \left(\frac{y-x_0}{\varrho}\right)d\mu = \theta_0 \int_{T} \Phi\, d\HH^1\quad \text{for } \Phi \in \mathcal C_c(\R^2),$$
where, by Lemma \ref{lem:unif_conv}, $\theta_0=\theta(x_0) \in \mathbf N$ and $T \in \mathbf G_1$. Note that $\HH^1$ almost all points of $M$ satisfy this property. Our aim is to prove that $\theta_0$ is an even number. In the sequel, to simplify notation, we assume that $T=e_2^\perp$, where $e_1=(1,0)$ and $e_2=(0,1)$. For  $(x,y) \in \R^2$, we denote by $\Pi(x,y):=(x,0)$ the orthogonal projection onto $T$.  Let
$$B^\pm=\{ x=(x_1,x_2) \in B_1 : \; \pm x_2>0\}.$$

 Fix $R>0$ such that $B_{2R}(x_0) \subset \Om \setminus X$. Since $\mu_\eps \rightharpoonup \mu$ and $\alpha_\eps \rightharpoonup \alpha$ weakly* in $\mathcal M(\Om)$, using the metrizability of  the weak* topology of $\mathcal M_{\rm loc}(\R^2)$ on bounded sets, we can find sequences $\eps_j \to 0$ and $\varrho_j \to 0$ such that, setting
$$\tilde \eps_j:=\frac{\eps_j}{\varrho_j},\quad \tilde v_j(y):=v_{\eps_j}(x_0+\varrho_j y), \quad \tilde f_j(y):=\varrho_j f_{\eps_j}(x_0+\varrho_j y)\quad\text{ for }y \in B_{R/\varrho_j,}$$
and 
$$
\begin{cases}
\tilde \mu_j:=\frac{1}{c_0}\left(\frac{\tilde \eps_j}{2}|\nabla \tilde v_j|^2+\frac{1}{\tilde \eps_j}W(\tilde v_j)\right)\LL^2\restr B_{R/\varrho_j},\\
\tilde \alpha_j:=\frac{1}{c_0\tilde \eps_j}\tilde f_j^2\, \LL^2\restr B_{R/\varrho_j},\\
{\tilde \xi_j:=\frac{\tilde \eps_j}{2}|\nabla \tilde v_j|^2-\frac{1}{\tilde \eps_j}W(\tilde v_j),}
\end{cases}
$$
then
$$-\tilde \eps_j \Delta \tilde v_j +\frac{1}{\tilde \eps_j}W'(\tilde v_j)=\tilde f_j \quad \text{ in }B_{R/\varrho_j},$$
and
$\tilde \eps_j \to 0$,  $\tilde v_j \to 1$ in $L^1(B_2)$ and locally uniformly on $B^\pm$ by Lemma \ref{lem:unif},  $\tilde \mu_j \rightharpoonup \mu_0:=\theta_0 \HH^1\restr T$ weakly* in $\mathcal M(B_2)$, {$\tilde \xi_j \to 0$ strongly in $L^1(B_2)$} and $\tilde \alpha_j \to 0$ strongly in $\mathcal M(B_2)$. Moreover, according to  \cite[Propositions 3.6 \& 4.7]{RogSchat}, there exists
$\tilde \Lambda>0$ such that
\begin{equation}\label{eq:bd-mu}
\tilde \mu_j(B_\varrho(x)) \leq \tilde \Lambda \varrho
\end{equation}
for all $j \in \N$, all $x \in B_1$ and all $\tilde \eps_j \leq \varrho \leq 1$.

Let $\tilde{\nu}_j: \Omega\to \S^{d-1}$ be a Borel  extension of $\nabla \tilde v_j/|\nabla \tilde v_j|$ on $\{\nabla \tilde v_j=0\}$ and $\tilde V_j \in \mathbf V_1(B_1)$ be defined as
$$ \int_{\mathbf G_{1}(\Om)} \Phi(x,S) \, d\tilde V_j(x,S)= \int_{B_1} \Phi\lt(x,{\rm Id}- \tilde{\nu}_j\otimes \tilde{\nu}_j  \rt) d\tilde{\mu}_j \quad \textrm{for } \Phi \in \mathcal C_c(\mathbf G_{1}(B_1)).$$
 According to Theorem \ref{theo_RogSchat} and the above convergences, we infer that $\tilde V_j \rightharpoonup V_0$ weakly* in $\mathcal M(\mathbf G_1(B_1))$ for some stationary $1$-rectifiable integral varifold $V_0$ in $B_1$ with $\mu_{V_0}=\mu_0=\theta_0\HH^1\restr T$. In particular, since $\mu_0(\partial B_1)=\theta_0 \HH^1(T \cap \partial B_1)=0$,
$$\tilde V_j(\mathbf G_1(B_1))=\tilde \mu_j(B_1) \to \mu_0( B_1)=V_0(\mathbf G_1(B_1)),$$
and we obtain that $\tilde V_j \rightharpoonup V_0$ weakly* in $[\mathcal C_b(\mathbf G_1(B_1))]'$. In the sequel, we omit $\sim$ in order to simplify notation.

\medskip

{\sf Step 2.} We next analyze more accurately  the proof of \cite[Proposition 5.2]{RogSchat} to describe the multiplicity $\theta_0$.  Given $\tau \in (0,1)$, $\delta>0$ and $\Lambda>0$ as in \eqref{eq:bd-mu}, let $\omega=\omega(\delta,\tau,\Lambda)>0$ and $L=L(\delta,\tau)>1$ be given by \cite[Proposition 5.5]{RogSchat}. As in \cite{HutTone,RogSchat}, we introduce the set\footnote{Note that our definition of $A_j$ differs from that in \cite{RogSchat}, since we do not require the additional condition $\alpha_j(B_\varrho(z)) \leq \omega \varrho^{\beta_0}$. It turns out that this condition is only useful  in dimension $3$ and is irrelevant in dimension 2 so that $A_j$ reduces to the same definition as in  \cite{HutTone} (see Formulas (69) and (81) in \cite{RogSchat}).}
\begin{multline*}
A_j:=\Big\{z \in B_1 : \, |v_j(z)| \leq 1-\tau \text{ and }\\
|\xi_j|(B_\varrho(z)) + \int_{B_\varrho(z)}  \eps_j |\nabla v_j|^2 \sqrt{1-(\nu_j)_2^2} \, dx\leq \omega\varrho\\
 \text{ for } \eps_j \leq \varrho \text{ with }B_\varrho(z) \subset B_1\Big\}.
\end{multline*}
By a careful inspection of the proof of  \cite[(88), Proposition 5.5]{RogSchat}, we see that  for every $z=(x,\bar y)\in A_j$, there exists $s=s(z)\in \{\pm 1\}$ such that
\begin{multline}\label{closekink}
 s(x,\bar y) v_j(x,\bar y+ y)\ge 1-\frac{\tau}{2} \quad \textrm{for }  y\in [-3L\eps_j, -L\eps_j] \qquad \textrm{and}\\
 s(x,\bar y) v_j(x,\bar y+ y)\le -1+\frac{\tau}{2} \quad \textrm{for }  y\in [L\eps_j, 3L\eps_j].
\end{multline}

According to \cite[Formula (96)]{RogSchat}, we have
\begin{equation}\label{eq:mukAk}
\mu_j(B_1 \setminus A_j ) \to 0.
\end{equation}
Moreover, as in \cite[Formula (5.9)]{HutTone}, we additionally have
\begin{equation}\label{eq:LebAk}
\HH^1\Big(\Pi\big(B_1 \cap \{|v_j| \leq 1-\tau\} \setminus A_j\big)\Big) \to 0.
\end{equation}

We proceed as in \cite[Proposition 5.2]{RogSchat}. Given $x \in T \cap B_1$, we consider a maximal subset  $X=\{x\} \times \{y_1,\ldots,y_K\}$ of  $A_j \cap \Pi^{-1}(\{x\})$ with $|z-z'|\geq 3L\eps_j$ for $z \neq z' \in X$. We observe that $K=K_j(x)$. Morever, by \cite[pp. 713]{RogSchat},
$$K_j(x) \leq \theta_0, \quad A_j \cap \Pi^{-1}(\{x\}) \subset \{x\} \times \bigcup_{k=1}^K (y_k-L\eps_j,y_k+L\eps_j)$$
and
$$\frac{1}{\eps_j} \int_{\{x\} \times (y_k-L\eps_j,y_k+L\eps_j)} W(v_j)\, d\HH^1\leq \frac{c_0}{2}+\delta \quad \text{ for }k \in \{1,\ldots,K\}.$$
Summing up with respect to $k$ yields
$$\frac{1}{\eps_j}\int_{\Pi^{-1}(\{x\}) \cap A_j} W(v_j)\, d\HH^1 \leq \frac{c_0}{2}K_j(x) + \theta_0\delta.$$

We introduce the set
$$G_j:=\Pi(A_j) \setminus \Pi\big(B_1 \cap \{|v_j| \leq 1-\tau\} \setminus A_j\big).$$
On the one hand, by Fubini,
\begin{eqnarray}\label{eq:26071}
\frac{1}{\eps_j} \int_{A_j  \cap  (G_j \times \R)} W(v_j)\, dx& =& \frac{1}{\eps_j} \int_{G_j} \left(\int_{\Pi^{-1}(\{x\}) \cap A_j} W(v_j)\, d\HH^1\right) d\HH^1\nonumber\\
& \leq &\frac{c_0}{2}\int_{G_j}K_j(x)\, d\HH^1(x) +C \delta \theta_0.
\end{eqnarray}
On the other hand,
\begin{eqnarray}\label{eq:26072}
\frac{1}{\eps_j} \int_{A_j\setminus  (G_j \times \R)} W(v_j)\, dx& =& \frac{1}{\eps_j} \int_{\Pi(A_j) \setminus G_j} \left(\int_{\Pi^{-1}(\{x\}) \cap A_j} W(v_j)\, d\HH^1\right) d\HH^1\nonumber\\
& \leq &\left(\frac{\theta_0 c_0}{2} +\theta_0 \delta\right) \HH^1(\Pi(A_j) \setminus G_j)\\
&\le & C\theta_0  \HH^1(\Pi(A_j) \setminus G_j)\le C\theta_0 \HH^1\Big(\Pi\big(B_1 \cap \{|v_j| \leq 1-\tau\} \setminus A_j\big)\Big) \nonumber.
\end{eqnarray}
Gathering \eqref{eq:26071} and \eqref{eq:26072},  we obtain
\begin{multline*}
\mu_j(A_j)\leq \frac{2}{c_0 \eps_j} \int_{A_j } W(v_j)\, dx +\frac{1}{c_0}|\xi_j|(B_1)\\
 \leq \int_{G_j}K_j(x)\, d\HH^1(x) +C\theta_0\left( \delta +\HH^1\Big(\Pi\big(B_1 \cap \{|v_j| \leq 1-\tau\} \setminus A_j\big)\Big)+ |\xi_j|(B_1)\right).
\end{multline*}
Recalling that $K_j(x) \leq \theta_0$ and using \eqref{eq:mukAk}--\eqref{eq:LebAk}, we obtain (recall that $\xi_j$ converges strongly to $0$ in $L^1(B_1)$)
\begin{multline*}
2\theta_0=\theta_0\HH^1(T \cap B_1)=\mu_0(B_1) \leq \liminf_{j \to \infty} \mu_j(B_1) =\liminf_{j \to \infty} \mu_j(A_j) \\
\leq \liminf_{j \to \infty} \int_{G_j}K_j(x)\, d\HH^1(x)+ C\theta_0 \delta \leq  \limsup_{j \to \infty} \int_{G_j}K_j(x)\, d\HH^1(x) + C\theta_0 \delta \\
\leq \theta_0 \limsup_{j \to \infty} \HH^1(G_j) + C\theta_0 \delta\leq 2\theta_0 + C\theta_0 \delta.
\end{multline*}
By a diagonal argument we may find a sequence $\delta_j\to 0$ such that
$$\int_{G_j}K_j(x)\, d\HH^1(x) \to 2\theta_0$$
and
\begin{equation}\label{eq:Gj}
\HH^1(B_1 \cap T \setminus G_j) \to 0.
\end{equation}
As a consequence 
$$\int_{G_j} |\theta_0 -K_j(x)| \, d\HH^1(x) \to 0.$$
Thus, up to a subsequence, $(\theta_0 -K_j(x)){\bf 1}_{G_j}(x)  \to 0$ for $\HH^1$-a.e. $x \in T \cap B_1$. By \eqref{eq:Gj} and since both $\theta_0$ and $K_j$ are integer, Egorov's Theorem ensures the existence of a closed set $C_0 \subset T \cap B_1$ and $j_0 \in \N$ such that $\HH^1(C_0 \cap G_j)>0$ and $K_j \equiv \theta_0$ on $C_0 \cap G_j$ for all $j \geq j_0$.

\medskip

{\sf Step 3.} In this final step we prove that for $j$ large enough, $K_j\in 2\N$ in $G_j\cap C_0$. By the previous step this would conclude the proof of $\theta_0\in 2 \N$. Since $v_j \to 1$ locally uniformly on $B^\pm$, for all $\eta>0$, there exists $j_1=j_1(\eta) \in \N$ such that for all $j \geq j_1$,
\begin{equation}\label{vjlarge}
v_j(x,y) > 1-\tau \quad \text{ for }(x,y) \in B_1 \text{ with }|y|>\eta.
\end{equation}
By Step 2 and using slicing properties of Sobolev functions, for all $j \geq j_0$ and for $\HH^1$ a.e.  $x \in C_0 \cap G_j$, we have  $K_j(x)=\theta_0$ and
$$y \mapsto v_j(x,y) \in H^2(B_1 \cap \Pi^{-1}(\{x\})).$$

Let $j \geq \max(j_0,j_1)$ and fix such a point $x_j$. Set $K=K_j(x_j)$ and, to simplify notation, let
$$\bar{v}_j(y)=v_j(x_j,y), \quad y \in \R$$
which is continuous. By \eqref{closekink}, the intermediate value Theorem and since $x_j \in G_j$,
\begin{equation}\label{422}
\emptyset \neq  \{|\bar{v}_j| \leq 1-\tau\} =\Pi^{-1}(\{x_j\}) \cap A_j \subset \bigcup_{k=1}^{K} (y_k-L\eps_j,y_k+L\eps_j),
\end{equation}
where we recall that  $X=\{x_j\} \times \{y_1,\ldots,y_{K}\}\subset \{|\bar{v}_j| \leq 1-\tau\} $ is the maximal set introduced in Step 2. 

For $k=1,\ldots, K$, let $s_k\in \{\pm 1\}$ be given by \eqref{closekink}. By \eqref{vjlarge} and \eqref{422}, we see that $s_1=1$ while $s_{K}=-1$.  To conclude the proof let us show that for $k=1,\ldots K-1$, $s_{k+1}=-s_k$. Since the proof is the same for every $k$, it is enough to check it for $k=1$. We claim that $s_2=-1$. Since $s_1=1$ we have by \eqref{closekink} that  $\bar{v}_j(y)\le -1+\tau/2$ for $y \in [y_1+L\eps_j,y_1+3L \eps_j]$. If $s_2=1$, by \eqref{closekink} we would have $\bar{v}_j(y)\ge 1 - \tau/2$ for $y\in [y_2-3L\eps_j,y_2-L \eps_j]$. Since $\bar{v}_j$ is continuous, this would imply that $y_1+3L\eps_j < y_2-3L\eps_j$ as well as the existence of $\bar y \in [y_1+3L \eps_j,y_2-3L\eps_j]$ with $\bar{v}_j(\bar y)=0$. As $\bar y \in  \{|\bar{v}_j| \leq 1-\tau\}$ but $\bar y \notin\bigcup_{k=1}^{K} (y_k-L\eps_j,y_k+L\eps_j)$ we reached a contradiction with \eqref{422}. 

Finally an elementary counting argument shows that $K$ is an even number as claimed.
\end{proof}

Together with \eqref{eq:mOmega}, the following result completes the proof of the lower bound.

\begin{lemma}\label{lem:partial-lwbd}
There holds
$$m(\Om) \geq \int_\Om |\nabla u|^2\, dx + \int_{J_u}(\phi(\nu_u)+|H_{J_u}|^2)\, d\HH^1+\gamma \HH^0(\mathcal P_{J_u}).$$
\end{lemma}

\begin{proof}
Arguing as in \cite{BraiChamSol} (see also \cite{BM}), we have for all open set $A \subset \Om$,
$$2\HH^1(J_u \cap A) \leq \liminf_{\eps \to 0} \int_A \frac{1}{c_0}\left(\frac{\eps}{2}|\nabla v_\eps|^2 +\frac{1}{\eps} W(v_\eps)\right)dx=\liminf_{\eps \to 0} \mu_\eps(A),$$
and thus
\begin{equation*}
\mu \geq 2\HH^1 \restr J_u.
\end{equation*}
By Proposition \ref{prop:even}, the $1$-rectifiable integral varifold $V$ has even multiplicity. Thus, owing to Lemma \ref{lem:pointX}, it is enough to show that 
$$m(\Om) \geq \int_\Om |\nabla u|^2\, dx + \frac12 \int_{\Om \setminus X}\left(\phi(\nu_V)+|H_V|^2\right) d\mu+\gamma \HH^0(X).$$

Since the Radon measures $\LL^2$, $\mu\restr (\Om \setminus X)$ and $\HH^0 \restr X$ are mutually singular, it suffices to establish the following lower bounds for the Radon-Nikod\'ym derivatives
\begin{eqnarray}
\frac{dm}{d\LL^2} & \geq & |\nabla u|^2 \quad \LL^2\text{-a.e. in }\Om,\label{eq:L2bis}\\
\frac{dm}{d\mu} & \geq & \frac{1}{2}\left(\phi(\nu_V)+|H_{V}|^2\right)\quad \mu\text{-a.e. in } \Om \setminus X,\label{eq:H1bis}\\
\frac{dm}{d\HH^0\restr X} & \geq & \gamma\quad \text{ on }X.\label{eq:H0bis}
\end{eqnarray}
By definition of $X$, \eqref{eq:H0bis} is immediate, while \eqref{eq:H1bis} results from \eqref{claimmOm}. Concerning the bulk part \eqref{eq:L2bis}, let us consider a Lebesgue point $x_0 \in \Om$ of $\nabla u$ such that $\frac{dm}{d\LL^2}(x_0)$ exists and is finite. Note that $\LL^2$ almost all points in $\Om$ satisfy these properties. Let $\{\varrho_j\}_{j \in \N} \searrow 0^+$ be such that $m(\partial B_{\varrho_j}(x_0))=0$ for all $j \in \N$. Arguing as in \cite{BraiChamSol} (see also \cite{BM}), we infer that 
\begin{equation*}
m(B_{\varrho_j}(x_0))= \lim_{\eps\to 0} \frac14 \int_{B_{\varrho_j}(x_0)} (v_\eps+1)^2|\nabla u_\eps|^2\, dx \ge \int_{B_{\varrho_j}(x_0)} |\nabla u|^2\, dx.
\end{equation*}
Dividing the previous inequality by $\pi\varrho_j^2$ and letting $j \to \infty$ yields
$$\frac{dm}{d\LL^2}(x_0) =\lim_{j\to \infty} \frac{m(B_{\varrho_j}(x_0))}{\pi\varrho_j^2} \geq \lim_{j\to \infty}\frac{1}{\pi\varrho_j^2}\int_{B_{\varrho_j}(x_0)} |\nabla u|^2\, dx=|\nabla u(x_0)|^2.$$
The proof of Lemma \ref{lem:partial-lwbd} is now complete.
\end{proof}

\subsubsection{Proof of Theorem \ref{BBG2}: the upper bound}\label{sec:upperbound}

In this section we discuss the $\Gamma$-$\limsup$ inequality following the construction of \cite[Lemma 4.1]{DoMuRo}. We use the same notation as in Section \ref{sec:ubdMM}. Recall in particular the definition of $q_\eps$ in \eqref{defqeps}.

\medskip

Since we assume that $J_u \in \mathscr C(\Om)$, then $J_u=\bigcup_{i=1}^N \Gamma_i$ where $\Gamma_i=\gamma_i([0,1])$ for some $\gamma_i \in \mathcal C^{2}([0,1];\R^2)$. Let $\{p_1,\ldots,p_M\}=\mathcal P_{J_u}$ and $R>0$ be such that for $i \in \{1,\ldots,M\}$, $B_R(p_i) \cap \mathcal P_{J_u}=\{p_i\}$.  For  $t >0$, we denote the $t$-neighborhood of $J_u$ by
\[U_t=\{x \in \R^2:\;  {\rm dist}(x, J_u)<t\},\]
and we consider the signed distance to $\partial U_{2\delta_\eps}$
$$d_\eps(\cdot)={\rm sdist}(\cdot,\partial U_{2\delta_\eps})={\rm dist}(\cdot,J_u)-2\delta_\eps$$
with the convention that $d_\eps <0$ in $U_{2\delta_\eps}$ and $d_\eps>0$ in $\R^2 \setminus \overline U_{2\delta_\eps}$. Note in particular that $-2 \delta_\eps \leq d_\eps <0$ in $U_{2\delta_\eps}$, and that $d_\eps \to {\rm dist}(\cdot, J_u) \geq 0$ uniformly in $\R^2$.

\medskip

\begin{figure}[!h]
\def\svgwidth{210pt}
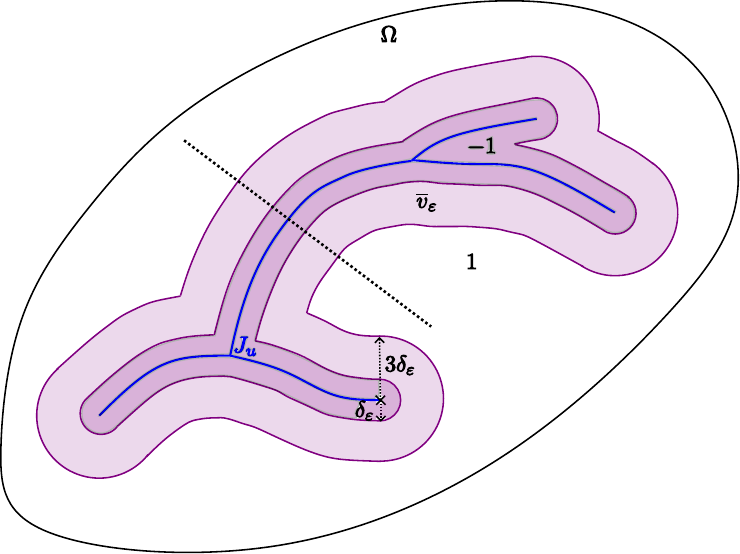
\def\svgwidth{210pt}
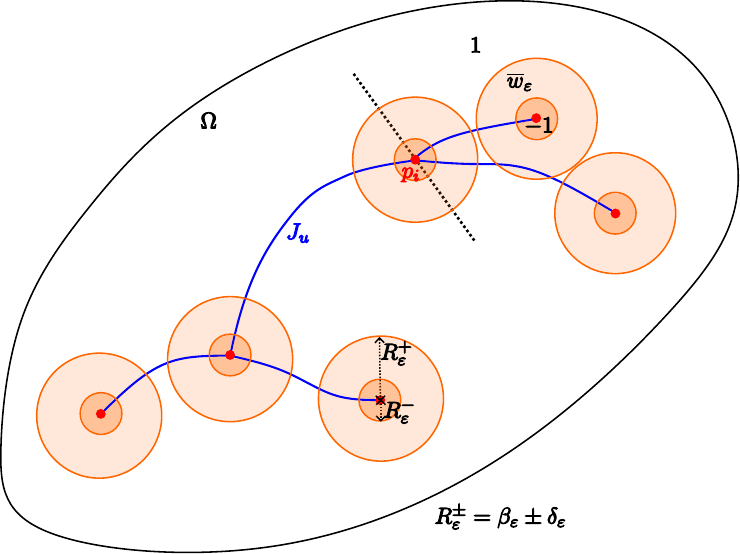
\caption{Definition of $\overline{v}_\epsilon$ (left) and $\overline{w}_\epsilon$ (right) according to offsets of $J_u$ and $\mathcal{P}_{J_u}$.}
\label{figOffsetsRecovery}
\end{figure}

\begin{figure}[!h]
\def\svgwidth{210pt}
\input{ ProfileVeps.pdf_tex}
\def\svgwidth{210pt}
\input{ ProfileWeps.pdf_tex}
\caption{Profile of $\overline{v}_\epsilon$ (left) and $\overline{w}_\epsilon$ (right) accross the respective dotted lines in Figure~\ref{figOffsetsRecovery}.}
\label{figProfilesRecovery}
\end{figure}

We now define our candidate recovery sequence by setting
$$\bar u_\eps(x)=
\begin{cases}
0 & \text{ if } d_\eps (x) \leq  -\frac{3\delta_\eps}{2},\\
\displaystyle \frac{2d_\eps(x)+3\delta_\eps}{\delta_\eps} u(x)  & \text{ if } -\frac{3\delta_\eps}{2} <d_\eps(x)<-\delta_\eps,\\
u(x) & \text{ if } d_\eps(x)\geq-\delta_\eps.
\end{cases}$$
Using that $J_u=\{d_\eps=-2\delta_\eps\}$, we get that $\bar u_\eps \in H^1(\Om)$. In addition, since
$$\{d_\eps<-\delta_\eps\} \subset \{d_\eps<\delta_\eps\} \subset \{x \in \R^2 : \; {\rm dist}(x,J_u)<3\delta_\eps\}=U_{3\delta_\eps}$$
whose Lebesgue measure tends to zero, we get that $\bar u_\eps \to u$ in $L^1(\Om)$. 

\medskip

We next define
\[
\bar v_\eps(x)=q_\eps(d_\eps(x)) \quad \text{ for } x \in \R^2,
\]
which satisfies $|\hat v_\eps| \leq 1$ in $\Om$, $\bar v_\eps \in H^1(\Om)$ and $\hat v_\eps \to 1$ in $L^1(\Om)$. Notice that in general we do not have  $\bar v_\eps\in H^2(\Om)$ because of the junction points $\mathcal P_{J_u}$.  A simple mollification argument allows to approximate $\bar v_\eps$ by smooth functions so that we will ignore this issue.

\medskip

In order to define the last phase-field function $\bar w_\eps$ related to the junction points, we set
$$d^{\beta_\eps}=|\cdot |-\beta_\eps={\rm sdist}(\cdot,\partial B_{\beta_\eps}).$$
We then define for $x \in \R^2$,
\[
\bar w_\eps(x)=\sum_{i=1}^M q_\eps ( d^{\beta_\eps}(x-p_i)).
\]
We immediately check that $ |\bar w_\eps| \leq 1$ in $\Om$, $\bar w_\eps \in H^2(\Om)$ and $\bar w_\eps \to 1$ in $L^1(\Om)$.

\bigskip

We now evaluate separately each terms of the energy $\mathcal G_\eps^\gamma$.

\medskip

\noindent {\sf Step 1.} First of all, since $\bar v_\eps=1$ on $\{d_\eps >\delta_\eps\}$, we have
\[
\frac{1}{\eta_\eps} \int_\Om (1-\bar v_\eps)^2\, dx =  \frac{1}{\eta_\eps} \int_{\{d_\eps \leq \delta_\eps\}} (1-\bar v_\eps)^2\, dx
 \leq \frac{1}{\eta_\eps}  \mathcal L^2(\{d_\eps \leq \delta_\eps\}) \leq \frac{1}{\eta_\eps} \mathcal L^2(U_{3\delta_\eps}).
 \]
Since $J_u \in \mathscr C(\Om)$, its $\HH^1$ measure coincides with its $1$-Minkowski content. Thus
$$\frac{\LL^2(U_t)}{2t} \to \HH^1(J_u) \quad \text{ as }t \to 0,$$
and using \eqref{eq:scaling} together with the definition \eqref{defdeps} of $\delta_\eps$ yields
\begin{equation}\label{eq:term1}
\frac{1}{\eta_\eps} \int_\Om (1-\bar v_\eps)^2\, dx \leq  \frac{6\lambda \eps |\log \eps|}{\eta_\eps}\frac{|U_{3\delta_\eps}|}{6\delta_\eps} \to 0.
\end{equation}

\medskip

\noindent {\sf Step 2.} We introduce the radii 
$$R_\eps^\pm=\beta_\eps\pm \delta_\eps$$ which satisfy, according to \eqref{eq:scaling}, 
$$\frac{R_\eps^\pm}{\delta_\eps} \to \infty$$
so that we can assume that for $\eps>0$ small enough
\begin{equation*}\label{eq:Reps}
R \geq 2\beta_\eps \geq R_\eps^+ \geq R_\eps^- \geq 4 \delta_\eps.
\end{equation*}
 Since $d^{\beta_\eps}(x-p_i)=|x-p_i|-\beta_\eps > \delta_\eps$ for all $x \in \R^2 \setminus \bigcup_{i=1}^M B_{R^+_\eps}(p_i)$, we deduce that  $\bar w_\eps=1$ on $\R^2 \setminus \bigcup_{i=1}^M B_{R^+_\eps}(p_i)$. We thus obtain thanks  to \eqref{eq:scaling},
\begin{equation}\label{eq:term1bis}
 \frac{1}{\eta_\eps}\int_{\Om} (1-\bar w_\eps)^2\, dx \le \sum_{i=1}^M\frac{|B_{R_\eps^+}(p_i)|}{\eta_\eps}\leq \pi M \frac{(\beta_\eps+\lambda \eps |\log\eps|)^2}{\eta_\eps}\to 0.
\end{equation}

\medskip

\noindent {\sf Step 3.}  Next, using that $\bar v_\eps=-1$ on $\{d_\eps<-\delta_\eps\}$, we infer that
\begin{equation}\label{eq:term2}
\int_\Om (\bar v_\eps+1)^2 |\nabla \bar u_\eps|^2 \, dx  \leq  4\int_{\{d_\eps\geq -\delta_\eps\} }|\nabla u|^2\, dx  \leq  4 \int_\Om |\nabla u|^2\, dx.
 \end{equation}

\medskip

\noindent {\sf Step 4.} We now address the point energy. Since $d^{\beta_\eps}(x-p_i)=|x-p_i|-\beta_\eps > \delta_\eps$ for all $x \in \R^2 \setminus \bigcup_{i=1}^M B_{R^+_\eps}(p_i)$ and  $d^{\beta_\eps}(x-p_i)=|x-p_i|-\beta_\eps <- \delta_\eps$ for all $x \in  \bigcup_{i=1}^M B_{R^-_\eps}(p_i)$, we deduce that $\bar w_\eps=1$ in $\R^2 \setminus \bigcup_{i=1}^M B_{R^+_\eps}(p_i)$ and $\bar w_\eps=-1$ in $\bigcup_{i=1}^M B_{R^-_\eps}(p_i)$. It is  thus enough to show that
\begin{equation}\label{convpoint}
\limsup_{\eps \to 0} G_{\eps,\beta_\eps}\left(\bar w_\eps,B_{R^+_\eps}(p_i) \setminus B_{R^-_\eps}(p_i)\right)\leq 4\pi  \quad \text{ for }1 \leq i \leq M.
\end{equation}

Without loss of generality we may assume that $p_i=0$. We now compute separately the two terms in $ G_{\eps,\beta_\eps}\big(\bar w_\eps,B_{R^+_\eps} \setminus B_{R^-_\eps}\big)$. Note that
$$\nabla \bar w_\eps(x)= q'_\eps ( |x|-\beta_\eps) \frac{x}{|x|} \quad \text{ for }x \in B_{R^+_\eps} \setminus B_{R^-_\eps}$$
so that, using polar coordinates  we have
\begin{multline}\label{eq:1246}
 \frac{1}{\beta_\eps} \int_{B_{R_\eps^+}\setminus B_{R_\eps^-}} \left(\frac{\eps}{2}|\nabla \bar w_\eps|^2 + \frac{1}{\eps} W(\bar w_\eps)\right)dx\\
 =\frac{2\pi}{\beta_\eps} \int_{R_\eps^-}^{R_\eps^+}\lt[\frac{\eps}{2}| q_\eps'(r-\beta_\eps)|^2 + \frac{1}{\eps} W( q_\eps(r-\beta_\eps))\rt]r\, dr\\
  \le2\pi \left(1+\frac{\delta_\eps}{\beta_\eps}\right) \int_{-\delta_\eps}^{\delta_\eps} \lt[\frac{\eps}{2}|q_\eps'(r)|^2 + \frac{1}{\eps} W(q_\eps(r))\rt]dr\to 2\pi c_0,
  \end{multline}
where we used the scaling assumption \eqref{eq:scaling} and Lemma \ref{lem:c_0}.

For the other term we  notice that since $\nabla d^{\beta_\eps}(x)=\frac{x}{|x|}$  for all $x \in B_{R_\eps^+}\backslash B_{R_\eps^-}$, we have
$$ -\eps \Delta \bar w_\eps +\frac{1}{\eps} W'(\bar w_\eps)=-\eps q_\eps''(d^{\beta_\eps}) +\frac{1}{\eps} W'(q_{\eps}(d^{\beta_\eps}))- \eps q'_\eps(d^{\beta_\eps}) \Delta d^{\beta_\eps} \quad \text{ in }B_{R_\eps^+}\backslash B_{R_\eps^-}.$$
Moreover, using that $\Delta d^{\beta_\eps}(x)= \frac{1}{|x|}$ for   $x \in B_{R_\eps^+}\backslash B_{R_\eps^-}$, we obtain
\begin{multline*}
 \frac{\beta_\eps}{\eps} \int_{B_{R_\eps^+}\backslash B_{R_\eps^-}} \left(-\eps \Delta \bar w_\eps + \frac{1}{\eps} W'(\bar w_\eps)\right)^2dx\\
 =\frac{ 2\pi\beta_\eps}{\eps} \int_{R_\eps^-}^{R_\eps^+} \left|-\eps q''_\eps(r-\beta_\eps)+\frac{1}{\eps}W'(q_\eps(r-\beta_\eps)) -\eps\frac{ q_\eps'(r-\beta_\eps)}{r}\right|^2 r\, dr\\
  =\frac{ 2\pi\beta_\eps}{\eps} \int_{-\delta_\eps}^{\delta_\eps} \left|-\eps q''_\eps(r)+\frac{1}{\eps}W'(q_\eps(r)) -\eps\frac{ q_\eps'(r)}{r+\beta_\eps}\right|^2 (r+\beta_\eps)\, dr.
\end{multline*}
By Lemma \ref{lem:epsp},
$$\frac{\beta_\eps}{\eps} \int_{-\delta_\eps}^{\delta_\eps} \left|-\eps q''_\eps(r)+\frac{1}{\eps}W'(q_\eps(r))\right|^2 (r+\beta_\eps)\, dr \leq  \frac{\beta_\eps  (\delta_\eps+\beta_\eps)\eps^p}{\eps} \to 0,$$
hence 
\begin{multline*}\label{3}
\limsup_{\eps \to 0} \frac{\beta_\eps}{\eps} \int_{B_{R_\eps^+}\backslash B_{R_\eps^-}} \left(-\eps \Delta \bar w_\eps + \frac{1}{\eps} W'(\bar w_\eps)\right)^2dx\\
=\limsup_{\eps \to 0} \frac{ 2\pi\beta_\eps}{\eps} \int_{-\delta_\eps}^{\delta_\eps} \left|\eps\frac{ q_\eps'(r)}{r+\beta_\eps}\right|^2 (r+\beta_\eps)\, dr\leq \limsup_{\eps \to 0} \frac{ 2\pi\eps\beta_\eps}{\beta_\eps-\delta_\eps} \int_{-\delta_\eps}^{\delta_\eps} | q_\eps'(r)|^2 \, dr.
\end{multline*}
Finally by \eqref{eqc0} together with the scaling laws \eqref{eq:scaling}, we obtain that
$$\limsup_{\eps \to 0} \frac{ 2\pi\eps\beta_\eps}{\beta_\eps-\delta_\eps} \int_{-\delta_\eps}^{\delta_\eps} | q_\eps'(r)|^2 \, dr = \limsup_{\eps \to 0}  \frac{ 2\pi}{1-\frac{\delta_\eps}{\beta_\eps}} \int_{-\delta_\eps/2}^{\delta_\eps/2} \eps | q_\eps'(r)|^2 \, dr=  2\pi c_0.$$
This yields
\begin{equation}\label{eq:1249}
 \limsup_{\eps \to 0}\frac{\beta_\eps}{\eps} \int_{B_{R_\eps^+}\backslash B_{R_\eps^-}} \left(-\eps \Delta \bar w_\eps + \frac{1}{\eps} W'(\bar w_\eps)\right)^2dx \le 2\pi c_0,
 \end{equation}
and gathering \eqref{eq:1246} and \eqref{eq:1249}  concludes the proof of \eqref{convpoint}.

\medskip

\noindent {\sf Step 5.}  We now address the jump term. Since $\nabla \bar v_\eps=q'_\eps(d_\eps) \nabla d_\eps$ and $|\nabla d_\eps|=1$  a.e. in $\Om$, we infer that
\[
\int_{\Om } \phi_\eps\left(\nabla \bar v_\eps\right) \left(\frac{\eps}{2}|\nabla \bar v_\eps|^2 + \frac{W(\bar v_\eps)}{\eps}\right) dx
=  \int_{\{|d_\eps|\leq \delta_\eps\}} \phi_\eps( q_\eps'(d_\eps)\nabla d_\eps)\left(\frac{\eps}{2}|q'_\eps(d_\eps)|^2 + \frac{W(q_\eps(d_\eps))}{\eps}\right)dx.
\]
We proceed as in the proof of \eqref{claimupperBBG1}. Indeed, using the co-area formula, \eqref{boundphieps}, \eqref{eq:term3.42}, $\HH^1(\{d_\eps=t\})\le C \HH^1(J_u)$ for $|t|\le \delta_\eps$ and the fact that $\frac{\eps}{2}|q_\eps'|^2=\frac{1}{\eps} W(q_\eps)$ on $[- \delta_\eps/2,\delta_\eps/2]$, we find that
\begin{multline}\label{eq:term3.11}
\limsup_{\eps\to 0} \int_{\Om } \phi_\eps\left(\nabla \bar v_\eps\right) \left(\frac{\eps}{2}|\nabla \bar v_\eps|^2 + \frac{W(\bar v_\eps)}{\eps}\right) dx\\
=\limsup_{\eps\to 0} \int_{- \frac{\delta_\eps}{2}}^{\frac{\delta_\eps}{2}}
\eps|q'_\eps(t)|^2 \left(\int_{\{d_\eps=t\}} \phi_\eps(q_\eps'(t)\nabla d_\eps)\, d\HH^1\right) dt
\nonumber.
\end{multline}
Now for $|t|\le \delta_\eps/2$ we have by \eqref{eq:LussardiPhirVsPhi},
\begin{eqnarray*}
 \int_{\{d_\eps=t\}} \phi_\eps(q_\eps'(t)\nabla d_\eps)\, d\HH^1& \le & \int_{\{d_\eps=t\}} \phi(\nabla d_\eps)\, d\HH^1+\frac{L r_\eps}{2|q_\eps'(t)|} \HH^1(\{d_\eps=t\})\\
&  \le & \int_{\{d_\eps=t\}} \phi(\nabla d_\eps)\, d\HH^1+C\frac{L r_\eps}{|q_\eps'(t)|} \HH^{1}(J_u).
\end{eqnarray*}
As
\[
 \limsup_{\eps\to 0}  r_\eps \int_{- \frac{\delta_\eps}{2}}^{\frac{\delta_\eps}{2}} \eps |q_\eps'|dt\leq 2\limsup_{\eps\to 0}  r_\eps \eps=0,
\]
we obtain from \eqref{eqc0}
\begin{multline*}
 \limsup_{\eps\to 0} \int_{\Om } \phi_\eps\left(\nabla \bar v_\eps\right) \left(\frac{\eps}{2}|\nabla \bar v_\eps|^2 + \frac{W(\bar v_\eps)}{\eps}\right) dx
=\limsup_{\eps\to 0}\int_{- \frac{\delta_\eps}{2}}^{\frac{\delta_\eps}{2}}
\eps|q'_\eps(t)|^2 \left(\int_{\{d_\eps=t\}} \phi(\nabla d_\eps)\, d\HH^1\right) dt\\
\le \limsup_{\eps\to 0}\lt(\int_{- \frac{\delta_\eps}{2}}^{\frac{\delta_\eps}{2}}
\eps|q'_\eps(t)|^2 dt\rt) \sup_{|t|\le \delta_\eps}\int_{\{d_\eps=t\}} \phi(\nabla d_\eps)\, d\HH^1\\
\le c_0 \limsup_{\eps\to 0}\sup_{|t|\le \delta_\eps}\int_{\{d_\eps=t\}} \phi(\nu_{\{d_\eps=t\}})\, d\HH^1.
\end{multline*}
In the last line we used that $\nabla d_\eps= \nu_{\{d_\eps=t\}}$ on $\{d_\eps=t\}$. Writing $\{d_\eps=t\}$ as a double graph on $J_u$ away from $\mathcal{P}_{J_u}$ we see that
\[
 \limsup_{\eps\to 0}\sup_{|t|\le \delta_\eps}\int_{\{d_\eps=t\}} \phi(\nu_{\{d_\eps=t\}})\, d\HH^1=2 \int_{J_u} \phi(\nu_u)d\HH^1
\]
so that finally
\begin{equation}\label{eq:term3}
\limsup_{\eps \to 0}\int_{\Om } \phi_\eps\left(\nabla \bar v_\eps\right) \left(\frac{\eps}{2}|\nabla \bar v_\eps|^2 + \frac{W(\bar v_\eps)}{\eps}\right) dx \leq 2c_0 \int_{J_u} \phi(\nu_u)\, d\HH^1.
\end{equation}

\medskip

\noindent {\sf Step 6.}  It remains to address the curvature term. Let $\Omega_\eps= \Omega\backslash \bigcup_{i=1}^M B_{R_\eps^-}(p_i)$. We first notice that $\bar w_\eps=-1$ in $\bigcup_{i=1}^M B_{R_\eps^-}(p_i)$ and that $\bar v_\eps$ is locally constant and equal to $\pm 1$ in $\Omega_\eps\cap\{|d_\eps|>\delta_\eps\}$. Therefore
\[
 \frac{1}{\eps}\int_\Om  \left(-\eps \Delta \bar v_\eps +\frac{1}{\eps} W'(\bar v_\eps)\right)^2 (1+\bar w_\eps)^2\, dx\le \frac{4 }{\eps}\int_{\Om_\eps\cap\{|d_\eps|\le \delta_\eps\}}  \left(-\eps \Delta \bar v_\eps +\frac{1}{\eps} W'(\bar v_\eps)\right)^2 \, dx.
\]
Moreover, we recall that we assumed that $J_u\in \mathscr C(\Om)$ and thus $J_u=\bigcup_{i=1}^N \Gamma_i$ where $\Gamma_i=\gamma_i([0,1])$ for some $\gamma_i \in \mathcal C^{2}([0,1];\R^2)$. In particular for every $\eps$ small enough we may write
\[
 \Om_\eps\cap\{|d_\eps|\le \delta_\eps\}=\bigcup_{i=1}^N \Omega_\eps^i
\]
with,
\[
 d_\eps(x)= {\rm dist}(x,\Gamma_i)-2\delta_\eps \qquad \textrm{for } x\in \Omega_\eps^i.
\]
Moreover, in each $\Omega_\eps^i$ the orthogonal projection $\Pi_{J_u}$ on $J_u$ coincides with the orthogonal projection on $\Gamma_i$ and is uniquely defined. We may thus decompose the integral as
\[
 \frac{4 }{\eps}\int_{\Om_\eps\cap\{|d_\eps|\le \delta_\eps\}}  \left(-\eps \Delta \bar v_\eps +\frac{1}{\eps} W'(\bar v_\eps)\right)^2 \, dx=\sum_{i=1}^N \frac{4 }{\eps}\int_{\Om_\eps^i}  \left(-\eps \Delta \bar v_\eps +\frac{1}{\eps} W'(\bar v_\eps)\right)^2 \, dx.
\]
We claim that for every $i=1,\dots,N$,
\begin{equation}\label{claimcurvMS}
 \limsup_{\eps\to 0} \frac{4 }{\eps}\int_{\Om_\eps^i}  \left(-\eps \Delta \bar v_\eps +\frac{1}{\eps} W'(\bar v_\eps)\right)^2 \, dx\le 8 c_0 \int_{\Gamma_i} |H_{\Gamma_i}|^2 d\HH^1.
\end{equation}
After summation over $i$ this would yield
\begin{equation}\label{eq:term4}
\limsup_{\eps \to 0} \int_\Om  \frac{1}{\eps} \left(-\eps \Delta\bar v_\eps +\frac{1}{\eps} W'(\bar v_\eps)\right)^2 (1+\bar w_\eps)^2 dx\leq 8 c_0 \int_{J_u} |H_{J_u}|^2\, d\HH^1.
\end{equation}
We now prove \eqref{claimcurvMS}. In $\Om_\eps^i$ we have $|\nabla d_\eps|=1$ and
\begin{equation*}
 -\eps \Delta \bar v_\eps +\frac{1}{\eps} W'(\bar v_\eps)=-\eps q_\eps''(d_\eps) +\frac{1}{\eps} W'(q_\eps(d_\eps))- \eps q'_\eps(d_\eps) \Delta d_\eps
\end{equation*}
so that
\[
  \frac{4 }{\eps}\int_{\Om_\eps^i}  \left(-\eps \Delta \bar v_\eps +\frac{1}{\eps} W'(\bar v_\eps)\right)^2 \, dx=  \frac{4}{\eps}\int_{\Om_\eps^i} \left(-\eps q_\eps''(d_\eps) +\frac{1}{\eps} W'(q_\eps(d_\eps))- \eps q'_\eps(d_\eps) \Delta d_\eps\right)^2dx.
\]
 We now notice that by Lemma \ref{lem:epsp}, and the co-area formula,
\begin{multline*}\frac{1}{\eps}\int_{\Om_\eps^i} \left(-\eps q_\eps''(d_\eps) +\frac{1}{\eps} W'(q_\eps(d_\eps))\right)^2dx\\=\frac{1}{\eps}\int_{-\delta_\eps}^{\delta_\eps}  \left(-\eps q_\eps''(t) +\frac{1}{\eps} W'(q_\eps(t))\right)^2 \HH^1(\{d_\eps=t\}\cap \Om_\eps^i) \, dt \leq C\eps^{p-1}\to 0,
\end{multline*}
hence
\begin{multline*}
\limsup_{\eps \to 0} \frac{1}{\eps}\int_{\Om_\eps^i} \left(-\eps q_\eps''(d_\eps) +\frac{1}{\eps} W'(q_\eps(d_\eps))- \eps q'_\eps(d_\eps) \Delta d_\eps\right)^2dx
=\limsup_{\eps \to 0} \eps \int_{ \Om_\eps^i} \left(q'_\eps(d_\eps) \Delta d_\eps\right)^2\, dx.
\end{multline*}
Using the co-area formula again we find
\begin{eqnarray*}
 \eps \int_{ \Om_\eps^i} \left(q'_\eps(d_\eps) \Delta d_\eps\right)^2\, dx &=&\int_{-\delta_\eps}^{\delta_\eps} \eps |q_\eps'(t)|^2 \int_{\Om_\eps^i\cap \{d_\eps=t\}} |\Delta d_\eps|^2 \, d\HH^1\,  dt\nonumber\\
 & \le  &\lt( \int_{-\delta_\eps}^{\delta_\eps} \eps |q_\eps'(t)|^2 dt\rt)\lt(\sup_{|t|\le \delta_\eps}  \int_{\Om_\eps^i\cap \{d_\eps=t\}} |\Delta d_\eps|^2 \, d\HH^1\rt).\label{eq:term4.2}
\end{eqnarray*}
Thanks to \eqref{eqc0}, in order to conclude the proof of \eqref{claimcurvMS}, we are left with the proof of
\begin{equation}\label{eq:term4.3}
 \limsup_{\eps\to 0} \sup_{|t|\le \delta_\eps}  \int_{\Om_\eps^i\cap \{d_\eps=t\}} |\Delta d_\eps|^2 \, d\HH^1\le 2 \int_{\Gamma_i} |H_{\Gamma_i}|^2 \, d\HH^1.
\end{equation}
For this we notice (see \cite[Appendix B]{Giusti}) that for $x\in  \Om_\eps^i\cap \{d_\eps=t\}$,
\[
 |\Delta d_\eps(x)| \le \frac{|H_{\Gamma_i}(\Pi_{J_u}(x))|}{1 \pm (2\delta_\eps+t)|H_{\Gamma_i}(\Pi_{J_u}(x))|}\le (1+C\delta_\eps)|H_{\Gamma_i}(\Pi_{J_u}(x))|.
\]
Writing $\Om_\eps^i\cap \{d_\eps=t\}$ as a double normal graph over $\Gamma_i$ then leads to \eqref{eq:term4.3}.

\medskip

\noindent {\sf Step 7.}  Gathering finally \eqref{eq:term1}, \eqref{eq:term1bis}, \eqref{eq:term2}, \eqref{convpoint}, \eqref{eq:term3} and \eqref{eq:term4}, we infer that
$$\limsup_{\eps \to 0} \mathcal G_\eps^{(\gamma)}(\bar u_\eps,\bar v_\eps,\bar w_\eps) \to \int_\Om |\nabla u|^2\, dx + \int_{J_u}(1+|H_{J_u}|^2)\, d\HH^1+\gamma\HH^0(\mathcal P_{J_u})=\mathcal G^{(\gamma)}(u),$$
which completes the proof of the upper bound.
\hfill$\Box$

\section{Appendix: the Gauss-Bonnet formula for varifolds}

The aim of this Appendix is to give a self contained proof of a generalized Gauss-Bonnet formula for $1$-rectifiable integral varifolds with square integrable bounded first variation. The proof below follows the lines of the arguments in \cite[Theorem 24.1]{menne2023sharp}, in a simplified setting.

\begin{theorem}\label{thm:Gauss-Bonnet}
Let $V\in \mathbf V_1(\R^2)$ be a $1$-rectifiable integral varifold in $\R^2$ with bounded first variation in $L^2_{\mu_V}(\R^2;\R^2)$ and such that
\begin{equation}\label{eq:assumpt}
0<\int_{\R^2}(1+|H_V|^2)\, d\mu_V<\infty.
\end{equation}
Then,
\begin{equation}\label{GBVarifold}
|\delta V|(\R^2)= \int_{\R^2} |H_V| \, d\mu_V \ge 2\pi.
\end{equation}
\end{theorem}

To simplify notation, we write in the sequel $\mu=\mu_V$ and $H=H_V$. By \cite[Remark 17.9]{Simon_GMT}, we can write $V=\mathbf v(M,\theta)$, where $M={\rm Supp}(\mu)$ is a closed $\HH^1$-rectifiable set and $\theta :M \to \N \setminus \{0\}$ is an $\HH^1$-integrable function. In particular
$$\mu=\theta \HH^1\restr M.$$
For later use, we define the measure $\psi=|\delta V|=|H|\mu$ and we also set
\begin{equation}\label{eq:alpha}
\alpha=\psi(\R^2)=\int_{\R^2}|H|\, d\mu.
\end{equation}

\begin{remark}
It is not enough to assume that $V$ has bounded first variation since e.g. for a square, $V=\mathbf v(\partial (0,1)^2,1)$, we have $\mu_V=\HH^1\restr \partial (0,1)^2$ and $$\delta V=  (e_1+e_2) \delta_{(0,0)}+(e_2-e_1)\delta_{(1,0)}+(e_1-e_2)\delta_{(0,1)}-(e_1+e_2)\delta_{(1,1)},$$
where $e_1=(1,0)$ and $e_2=(0,1)$, hence $|\delta V|(\R^2)= 4 \sqrt{2}< 2\pi$.
\end{remark}

The proof of Theorem \ref{thm:Gauss-Bonnet} is based  on a Lusin type property for the Gauss map (see Proposition \ref{prop:lusin} below). This result, which is the two-dimensional version of \cite[Theorem 21.16]{menne2023sharp}, will be proven in Subsection \ref{sec:lusin}.

\medskip

Let us fix some notation. Given a vector $\nu=(\nu_1,\nu_2) \in \S^1$, we write $\nu^\perp=(-\nu_2,\nu_1)$ and $\langle \nu^\perp\rangle={\rm Vect}(\nu^\perp)$. For all $(x,\nu) \in \R^2 \times \S^1$, we denote by $\Pi_{\S^1}(x,\nu)=\nu$ (resp. $\Pi_{\R^2}(x,\nu)=x$) the projections onto $\S^1$ (resp. $\R^2$).

\begin{proposition}\label{prop:lusin}
Let $N \subset M$ be such that $\HH^1(N)=0$ and
\begin{equation}\label{eq:B}
B=\Big\{(x,\nu)\in M \times \S^1  :\;  M \subset \{y \in \R^2  : \;(y-x)\cdot \nu\ge 0\}\Big\}.
\end{equation}
Then
$$\HH^1\left(\Pi_{\S^1}\left((N \times \S^1) \cap B \right)\right)=0.$$ 
\end{proposition}

Before proving Theorem~\ref{thm:Gauss-Bonnet}, let us check the following property satisfied by the set $B$:

\begin{figure}[!h]
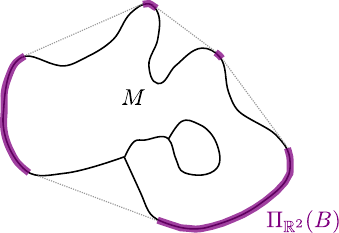
\caption{}
\end{figure}

\begin{lemma}\label{lem:CR}
Let us consider the set $B$ defined in \eqref{eq:B}, then
\begin{equation}\label{eq:1418}
B \subset  \Big\{ (x,\nu)\in M \times \S^1  :\ T_x M \text{ exists and } \ T_x M =\langle \nu^\perp\rangle \Big\}.
\end{equation}
\end{lemma}

\begin{proof}
Let $(x,\nu) \in B$ and assume without loss of generality that $x=0$. As in the proof of Lemma \ref{lemdenssing}, we perform a blow-up at the point $0$. Given $\varrho > 0$, let $V_{\varrho} \in \mathcal M_{\rm loc}(\mathbf G_1(\R^2))$ be the $1$--varifold defined by 
\begin{equation*}
\langle V_{\varrho},\Phi\rangle:=\frac{1}{\varrho}\int_{\mathbf G_1(\R^2)} \Phi\left(\frac{y}{\varrho},S\right)\, dV(y,S) \quad\text{ for }\Phi \in \mathcal C_c(\mathbf G_1(\R^2)).
\end{equation*}
Up to a subsequence $\{\varrho_j\}_{j \in \N} \searrow 0$, there exists a stationary $1$-rectifiable integral varifold $V_0$ such that $V_{\varrho_j} \wto V_0$ weakly* in $\mathcal M_{\rm loc}(\mathbf G_1(\R^2))$. Moreover, by \cite[Theorem 19.3]{Simon_GMT}, $V_0=\mathbf v(C,\psi)$ where $C$ is an $\HH^1$-rectifiable cone and $\psi:C \to \N\setminus \{0\}$ is a locally $\HH^1$-integrable function with $\psi(x)=\psi(t x)$, for all $x \in C$ and $t>0$. 
We furthermore claim that 
\begin{equation}\label{eq:claimC}
C={\rm Supp}(\mu_{V_0}) =\langle \nu^\perp\rangle.
\end{equation}
Indeed, on the  one hand $(0,\nu) \in B$ so that for every $\varrho > 0$, ${\rm Supp}(\mu_{V_\varrho})=\varrho^{-1} M \subset \{y \in \R^2: \ y \cdot \nu \geq 0\}$ and we deduce that $C={\rm Supp}(\mu_{V_0}) \subset \{y \in \R^2: \ y \cdot \nu \geq 0\}$ as well.
On the other hand, the set $C$ being an $\HH^1$-rectifiable cone in $\R^2$, there exists a finite collection of half-lines $\{\R^+ u_i\}_{i=1}^I$ with $I\le 2 \Theta^1(\mu,0)$ such that
$$C=\bigcup_{i=1}^I \R^+ u_i ,$$
where $\psi|_{\R^+ u_i}=\psi_i \in \N\setminus \{0\}$ is constant and $u_i \in \R^2$ satisfy $u_i \cdot \nu \geq 0$. Let $\zeta \in \mathcal C^1_c(\R^2)$ be a scalar function satisfying $\zeta(0)=1$ then
$$\delta V_0 (\zeta \nu) = \int_C {\rm div}^{C} (\nu \zeta) \psi \, d\HH^1= \sum_{i =1}^I \psi_i \nu\cdot u_i \int_{L_i} u_i \cdot \nabla \zeta\, d\HH^1= - \sum_{i =1}^I \psi_i u_i \cdot \nu .$$
Recalling that $V_0$ is stationary, $\psi_i \geq 1$ and $u_i \cdot \nu\geq 0$, we deduce that $u_i\cdot \nu=0$ for all $i \in \N$. Therefore,
$$C \subset \langle \nu^\perp\rangle \quad \text{and} \quad \mu_{V_0}=\psi^+ \HH^1 \restr (\R^+ \nu^\perp) + \psi^- \HH^1 \restr (\R^- \nu^\perp)$$
for some $\psi^\pm \in \N\setminus \{0\}$. Using again the stationarity of $V_0$, we have
$
0 = \delta V_0 = (\psi^- - \psi^+) \nu^\perp
$
which actually yields $\psi^+=\psi^-$. Thus $\mu_{V_0}=\psi \HH^1 \restr \langle \nu^\perp\rangle$ for some $\psi \in \N \setminus \{0\}$, hence $C=\langle \nu^\perp \rangle$, which shows \eqref{eq:claimC}. 
Therefore, $M$ admits an approximate tangent line at $0$ given by $T_0 M=\langle \nu^\perp \rangle$, and it implies the desired inclusion \eqref{eq:1418}.
\end{proof}

\begin{proof}[Proof of Theorem~\ref{thm:Gauss-Bonnet}]
 We  first record  that, as a consequence of Lemma \ref{lem:monoton}, and Young inequality,   there exists $R_0 > 0$ such that for all $x \in M$,
\begin{equation} \label{loweruppermu}
\begin{cases}
 \varrho \leq \mu(B_\varrho(x)) &  \text{ for all } \varrho \in (0, R_0] ,\\
  \mu(B_\varrho(x)) \leq \alpha \varrho  & \text{ for all } \varrho > 0,
  \end{cases}
\end{equation}
where we recall that $\alpha$ is defined in \eqref{eq:alpha}. 
As a consequence of the lower density estimate in \eqref{loweruppermu} we obtain the compactness of $M$.
Indeed, let us consider a maximal set $M'=\{x_i\}_{i \in I} \subset M$ such that $|x_i-x_j|\geq 2R_0$ for all $i$, $j \in I$ with $i \neq j$. By maximality, we must have that
\begin{equation}\label{eq:cov}
M \subset \bigcup_{i \in I} B_{2R_0}(x_i).
\end{equation}
Moreover, $B_{R_0}(x_i) \cap B_{R_0}(x_j) = \emptyset$ for all $i$, $j \in I$ with $i \neq j$. Thus, by \eqref{eq:assumpt} and \eqref{loweruppermu},
$$R_0 \#(I) \leq \sum_{i \in I} \mu( B_{R_0}(x_i)) =\mu\left( \bigcup_{i \in I} B_{R_0}(x_i)\right) \leq \mu(M) = \mu(\R^2)<\infty,$$
which shows that $I$ is a finite set. Thus, by \eqref{eq:cov}, the closed set $M$ is bounded hence compact.

\medskip

\noindent {\sf Step 1.} We now prove a coarea formula, see \eqref{eq:coarea}.
Since $H \in L_{\mu}^2(\R^2;\R^2)$, by \cite[Theorem 5.1]{schatzle2004quadratic} combined with \cite[Theorem 3.1]{SchatzLower} (see also \cite[Theorem 1]{Menne}) we have that $M$ is $\mathcal C^2$-rectifiable so that there exists a countable collection $\{M_i\}_{i \in \N}$ of $\mathcal C^2$ curves with $\HH^{1}(M\backslash \bigcup_i M_i)=0$. Moreover, by \cite[Corollary 4.2]{SchatzLower}, see also \cite[Theorem 2]{AmbMas} or  \cite[Theorem 1]{Menne}, we have the following locality property of the mean curvature
\begin{equation}\label{localH}
 H=H_{M_i} \qquad \HH^1\text{-a.e. on } M \cap M_i.
\end{equation}
Fix $i \in \N$ and let $\gamma_i:[0,L_i] \to \R^2$ be an arc-length parametrization of $M_i$. Arguing as in \cite[Lemma 21.3]{menne2023sharp}, for every Borel set $A \subset \{ (x,\nu)\in M\times \S^1  : \; T_x M =\langle \nu^\perp\rangle\}$, we have
\begin{multline*}
 \int_{M_i} \HH^0(\{\nu \in \S^1 : \ (x,\nu)\in A\}) |H_{M_i}(x)|\, d\HH^1(x)\\
=  \int_{0}^{L_i} \HH^0(\{\nu \in \S^1 : \ (\gamma_i(t),\nu)\in A\}) |\ddot \gamma_i(t)|\, dt\\
 =  \int_0^{L_i} {\bf 1}_A(\gamma_i(t), \dot{\gamma_i}(t)^\perp)|\ddot{\gamma_i}(t)|\, dt+\int_0^{L_i} {\bf 1}_A(\gamma_i(t), -\dot{\gamma_i}(t)^\perp)|\ddot{\gamma_i}(t)|\, dt.
 \end{multline*}
Moreover, by the generalized area formula, see e.g. \cite[Theorem 2.91]{AFP}, 
$$\int_0^{L_i} {\bf 1}_A(\gamma_i(t), \pm\dot{\gamma_i}(t)^\perp)|\ddot{\gamma_i}(t)|\, dt=\int_{\S^1}\left[\int_{\{t \in [0,L_i] : \, \dot \gamma_i(t)=\nu\}} {\bf 1}_A(\gamma_i(t),\pm \dot \gamma_i(t)^\perp)\, d\HH^0(t)\right]d\HH^1(\nu),$$
hence
\begin{equation*}
 \int_{M_i} \HH^0(\{\nu \in \S^1 : \ (x,\nu)\in A\}) |H_{M_i}(x)|\, d\HH^1(x)
  = \int_{\S^1} \HH^0(\{x \in M_i : \;  (x,\nu)\in A \})\,  d\HH^1(\nu).
\end{equation*}

Let us decompose the $\HH^1$-rectifiable set $M$ as $M=N \cup \bigcup_i A_i$ where $N \subset M$ is $\HH^1$-negligible, and $A_i \subset M_i$ for all $i \in \N$ with $A_i \cap A_j=\emptyset$ if $i \neq j$. Using the locality of the mean curvature \eqref{localH}, we get that 
\begin{align*}
& \int_{A_i} \HH^0(\{\nu \in \S^1 : \ (x,\nu)\in A  \}) |H(x)|\, d\HH^1(x) \\
 = & \int_{M_i} \HH^0(\{\nu \in \S^1 : \ (x,\nu)\in A \cap (A_i \times \S^1) \}) |H_{M_i}(x)|\, d\HH^1(x)\\
  = & \int_{\S^1} \HH^0(\{x \in M_i : \;  (x,\nu)\in A \cap (A_i \times \S^1)\})\,  d\HH^1(\nu)\\
  =  & \int_{\S^1} \HH^0(\{x \in M : \;  (x,\nu)\in A \cap (A_i \times \S^1)\})\,  d\HH^1(\nu).
\end{align*}
Summing up the previous equality and using that $N=M \setminus \bigcup_i A_i$ is $\HH^1$-negligible yields for every Borel set $A \subset \{ (x,\nu)\in M \times \S^1  : \; T_x M =\langle \nu^\perp\rangle\}$,
\begin{multline}\label{eq:coarea}
 \int_{M} \HH^0(\{\nu \in \S^1 : \ (x,\nu)\in A  \}) |H(x)|\, d\HH^1(x) \\
  =   \int_{\S^1} \HH^0(\{x \in M : \;  (x,\nu)\in A \})\,  d\HH^1(\nu) \\
  - \int_{\S^1} \HH^0(\{x \in M : \;  (x,\nu)\in A \cap (N \times \S^1) \})\,  d\HH^1(\nu).
\end{multline}
 Specifying \eqref{eq:coarea} with $A=B$ defined in \eqref{eq:B}, and using Proposition \ref{prop:lusin}, we get that
\begin{equation}\label{coareaV}
 \int_M \HH^0(\{\nu\in \S^1  : \ (x,\nu)\in B\})\, |H(x)|\, d\HH^1(x) = \int_{\S^1} \HH^0(\{x\in M  : \  (x,\nu)\in B\})\,  d\HH^1(\nu).
\end{equation}

\medskip

\noindent {\sf Step 2.} We first prove that for every $\nu\in \S^1$, 
\begin{equation}\label{eq:1353}
\HH^0(\{x\in M  :\ (x,\nu)\in B\})\ge 1.
\end{equation}
Indeed, let $\nu \in \S^1$. Since $M$ is compact, the continuous function $y \mapsto y \cdot \nu$ reaches its minimum at some $\bar x \in M$ that is, 
for all $y \in M$, $(y-\bar x) \cdot \nu \geq 0$, which shows that $(\bar x,\nu) \in B$, and establishes \eqref{eq:1353}. We then show that for all $x \in M$,
\begin{equation}\label{eq:1352}
\HH^0(\{\nu\in \S^1  : \ (x,\nu)\in B\})\le 1.
\end{equation}
Indeed, let $x \in M$. According to \eqref{eq:1418}, we always have that $\HH^0(\{\nu\in \S^1 : \ (x,\nu)\in B\})\le 2$. If $\HH^0(\{\nu\in \S^1  : \ (x,\nu)\in B\})= 2$, it implies by definition of $B$ that $M\subset x+ \langle\nu^\perp\rangle=:L$. By locality of the mean curvature (see \cite[Theorem 2]{AmbMas}), we get that $H=H_{L}=0$ $\HH^1$-a.e. on $M$ 
so that $\delta V = - H \mu = 0$ and $V$ is a stationary varifold in $\R^2$.  Applying the Constancy Theorem (see \cite[Theorem 41.1]{Simon_GMT}), we infer that $M=L$ and $\theta$ is constant on $M$, which implies that $\mu(\R^2)=\theta \HH^1(L)=\infty$. We thus reach a contradiction with \eqref{eq:assumpt} and deduce the validity of \eqref{eq:1352}.

\medskip

Finally, gathering \eqref{coareaV},  \eqref{eq:1353} and \eqref{eq:1352}, we obtain 
\begin{multline*}
 \int_{\R^2} |H| \, d\mu \ge \int_{M} |H| \, d\HH^1\ge \int_M \HH^0(\{\nu\in \S^1  : \ (x,\nu)\in B\}) |H(x)|\, d\HH^1(x)\\
 = \int_{\S^1} \HH^0(\{x\in M  : \  (x,\nu)\in B\}) \, d\HH^1(\nu)
 \ge  \HH^1(\S^1)=2\pi,
\end{multline*}
concluding the proof of \eqref{GBVarifold}.
\end{proof}

\subsection{Lusin property of the Gauss map}\label{sec:lusin}
The aim of this section is to prove  Proposition \ref{prop:lusin}. Let us start by introducing some notation. Given sets $C\subset \R^2\times \S^1$ and $A\subset \R^2$, we define
\[
   C|A=(A\times \S^1) \cap C\subset \R^2\times \S^1, \qquad   C[A]=\Pi_{\S^1}(C|A) \subset \S^1.
\]
We recall the definition \eqref{eq:B} of the set $B$.

\medskip

\noindent {\it Proof of Proposition~\ref{prop:lusin}.}
With the previous notation, our aim is to show that if $N \subset M$ is such that $\HH^1(N)=0$, then $\HH^1(B[N])=0$. By \cite[Remark 2.54]{AFP}, it is enough to prove that $\HH^1_\infty(B[N])=0$, where $\HH^1_\infty$ is the infinite pre-Hausdorff measure, see \cite[Definition 2.46]{AFP}.\\
Since $\HH^1(N)=0$, then $\mu(N)=\psi(N)=0$. By outer regularity, for each $\eps>0$, there exists an open set $U \supset N$ such that $\mu(U)+\psi(U)\le \eps$.
Let $\lambda=2^{-5}5^{-1}$,
we claim that there exists a constant $c>0$ such that the family
\begin{multline*}
 \mathcal{F}=\Big\{ B_\varrho(x) :\  x\in \Pi_{\R^2}(B|N), \ 0< \varrho \le \lambda  R_0, \ B_{\varrho/\lambda}(x)\subset U \\
 \textrm{ and } \ \HH^1_\infty(B[B_{5\varrho}(x)])\le c \left(\mu(B_\varrho(x)+\psi(B_\varrho(x)\right)\Big\}
\end{multline*}
is a covering of $\Pi_{\R^2}(B|N)$, i.e.
\begin{equation}\label{coverPiCn}
 \Pi_{\R^2}(B|N)\subset \bigcup_{B_\varrho(x)\in \mathcal{F}} B_\varrho(x).
\end{equation}
Indeed, let $x\in \Pi_{\R^2}(B|N)$ so that, in particular, $x\in N \subset U$. As $U$ is open, we have $B_{\varrho/\lambda}(x)\subset U$ for $\varrho \in (0,\lambda R_0)$ small enough. Applying Lemma \ref{lem:aux} below to the function $r \mapsto g(r)=\psi(B_r(x))$, we may possibly decrease $\varrho \in (0,\lambda R_0)$ in such a way that
\begin{equation}\label{eq:changelamr}
 \psi(B_{\varrho/\lambda}(x))\leq \frac{2}{\lambda}(  \varrho +\psi(B_\varrho(x))) \leq  \frac{2}{\lambda}( \mu(B_\varrho(x)) +\psi(B_\varrho(x))) ,
\end{equation}
where the last inequality follows from \eqref{loweruppermu}.
By Lemma \ref{lem:centralLusin} below, \eqref{loweruppermu} and \eqref{eq:changelamr}, there exists a constant $c >0$ such that
\begin{equation*}
 \HH^1_\infty(B[B_{5\varrho}(x)])=\HH^1_\infty(B[B_{2^{-5}\frac{\varrho}{\lambda}}(x)]) \leq c \psi(B_{\frac{\varrho}{\lambda}}(x)) \leq \frac{2c}{\lambda}(\mu(B_\varrho(x))+ \psi(B_\varrho(x))),
\end{equation*}
which concludes the proof of \eqref{coverPiCn}.

By Vitali's covering Theorem, there exists a countable subfamily $\mathcal{G}=\{B_{\varrho_i}(x_i)\}_{i \in \N}\subset \mathcal{F}$ such that the elements of $\mathcal{G}$ are pairwise disjoint and
$$\Pi_{\R^2}(B|N)\subset \bigcup_{i \in \N} B_{5\varrho_i}(x_i).$$
This implies that
$$B|N \subset \bigcup_{i \in \N} B|B_{5\varrho_i}(x_i),$$
and thus,
$$B[N] \subset \bigcup_{i\in \N} B[B_{5\varrho_i}(x_i)].$$
Then, we have
\[
 \HH^1_\infty(B[N])\le \sum_{i \in \N} \HH^1_\infty(B[B_{5\varrho_i}(x_i)])\leq c \sum_{i \in \N} (\mu+ \psi)(B_{\varrho_i}(x_i))\le c(\mu+\psi)(U)\le c \eps.
\]
Since $\eps$ is arbitrary, it yields $\HH^1_\infty(B[N])=0$ as claimed.

\begin{remark}
In the proof above, we avoided the analog of \cite[Lemma 21.13, Lemma 21.14, Lemma 21.15]{menne2023sharp}. Moreover, compared to \cite[Theorem 21.16]{menne2023sharp}, we have at our disposal the monotonicity formula \eqref{loweruppermu} which shows that the density $\Theta^1(\mu,\cdot)$ is uniformly bounded. In particular, it avoids a truncation argument as well as the analog of \cite[Section 22]{menne2023sharp}.
\end{remark}

We now adapt the proof of \cite[Lemma 21.12]{menne2023sharp} to show the following measure estimate used in the proof above. We recall that $c_1>0$ is the constant of the  Sobolev-Michael-Simon inequality (see Theorem \ref{thm:SMS}) and $\alpha$ is defined by \eqref{eq:alpha}.

\begin{lemma}\label{lem:centralLusin}
There exists a constant $c=c(c_1,\alpha)>0$ such that for all $x_0 \in \Pi_{\R^2}(B|N)$ and all $0<R \leq R_0$,
\begin{equation}\label{statement:centralLusin}
\HH^1_\infty(B[B_{2^{-5}R}(x_0)])\leq c  \psi(B_R(x_0)) \: .
\end{equation}
\end{lemma}
\begin{proof}
We introduce the following notation
\begin{equation} \label{eq:defGamma}
 \gamma = \psi (B_R(x_0)) =  \int_{B_R(x_0)} |H| \, d \mu \: .
\end{equation}

Using that $\HH^1_\infty(B[B_{2^{-5}R}(x_0)]) \leq \HH^{1}_\infty(\S^1)\leq 2\pi$, there is no loss of generality in assuming that $\gamma$ satisfies condition \eqref{eq:SMSineq} so that Lemma \ref{21.8} below applies. Let $\Delta>0$ be the constant given by that Lemma.

Let us consider a maximal subset $A\subset B[B_{2^{-5}R}(x_0)]\subset \S^1$ such that $|v-v^\prime|\geq 1$ for every $v$, $v^\prime \in A$ with $v\neq v^\prime$. We then have
\begin{equation} \label{coverA}
B[B_{2^{-5}R}(x_0)] \subset \bigcup_{v\in A} B_{1}(v), \qquad \#(A)\leq \kappa ,
\end{equation}
for some universal constant $\kappa>0$. We claim that for every $v\in A$, if $(\bar x,v)\in B|B_{2^{-5}R}(x_0)$ and  $(x,\nu)\in B|B_{2^{-5}R}(x_0)$ with $|v-\nu|<1$ then
\begin{equation}\label{claimsmalloscillnormals}
 |v-\nu|\le c \Delta \gamma,
 \end{equation}
for some universal constant $c>0$.

\begin{figure}[!h]
\def\svgwidth{425pt}
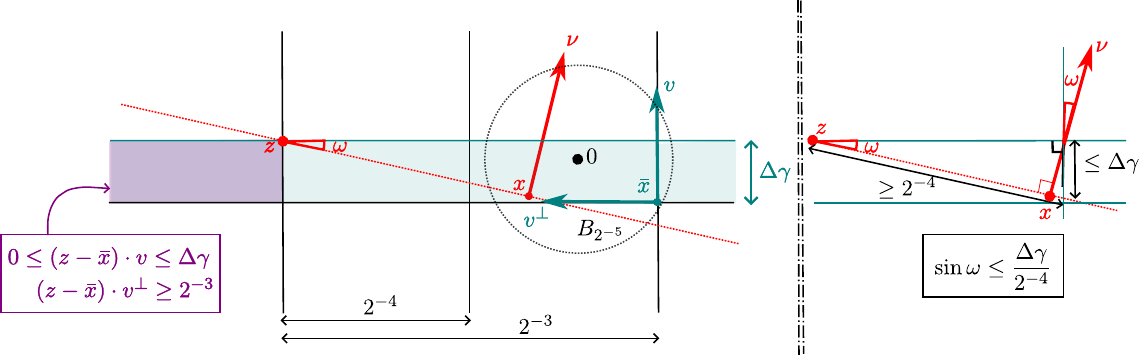
\caption{Geometric illustration of $| v^\perp \cdot \nu | \leq 2^4 \Delta \gamma$, with $x_0 = 0$ and $R = 1$}
\label{fig:v_nu}
\end{figure}

Since  $|v-\nu|<1$, we have in particular that $\nu \cdot v>0$.
Up to an axial symmetry, we may assume without loss of generality that $\nu\cdot v^\perp\leq 0$.
We first check that $M$ intersects the purple region region in Figure~\ref{fig:v_nu}. More precisely, thanks to Lemma \ref{21.8} below related to $(\bar x,v)$, we may find a point $z\in M \cap B_{R/2}(\bar x)$ such that
$$(z-\bar x)\cdot v^\perp \ge 2^{-3}R, \qquad  (z-\bar x)\cdot v \le \Delta \gamma R.$$
Moreover, by definition of $B$ and since $x$ and $z \in M$, we also have
$$  (x-\bar x)\cdot v \ge 0 \quad \text{and} \quad (z-x)\cdot \nu\ge 0 \: .$$
It remains to check that the existence of such a $z \in M$ allows one to bound the angle between $\nu$ and $v$ as suggested in Figure~\ref{fig:v_nu}, where the maximal angle $\omega$ between $\nu$ and $v$ corresponds to the case $z$ in the top right corner of the purple region and $(z-x) \cdot \nu = 0$.
From the above properties of $z$ and recalling that $x$, $\bar x \in B_{2^{-5}R}(x_0)$ so that $| \bar x - x| \leq 2^{-4}R$, we infer
\[
 (z-x) \cdot v^\perp = (z-\bar x) \cdot v^\perp + (\bar x -x) \cdot v^\perp \geq 2^{-3}R - | \bar x - x| \geq 2^{-3}R - 2^{-4}R = 2^{-4}R \: ,
\]
and
\[
 (z-x) \cdot v = (z-\bar x) \cdot v + (\bar x -x) \cdot v \leq \Delta \gamma R \: .
\]
As $\nu\cdot v^\perp \le 0$ and $0 \leq \nu \cdot v \leq 1$, we consequently have
\begin{eqnarray*}
 0 & \leq & (z-x)\cdot \nu = ((z-x)\cdot v) (v \cdot \nu) + ((z-x)\cdot v^\perp) (v^\perp \cdot \nu)\\
 &  \leq & \Delta \gamma R + 2^{-4}R (v^\perp \cdot \nu),
\end{eqnarray*}
hence
$$| v^\perp \cdot \nu | \leq 2^4 \Delta \gamma \: .$$
Using again that $\nu\cdot v \geq 0$,
\[
 |\nu-v|^2=2(1-\nu\cdot v)=2\frac{1-(\nu\cdot v)^2}{1+\nu\cdot v}=2 \frac{(\nu\cdot v^\perp)^2}{1+\nu\cdot v}\le 2(\nu\cdot v^\perp)^2\le (c\Delta \gamma)^2
\]
where $c>0$ is a universal constant, and we conclude the proof of \eqref{claimsmalloscillnormals}. As a consequence of \eqref{claimsmalloscillnormals}, we find that for every $v \in A$ we have 
$${\rm diam} (B[B_{2^{-5}R}(x_0)] \cap B_1(v))\leq 2c \Delta \gamma.$$
Recalling \eqref{coverA}, this concludes the proof of \eqref{statement:centralLusin}. 
\end{proof}

The proof of Lemma \ref{lem:centralLusin} relies on the following result, which is an adaptation of \cite[Lemma 21.8]{menne2023sharp}. The idea is to find a level set of the  function $y \mapsto (y-x)\cdot \nu$ (where $(x,\nu) \in B$) corresponding to a height of order at most $\gamma R$ which intersects a `large' portion $M$, and such that the first variation of the restricted varifold to this level set does not have too much mass. Thanks to Lemma \ref{lem:movemass} below this would imply that the same level set intersects the balls $B_{2^{-3}R}(x\pm 2^{-2} R \nu^\perp)$. Let $\eps_0=\eps_0(\alpha,2^{-4})>0$ be given by Proposition \ref{prop:Sob} with $\tau=2^{-4}$ and  $\hat c=\hat c(c_1,\alpha)>0$ be the constant appearing in \eqref{hatc} below.

\begin{lemma}\label{21.8}
There exists $\Delta=\Delta(c_1,\alpha)>0$ such that for all $x_0 \in \Pi_{\R^2}(B|N)$ and all $0<R \leq R_0$, if $\gamma > 0$ (defined in \eqref{eq:defGamma}) satisfies
 \begin{equation}\label{eq:SMSineq}
\gamma \leq \min\left\{\frac{1}{2c_1},\eps_0,\frac{\eps_0}{4\hat c}\right\},
 \end{equation}
then the following property holds: for $(x,\nu) \in B|B_{2^{-5}R}(x_0)$,  there exist $z^\pm \in M\cap B_{R/2}(x)$ such that
  \[
\pm (z^\pm-x)\cdot \nu^\perp \ge 2^{-3}R \qquad \textrm{ and } \qquad  (z^\pm -x)\cdot \nu \le \Delta \gamma R.
  \]
\end{lemma}

\begin{proof}
Let $(x,\nu)\in B|B_{2^{-5}R}(x_0)$, and assume without loss of generality that $(x,\nu)=(0,e_2)$. Let $f: B_R(x_0)  \to \R^+$ be the nonnegative Lipschitz function defined by 
 \[
  f(z)=\max\{z_2,0\}.
 \]
By definition of $B$, for all $z \in M$, $f(z)=z_2 = (z-x) \cdot \nu \geq 0$ and 
\begin{equation}\label{eq:nablaMf}
\nabla^M f(z)=\Pi_{T_z M} Df(z)=(\tau(z)\cdot e_2)\tau(z),
\end{equation}
where $\tau(z) \in T_z M \cap \S^1$.

\medskip

{\sf Step 1.} We first note that $B_{5R/6} \subset B_R(x_0)$. We now claim that
\begin{equation}\label{20.6}
 \int_{B_{4R/5}} \frac{1}{f +\gamma R} |\nabla^M f| \,d\mu \leq \bar{c},
\end{equation}
for some constant $\bar{c}= \bar{c}(\alpha,c_1)>0$. Indeed,
adapting \cite[Lemma 20.6]{menne2023sharp}, we consider a cut-off function $\eta \in \mathcal C_c^\infty(\R^2;[0,1])$ such that $\eta=1 $ in $B_{4R/5}$, ${\rm Supp}(\eta) \subset B_{5R/6}$ and $|\nabla \eta| \leq c/R$.
Since $f + \gamma R> 0$, we can define 
\[
 u=\frac{1}{f+\gamma R}, \quad \phi=\eta^2 u, \quad g=\log(f+\gamma R) \quad \text{ in }B_R(x_0).
\]
These functions are Lipschitz continuous in $B_R(x_0)$ and we have
\[
 \nabla^M \phi= 2\eta u\nabla^M \eta + \eta^2\nabla^M u, \quad \nabla^M u=- \frac{1}{(f+\gamma R)^2}\nabla^M f, \quad \nabla^M g=\frac{1}{f+\gamma R}\nabla^M f.
\]
In particular, $|\nabla^M g|^2=-\nabla^M u\cdot \nabla^M f$ so that $\nabla^M \phi\cdot \nabla^M f=2\eta u\nabla^M \eta\cdot \nabla^M f -\eta^2 |\nabla^M g|^2$. Therefore, rearranging terms, we find
\begin{equation}\label{eq:1538}
 \int_{B_{5R/6}} \eta^2 |\nabla^M g|^2 \, d\mu= \int_{B_{5R/6}} 2 \eta u \nabla^M \eta\cdot \nabla^M f \, d\mu - \int_{B_{5R/6}} \nabla^M\phi\cdot \nabla^M f \, d\mu.
\end{equation}
Regarding the first term at the right-hand side of \eqref{eq:1538}, we estimate it using that $u|\nabla^M f|\le |\nabla^M g|$, $|\nabla^M \eta|\leq c/R$ and $\mu(B_{5R/6})\leq \alpha R$ (thanks to  \eqref{loweruppermu} together with $x=0 \in M$). Thus, Young's inequality shows that
\begin{equation}\label{eq:15382}
\int_{B_{5R/6}} 2 \eta u \nabla^M \eta\cdot \nabla^M f \, d\mu\leq \frac{c}{R}+\frac{1}{4} \int_{B_{5R/6}} \eta^2 |\nabla^M g|^2 \, d\mu,
\end{equation}
where $c=c(\alpha)>0$ only depends on $\alpha$. For the second term at the right-hand side of \eqref{eq:1538}, we  recall \eqref{eq:nablaMf} and using that $\mu$ is supported in $M$,
\begin{multline}\label{eq:15383}
 - \int_{B_{5R/6}} \nabla^M\phi\cdot \nabla^M f\,  d\mu=-\int_{B_{5R/6}} (\nabla \phi\cdot \tau) (\tau \cdot e_2) \, d\mu\\
 =-\int_{B_{5R/6}} {\rm div}^M (\phi e_2 ) \, d\mu=\int_{B_{5R/6}} H\cdot e_2 \phi \, d\mu \le \gamma\|\varphi\|_{L^\infty_\mu(B_{5R/6})}.
\end{multline}
Next applying the Sobolev-Michael-Simon inequality, Theorem \ref{thm:SMS}, to the Lipschitz function $\phi$ which is compactly supported in $B_{5R/6}$, we infer 
$$ \|\varphi\|_{L^\infty_\mu(B_{5R/6})}  \leq 2(1-c_1\gamma) \|\varphi\|_{L^\infty_\mu(B_{5R/6})} \leq 2 c_1 \int_{B_{5R/6}} |\nabla^M\phi|\, d\mu,$$
where we used that $\gamma \leq 1/(2c_1)$ thanks to assumption \eqref{eq:SMSineq}. Using that $\gamma R u\le 1$ and  $\gamma R|\nabla^M u|\le |\nabla^M g|$, we find $ \gamma R |\nabla^M \phi|\le  2 \eta |\nabla^M \eta|+\eta^2 |\nabla^M g|$. Hence, since $\eta\le 1$, $|\nabla^M \eta|\leq c/R$ and $2 c_1 \gamma \le 1$, we find again thanks to Young's inequality and $\mu(B_{5R/6})\leq \alpha R$,
\begin{equation}\label{eq:15384}
2c_1  \gamma R  \int_{B_{5R/6}} |\nabla^M\phi|\, d\mu \leq c+ \int_{B_{5R/6}} \eta |\nabla^M g| \, d\mu\leq c +\frac{R}{4} \int_{B_{5R/6}} \eta^2 |\nabla^M g|^2 \, d\mu,
\end{equation}
where $c=c(c_1,\alpha)>0$. Collecting estimates \eqref{eq:1538}, \eqref{eq:15382}, \eqref{eq:15383} and \eqref{eq:15384}, we deduce that
\begin{equation*}
\int_{B_{4R/5}} |\nabla^M g|^2 \, d\mu\le \int_{B_{5R/6}} \eta^2 |\nabla^M g|^2 \, d\mu\leq  \frac{c}{R},
\end{equation*}
where $c=c(c_1,\alpha)>0$. Using H\"older inequality and once again $\mu(B_{5R/6})\leq \alpha R$, this proves
\begin{equation}\label{20.6bis}
 \int_{B_{4R/5}}|\nabla^M g|\, d\mu \leq  \bar c,
\end{equation}
for some constant $\bar c=\bar c(c_1,\alpha)>0$, and thus also \eqref{20.6}.

\medskip

{\sf Step 2.} Let $E_s=B_R(x_0) \cap M \cap \{f<s\}$ and let $\eps_0>0$ given by Proposition \ref{prop:Sob} (with $\tau=2^{-4}$) below in \eqref{eq:epsilon0} so that in particular, $ \varepsilon_0 R<  2^{-4}R \leq\mu (B_{2^{-4}R})$ by \eqref{loweruppermu}.
We then define
\[
 t_0=\sup\lt\{ s>0  : \ \mu(E_s\cap B_{2^{-4}R})\le \eps_0 R\rt\} < \infty \: ,
\]
so that for all $t > t_0$,
\begin{equation} \label{massEt}
 \mu(E_{t_0} \cap B_{2^{-4}R})\le \eps_0 R \le \mu(E_{t}\cap B_{2^{-4}R}).
\end{equation} 
As in \cite[Theorem 20.9]{menne2023sharp}, we prove that
\begin{equation}\label{estimtHarnack}
 t_0 \leq \underline{c}\gamma R,
\end{equation}
for some constant $\underline{c}=\underline{c}(c_1,\alpha)>0$.
To this aim, we consider the Lipschitz function $\hat g:B_R(x_0)\to \R^+$ defined by
$$\hat{g}=\max\{\log(t_0+\gamma R)-\log(f+\gamma R),0\}=(\log(t_0+\gamma R)-g){\bf 1}_{E_{t_0}} \geq 0$$ 
so that $|\nabla^M \hat{g}|=|\nabla^M g|$ in $E_{t_0}$ and $\nabla \hat g=0$ in $B_R(x_0) \setminus E_{t_0}$. Since $B_{2^{-4}R} \subset B_R(x_0)$, by \eqref{eq:SMSineq} together with \eqref{massEt}, we have
$$\mu(E_{t_0} \cap  B_{2^{-4}R}) \leq \eps_0 R, \quad \int_{E_{t_0} \cap B_{2^{-4}R}} |H|\, d\mu \leq \eps_0.$$
We can thus apply Proposition \ref{prop:Sob} to the function $\hat g$, the set $W=E_{t_0} \cap B_{2^{-4}R}$ and $\tau=2^{-4}$, and get that
\[
\sup_{B_{2^{-5}R}\cap M} \hat{g} \leq \Lambda  \int_{B_{2^{-4}R}} |\nabla^M \hat{g}| \, d\mu.
\]
On the other hand, by \eqref{20.6bis},
\[
 \int_{B_{2^{-4}R}} |\nabla^M \hat{g}|\, d\mu\le \bar c.
\]
In particular we have $\hat{g} \leq \Lambda \bar c$ on $M\cap E_{t_0}\cap B_{2^{-5}R}$ i.e.
\[
 \log (f+\gamma R)\ge \log(t_0+\gamma R)-\Lambda \bar c \quad \text{ on }M\cap E_{t_0}\cap B_{2^{-5}R},
\]
and passing to the exponential yields
\[
 f+\gamma R\ge e^{-\Lambda \bar c}(t_0+\gamma R)\ge  e^{-\Lambda \bar c} t_0\quad \text{ on }M\cap E_{t_0}\cap B_{2^{-5}R}.
\]
Finally, since $f \geq 0$ in $M$, $0 \in M$ and $f(0) = 0$, then $\inf_{M\cap E_{t_0}\cap B_{2^{-5}R}} f= 0$ and 
\[
\gamma R \geq  \inf_{M\cap E_{t_0} \cap B_{2^{-5}R}} (f +\gamma R)\ge t_0 e^{-\Lambda \bar c}.
\]
This concludes the proof of \eqref{estimtHarnack} with $\underline{c}=e^{\Lambda \bar c}$.

\medskip

{\sf Step 3.} We now set $s_1=\log(t_0+\gamma R)$ and $s_2=\log(t_0+\gamma R)+\frac{\bar c}{\eps_0}$ with $\bar c>0$ given by \eqref{20.6bis}. 
For $s \geq s_1$, we introduce
\begin{equation*}
\hat{E}_s={B_R(x_0)\cap M}\cap \{g<s\}=E_{e^s-\gamma} \:, \quad V_s=V \restr (\hat E_s \times \mathbf G_1) \quad \text{and} \quad \mu_s=\mu_{V_s}= \mu \restr \hat E_s \: ,
\end{equation*}
so that  $V_s$ is a $1$-rectifiable integral varifold in $B_1$.
In particular, we have $\hat{E}_{s_1}=E_{t_0}$ and
for every $s\in(s_1,s_2)$, using \eqref{massEt} with $t = e^s - \gamma R> e^{s_1} - \gamma R= t_0$, we obtain 
\begin{equation}\label{massmus}
 \mu_s(B_{2^{-4}R}) = \mu (E_t \cap B_{2^{-4}R}) \ge \eps_0 R .
\end{equation}
Regarding the first variation $\delta V_s$, $g$ is smooth and we can apply the slicing Theorem for varifolds, \cite[Theorem 2, Section 4.10]{allardfirst}, together with \eqref{20.6bis} to infer that $\delta V_s$ is a Radon measure and moreover,
\begin{multline*}
 \int_{s_1}^{s_2} |\delta V_s|(B_{4R/5}) \, ds\le \int_{s_1}^{s_2} |\delta V |(B_{4R/5} \cap \hat E_s) \,ds +\int_{B_{4R/5}} |\nabla^M g| \, d\mu\\
 \le (s_2-s_1) |\delta V|(B_{4R/5})+ \int_{B_{4R/5}} |\nabla^M g| \,d\mu\leq  (s_2-s_1) \gamma + \bar c,
\end{multline*}
where we used that $B_{4R/5} \subset B_R(x_0)$ and the definition \eqref{eq:defGamma} of $\gamma$ in the last inequality. We recall that $s_2 - s_1 = \frac{\bar c}{\varepsilon_0}$. Therefore, by the mean value formula there exists $s_\ast \in(s_1,s_2)$ such that
\begin{equation}\label{estimdeltaVs}
 |\delta V_{s_\ast}|(B_{4R/5})\le \frac{1}{s_2-s_1}  \int_{s_1}^{s_2} |\delta V_s|(B_{4R/5})\,  ds\leq \gamma + \eps_0.
\end{equation}
According to \eqref{estimtHarnack}, we have 
\begin{equation}\label{eq:fbound}
f+\gamma R \leq e^{s_\ast} \leq e^{s_2} = (t_0+\gamma R)e^{\frac{\bar c}{\eps_0}} \leq \Delta \gamma R \quad \text{ in }\hat{E}_{s_\ast},
\end{equation}
where $\Delta=(\underline{c}+1)e^{\frac{\bar c}{\eps_0}}$ only depends on $\alpha$ and $c_1$, so that by \eqref{20.6},
\begin{equation}\label{estimnabfEs}
\int_{\hat{E}_{s_\ast}\cap B_{4R/5}} |\nabla^M f|\, d\mu\leq \Delta \gamma R \int_{\hat{E}_{s_\ast}\cap B_{4R/5}} \frac{1}{f+\gamma R} |\nabla^M f |\,d\mu\leq \overline{c} \Delta\gamma R.
\end{equation}
Notice that from \eqref{eq:nablaMf}, for $x \in M$ and $\tau \in T_x M \cap \S^1$, we have
$$|\Pi_{T_xM} -\Pi_{\R e_1}|^2=2 (\tau\cdot e_2)^2= |\nabla^M f(x)|^2,$$
and thus \eqref{estimnabfEs} yields
\begin{equation} \label{eq:gradfBound}
 \int_{B_{4R/5}}|\Pi_{T_xM}-\Pi_{\R e_1}| \, d\mu_{s_\ast}= \int_{B_{4R/5}}|\nabla^M f|\, d\mu_{s_\ast} \leq \overline{c} \Delta\gamma R.
\end{equation}
Applying Lemma \ref{lem:movemass} with $R_1=2^{-4}R$, $R_2=2^{-3}R$ and $L=2^{-2}R$, and noticing that
$$K=\bigcup_{t\in(0,2^{-2}R)} \overline{B_{2^{-3}R}(te_1)} \subset B_{4R/5},$$
we get using  \eqref{eq:SMSineq}, \eqref{massmus}, \eqref{estimdeltaVs} and \eqref{eq:gradfBound},
\begin{multline}\label{hatc}
 \mu_{s_\ast}(B_{2^{-3}R}(2^{-2}R e_1))\geq \mu_{s_\ast}(B_{2^{-4}R}) -\frac{R}{4} \lt(|\delta V_{s_\ast}|(B_{4R/5})+\frac{2^5}{R}\int_{B_{4R/5}}|\Pi_{T_x M}-\Pi_{\R e_1}| \, d\mu_{s_\ast} \rt)\\
\geq  \eps_0 R-\frac{R}{4}(\gamma+\eps_0)- 8\overline{c}\Delta \gamma R \geq \frac{\eps_0}{2}R -\hat c\gamma R,
\end{multline}
for some constant $\hat c=\hat c(c_1,\alpha)>0$. According to our choice of $\gamma$ in \eqref{eq:SMSineq}, we have $\gamma \leq  \frac{\eps_0}{4\hat c}$ so that \[
 \mu_{s_\ast}\lt(B_{2^{-3}R}(2^{-2} Re_1)\rt)\ge \frac{\eps_0}{4}R>0.
\]
Since by \eqref{eq:fbound}, $\hat E_{s_\ast} \subset \{ g < s_\ast \} \subset \{ f \leq \Delta \gamma R\}$
we infer that
\[
\mu \lt( \{f \leq \Delta \gamma R\} \cap B_{2^{-3}R}(2^{-2} Re_1) \rt) \geq \mu \lt(\hat E_{s_\ast} \cap B_{2^{-3}R}(2^{-2} Re_1) \rt) = \mu_{s_\ast}\lt(B_{2^{-3}R}(2^{-2} Re_1)\rt)  >0
\]
and eventually, any element $z\in  \{f \leq \Delta \gamma R\} \cap B_{2^{-3}R}(2^{-2} R e_1) \cap {\rm Supp}(\mu)\neq \emptyset$ satisfies 
$$z \in M \cap B_{R/2}, \quad z_1 \geq 2^{-3} R\quad \text{and} \quad z_2 = f(z) \leq \Delta \gamma R,$$
which completes the proof of the lemma. 
\end{proof}

\subsection{Some technical results}

The following result is a simplified variant of \cite[Lemma 21.1]{menne2023sharp}.
\begin{lemma}\label{lem:aux}
 Let $g:\R\to \R^+$. Then, for every $t>0$ and every $\lambda\in(0,1)$, there exists $0<r\le t$ such that
\[
 g(r/\lambda)\le 2 \lambda^{-1}(r+g(r)).
\]
\end{lemma}
\begin{proof}
Set $f(r)=r+g(r)$ for $r>0$. Assume that the conclusion does not hold for some $t>0$ and $\lambda \in (0,1)$ so that for all $r \in (0,t]$, 
\[
 f(r) =r + g(r) < \frac{\lambda}{2} g(r/\lambda) \: .
\]
For each $n\ge 0$, we define $r_n=\lambda^n t \leq t$ 1so that,
 \[
  f(r_{n+1})< \frac{\lambda}{2} g(r_n) \leq  \frac{\lambda}{2} f(r_n),
 \] 
 and thus, by induction, 
\[f(r_n)  \le \lt(\frac{\lambda}{2}\rt)^{n} f(t)= \frac{r_n}{t 2^n} f(t).
 \]
Let $n$ be large enough so that $f(t)2^{-n} < \lambda t$, we get a contradiction with the fact that $f(r_{n})\ge r_{n}$.
\end{proof}

The following result states a Poincar\'e-Sobolev  type inequality on varifolds, for Lipschitz functions vanishing on a large part of the domain. Notice that it also follows from \cite[Theorem 9.1]{menne2023sharp} (or \cite[Theorem 10.1]{menneIndiana}) which we partially reproved in our situation for the reader's convenience.

\begin{proposition}\label{prop:Sob}
For every $\tau\in (0,1)$, there exists $\eps_0=\eps_0(\alpha,\tau)>0$ such that the following property holds:  let $x \in M$, $r \in (0,R_0]$, $g:B_r(x) \to \R^+$ 
and $W \subset B_{\tau r}(x)$ be a Borel set. Suppose that 
\begin{equation*}
g=0 \text{ on }M \cap B_{\tau r}(x) \setminus W,\quad  \mu(W) \leq \eps_0 r, \quad \int_W |H|\, d\mu \leq \eps_0.
\end{equation*}
Then 
\begin{equation*}
\sup_{M \cap B_{\frac{\tau r}{2}}(x)} g \leq \Lambda \int_{B_{\tau r}(x)}|\nabla^M g|\, d\mu,
\end{equation*}
for some constant $\Lambda>0$ only depending on $\alpha$.
\end{proposition}

\begin{proof}
Set $Q=1$, $M=2$ and let $\Lambda=\Lambda(2)>0$ be the constant given by \cite[Corollary 7.11]{menneIndiana}. Fix
\begin{equation}\label{eq:epsilon0}
\eps_0 < \min\left\{\frac{\tau}{2}, \frac{\Lambda^{-1} \tau}{2},\Lambda^{-1}\right\} .
\end{equation}

We define the (vector) Radon measure $T\in \mathcal M(B_r(x) \times \R;\R^2)$ by
\begin{equation*}
 \langle T,\zeta \rangle
 =\int_{B_{r}(x)} \zeta(y,g(y))\cdot \nabla^Mg(y)\, d\mu(y)\quad \text{ for all }\zeta \in \mathcal C_c(B_{r} (x)\times \R;\R^2).
\end{equation*}
We then have that
\begin{equation}\label{eq:T}
|T|(B_{r}(x) \times \R)= |T|(B_{r}(x) \times \R_+) =\int_{B_{r}(x)}|\nabla^Mg|\, d\mu .
\end{equation}

Next, for $s\in \R_+$, we introduce the Borel set $F_s=\{g>s\} \subset W$ and, as in \cite[Definition 5.1]{menneIndiana}, we define the distributional $V$--boundary of $F_s$ as
$$V \partial F_s:=\delta V \restr F_s - \delta (V\restr (F_s \times \mathbf G_1)) \in \mathcal D'(B_{r}(x);\R^2),$$
i.e. 
\[
V \partial F_s(\zeta)=\int_{F_s} H \cdot \zeta\, d\mu  -\int_{F_s} {\rm div}^M \zeta \, d\mu \quad \text{ for }\zeta \in \mathcal C^\infty_c(B_{r}(x);\R^2).
\]
Then, by \cite[Theorem 8.30]{menneIndiana} (see also \cite[Theorem 2, Section 4.10]{allardfirst}) and \eqref{eq:T}, $V \partial F_s$ is a Radon measure for a.e. $s \in \R_+$ and the following coarea formula holds:
\begin{equation} \label{eq:coareaGradg}
 \int_{B_{\tau r}(x)} |\nabla^M g|\: d\mu=|T|(B_{\tau r}(x) \times \R_+) =\int_{0}^{+\infty} |V \partial F_s|( B_{\tau r}(x)) \: ds \: .
\end{equation}
Let $s \in \R_+$ such that $V \partial F_s$ is a Radon measure and furthermore assume that $s < \sup_{M \cap B_{\frac{\tau r}{2}}(x)} g$. Then, $M \cap F_s \cap B_{\frac{\tau r}{2}}(x) \neq \emptyset$ and by continuity of $g$, there exists a small ball $B_\varrho (z) \subset F_s \cap B_{\frac{\tau r}{2}}(x)$ with $z \in M \cap F_s \cap B_{\frac{\tau r}{2}}(x)$ and $\varrho\leq R_0$. By \eqref{loweruppermu}, $\mu(B_\varrho(z)) \geq \varrho > 0$ and consequently
\begin{equation} \label{eq:muMassNonZero}
 \mu (F_s \cap B_{\frac{\tau r}{2}}(x)) > 0 \: .
\end{equation}
We now apply the relative isoperimetric inequality, \cite[Corollary 7.11]{menneIndiana}, with $U= B_{\tau r}(x)$, $B=\partial B_{\tau r}(x)$, $E = F_s$, $Q=1$ and $M=2$. 
By our choice of $\eps_0$ in \eqref{eq:epsilon0}, as $F_s \subset W$ then in particular
\[
 \mu (F_s \cap B_{\tau r}(x)) \leq \varepsilon_0 r \leq \min \left\lbrace \frac{\tau r}{2} , \Lambda^{-1} \frac{\tau r}{2} \right\rbrace \quad \text{and} \quad \int_{F_s} |H| \: d \mu \leq \varepsilon_0  \leq \Lambda^{-1}
\]
so that the assumptions of \cite[Corollary 7.11]{menneIndiana} are met and we get thanks to \eqref{eq:muMassNonZero}
\[
\Lambda |V \partial F_s| (B_{\tau r}(x))\geq \mu (F_s \cap B_{\frac{\tau r}{2}}(x) )^0 = 1.
\]
Therefore, integrating over $s$ from $0$ to $\sup_{M \cap B_{\frac{\tau r}{2}}(x)}g$ and applying \eqref{eq:coareaGradg}, we obtain
\[
\Lambda \int_{B_{\tau r}(x)} |\nabla^M g| \, d\mu\geq  \sup_{M \cap B_{\frac{\tau r}{2}}(x)}g,
\]
and the proof is complete.
\end{proof}

The next result is a specific version of \cite[Lemma 21.6]{menne2023sharp}. 

\begin{lemma}\label{lem:movemass}
Let $W = v(M,\theta) \in \mathbf V_1(\R^2)$ be a $1$-rectifiable varifold with bounded first variation.
For every $R_2>R_1>0$ and every $L>0$, let 
\[
 K=\bigcup_{t\in(0,L)} \overline{B_{R_2}(te_1)}.
\]
Then,
$$ \mu_W(B_{R_2}(Le_1))- \mu_W(B_{R_1})\geq  - L\left(|\delta W|(K)+ \frac{2}{R_2-R_1}\int_{K} |\Pi_{\R e_1} -\Pi_{T_x M}| \,d\mu_W\right),$$
where $\Pi_{T_x M}$ denotes the projection onto the line $T_x M$.
\end{lemma}
\begin{proof}
 Let $\chi\in \mathcal C_c^\infty(\R^2;[0,1])$ be a cut-off function with ${\bf 1}_{B_{R_1}}\le \chi \le {\bf 1}_{B_{R_2}}$, ${\rm Supp}(\chi) \subset B_{R_2}$ and $|\nabla \chi|\leq 2/r$ with $r=R_2-R_1$. For $t \in (0,L)$, we define
 \[
  \phi(t)=\int_{\R^2} \chi(x-te_1)\, d\mu_W(x).
 \]
We compute for $t \in (0,L)$
\[
 \phi'(t)=-\int_{\R^2} \nabla \chi(x-t e_1)\cdot e_1 \,d\mu_W(x).
\]

For $x\in M$, let $\tau \in \S^1 \cap  T_x M$ and $\nu\in \S^1$ such that $\nu\cdot \tau=0$, then
\[
 -\phi'(t)=\int_{\R^2} (\nabla \chi(x-t e_1)\cdot \tau)(\tau\cdot e_1) \, d\mu_W(x)+\int_{\R^2} (\nabla \chi(x-t e_1)\cdot \nu)(e_1\cdot \nu) \, d\mu_W(x).
\]
If $\xi= \chi(\cdot -te_1) e_1$ we get that for $x \in \R^2$,
\[{\rm div}^{M} \xi(x)  = (D\xi(x) \tau)\cdot \tau=-(\nabla \chi(x-te_1) \cdot \tau)(\tau\cdot e_1)\]
so that, by definition of the first variation of $V$, 
$$ \lt|\int_{\R^2} (\nabla \chi(x-t e_1)\cdot \tau)(\tau\cdot e_1) \,d\mu_W(x) \rt|=\left| \int_{\R^2} {\rm div}^{M} \xi\, d\mu_W\right|
 \leq |\delta W|(K),$$
where we used that ${\rm Supp}(\xi) \subset B_{R_2}(te_1) \subset K$ and $|\xi| \leq 1$.

For the second term we notice that since $e_1 = \Pi_{\R e_1} (e_1)$ we have
\begin{equation*}
 (e_1 \cdot \nu) \nu = e_1 - \Pi_{T_x M}(e_1) = ( \Pi_{\R e_1} - \Pi_{T_x M})(e_1)
\end{equation*}
and consequently
\begin{multline*}
 \lt|\int_{\R^2} (\nabla \chi(x-t e_1)\cdot \nu)(e_1\cdot \nu) d\mu_W(x) \rt|\le \frac{2}{r} \int_K |(e_1 \cdot \nu) \nu |\, d\mu_W \leq \frac{2}{r}  \int_{K} |\Pi_{\R e_1} -\Pi_{T_x M}| \, d\mu_W.
\end{multline*}
In conclusion, we obtain that for $t \in (0,L)$,
\[
 |\phi'(t)|\leq  |\delta W|(K)+ \frac{2}{r} \int_{K} |\Pi_{\R e_1} -\Pi_{T_x M}|\, d\mu_W.
\]
Since $\phi(0)\ge \mu_W(B_{R_1})$ and $\phi(L)\le \mu_W(B_{R_2}(Le_1))$, the conclusion follows.
\end{proof}

\section*{Acknowledgments}
This research has been partially supported by the ANR STOIQUES and GeMfaceT. This work was supported by a public grant from the Fondation Mathématique Jacques Hadamard.

\bibliographystyle{acm}
\bibliography{Anisotropic}

\begin{thebibliography}{10}

\bibitem{allardfirst}
{\sc Allard, W.~K.}
\newblock On the first variation of a varifold.
\newblock {\em Annals of mathematics 95}, 3 (1972), 417--491.

\bibitem{AllardAlmgren}
{\sc Allard, W.~K., and Almgren, Jr., F.~J.}
\newblock The structure of stationary one dimensional varifolds with positive
  density.
\newblock {\em Invent. Math. 34}, 2 (1976), 83--97.

\bibitem{AFP}
{\sc Ambrosio, L., Fusco, N., and Pallara, D.}
\newblock {\em Functions of bounded variation and free discontinuity problems}.
\newblock Oxford Mathematical Monographs. The Clarendon Press, Oxford
  University Press, New York, 2000.

\bibitem{AmbMas}
{\sc Ambrosio, L., and Masnou, S.}
\newblock A direct variational approach to a problem arising in image
  reconstruction.
\newblock {\em Interfaces and Free Boundaries 5}, 1 (2003), 63--81.

\bibitem{Ambrosio-Tortorelli}
{\sc Ambrosio, L., and Tortorelli, V.~M.}
\newblock On the approximation of free discontinuity problems.
\newblock {\em Boll. Un. Mat. Ital. B (7) 6}, 1 (1992), 105--123.

\bibitem{BM}
{\sc Babadjian, J.-F., and Millot, V.}
\newblock Unilateral gradient flow of the {A}mbrosio-{T}ortorelli functional by
  minimizing movements.
\newblock {\em Ann. Inst. H. Poincar\'{e} C Anal. Non Lin\'{e}aire 31}, 4
  (2014), 779--822.

\bibitem{BellMugnai}
{\sc Bellettini, G., and Mugnai, L.}
\newblock Characterization and representation of the lower semicontinuous
  envelope of the elastica functional.
\newblock {\em Ann. Inst. H. Poincar\'{e} C Anal. Non Lin\'{e}aire 21}, 6
  (2004), 839--880.

\bibitem{BellMugnai2}
{\sc Bellettini, G., and Mugnai, L.}
\newblock On the approximation of the elastica functional in radial symmetry.
\newblock {\em Calc. Var. Partial Differential Equations 24}, 1 (2005), 1--20.

\bibitem{BFM}
{\sc Bourdin, B., Francfort, G.~A., and Marigo, J.-J.}
\newblock {\em The variational approach to fracture}.
\newblock Springer, New York, 2008.
\newblock Reprinted from J. Elasticity {\bf 91} (2008), no. 1-3 [MR2390547],
  With a foreword by Roger Fosdick.

\bibitem{BraiChamSol}
{\sc Braides, A., Chambolle, A., and Solci, M.}
\newblock A relaxation result for energies defined on pairs set-function and
  applications.
\newblock {\em ESAIM: Control, Optimisation and Calculus of Variations 13}, 4
  (2007), 717--734.

\bibitem{BraiMal}
{\sc Braides, A., and Malchiodi, A.}
\newblock Curvature theory of boundary phases: the two-dimensional case.
\newblock {\em Interfaces Free Bound. 4}, 4 (2002), 345--370.

\bibitem{BraiMar}
{\sc Braides, A., and March, R.}
\newblock Approximation by ${\Gamma}$-convergence of a curvature-depending
  functional in visual reconstruction.
\newblock {\em Communications on pure and applied mathematics 59}, 1 (2006),
  71--121.

\bibitem{BreChamMa}
{\sc Bretin, E., Chambolle, A., and Masnou, S.}
\newblock A {C}ahn--{H}illiard--{W}illmore phase field model for non-oriented
  interfaces.
\newblock {\em arXiv preprint arXiv:2412.15926\/} (2024).

\bibitem{bretin2015phase}
{\sc Bretin, E., Masnou, S., and Oudet, E.}
\newblock Phase-field approximations of the willmore functional and flow.
\newblock {\em Numerische Mathematik 131\/} (2015), 115--171.

\bibitem{Coscia}
{\sc Coscia, A.}
\newblock On curvature sensitive image segmentation.
\newblock {\em Nonlinear Anal. 39}, 6 (2000), 711--730.

\bibitem{DoMuRo}
{\sc Dondl, P.~W., Mugnai, L., and R{\"o}ger, M.}
\newblock Confined elastic curves.
\newblock {\em SIAM Journal on Applied Mathematics 71}, 6 (2011), 2205--2226.

\bibitem{DoMuMa2}
{\sc Dondl, P.~W., Mugnai, L., and R\"{o}ger, M.}
\newblock A phase field model for the optimization of the {W}illmore energy in
  the class of connected surfaces.
\newblock {\em SIAM J. Math. Anal. 46}, 2 (2014), 1610--1632.

\bibitem{Foc}
{\sc Focardi, M.}
\newblock On the variational approximation of free-discontinuity problems in
  the vectorial case.
\newblock {\em Math. Models Methods Appl. Sci. 11}, 4 (2001), 663--684.

\bibitem{fonseca2015motion}
{\sc Fonseca, I., Fusco, N., Leoni, G., and Morini, M.}
\newblock Motion of three-dimensional elastic films by anisotropic surface
  diffusion with curvature regularization.
\newblock {\em Analysis \& PDE 8}, 2 (2015), 373--423.

\bibitem{FranMar}
{\sc Francfort, G.~A., and Marigo, J.-J.}
\newblock Revisiting brittle fracture as an energy minimization problem.
\newblock {\em J. Mech. Phys. Solids 46}, 8 (1998), 1319--1342.

\bibitem{Giusti}
{\sc Giusti, E.}
\newblock {\em Minimal surfaces and functions of bounded variation}, vol.~80 of
  {\em Monographs in Mathematics}.
\newblock Birkh\"auser Verlag, Basel, 1984.

\bibitem{GNRWill}
{\sc Goldman, M., Novaga, M., and Ruffini, B.}
\newblock A charged liquid drop model with {W}illmore energy.
\newblock {\em arXiv preprint arXiv:2409.01045\/} (2024).

\bibitem{GurtJabb}
{\sc Gurtin, M.~E., and Jabbour, M.}
\newblock Interface evolution in three dimensions with curvature-dependent
  energy and surface diffusion: Interface-controlled evolution, phase
  transitions, epitaxial growth of elastic films.
\newblock {\em Arch. Rational Mech. Anal. 163\/} (2002), 171--208.

\bibitem{HutTone}
{\sc Hutchinson, J.~E., and Tonegawa, Y.}
\newblock Convergence of phase interfaces in the van der waals-cahn-hilliard
  theory.
\newblock {\em Calculus of Variations and Partial Differential Equations 10\/}
  (2000), 49--84.

\bibitem{kuwert2004removability}
{\sc Kuwert, E., and Sch{\"a}tzle, R.}
\newblock Removability of point singularities of {W}illmore surfaces.
\newblock {\em Annals of Mathematics\/} (2004), 315--357.

\bibitem{leonardi2009locality}
{\sc Leonardi, G.~P., and Masnou, S.}
\newblock Locality of the mean curvature of rectifiable varifolds.

\bibitem{BinMau}
{\sc Li, B., and Maurini, C.}
\newblock Crack kinking in a variational phase-field model of brittle fracture
  with strongly anisotropic surface energy.
\newblock {\em J. Mech. Phys. Solids 125\/} (2019), 502--522.

\bibitem{Lussardi}
{\sc Lussardi, L.}
\newblock A note on a phase-field model for anisotropic systems.
\newblock {\em Asymptotic Analysis 94}, 3-4 (2015), 241--254.

\bibitem{Menne}
{\sc Menne, U.}
\newblock Second order rectifiability of integral varifolds of locally bounded
  first variation.
\newblock {\em J. Geom. Anal. 23}, 2 (2013), 709--763.

\bibitem{menne2023sharp}
{\sc Menne, U.}
\newblock A sharp lower bound on the mean curvature integral with critical
  power for integral varifold.
\newblock {\em arXiv preprint arXiv:2310.01754\/} (2014).

\bibitem{menneIndiana}
{\sc Menne, U.}
\newblock Weakly differentiable functions on varifolds.
\newblock {\em Indiana University Mathematics Journal\/} (2016), 977--1088.

\bibitem{merlet2015highly}
{\sc Merlet, B.}
\newblock A highly anisotropic nonlinear elasticity model for vesicles ii:
  Derivation of the thin bilayer bending theory.
\newblock {\em Archive for Rational Mechanics and Analysis 217}, 2 (2015),
  681--740.

\bibitem{ModMort}
{\sc Modica, L., and Mortola, S.}
\newblock Il limite nella {$\Gamma $}-convergenza di una famiglia di funzionali
  ellittici.
\newblock {\em Boll. Un. Mat. Ital. A (5) 14}, 3 (1977), 526--529.

\bibitem{moser2012towards}
{\sc Moser, R.}
\newblock Towards a variational theory of phase transitions involving
  curvature.
\newblock {\em Proceedings of the Royal Society of Edinburgh Section A:
  Mathematics 142}, 4 (2012), 839--865.

\bibitem{moser2015singular}
{\sc Moser, R.}
\newblock Singular perturbation problems involving curvature.
\newblock In {\em Differential Geometry and Continuum Mechanics\/} (2015),
  Springer, pp.~49--75.

\bibitem{muller2023li}
{\sc M{\"u}ller, M., and Rupp, F.}
\newblock A {L}i--{Y}au inequality for the 1-dimensional {W}illmore energy.
\newblock {\em Advances in Calculus of Variations 16}, 2 (2023), 337--362.

\bibitem{mumford1989optimal}
{\sc Mumford, D.~B., and Shah, J.}
\newblock Optimal approximations by piecewise smooth functions and associated
  variational problems.
\newblock {\em Communications on pure and applied mathematics\/} (1989).

\bibitem{Nagase-Tonegawa}
{\sc Nagase, Y., and Tonegawa, Y.}
\newblock A singular perturbation problem with integral curvature bound.
\newblock {\em Hiroshima Math. J. 37}, 3 (2007), 455--489.

\bibitem{philippe2022regularized}
{\sc Philippe, T., Henry, H., and Plapp, M.}
\newblock Regularized anisotropic motion-by-curvature in phase-field theory:
  Interface phase separation of crystal surfaces.
\newblock {\em Physical Review E 106}, 3 (2022), 034119.

\bibitem{pozzetta}
{\sc Pozzetta, M.}
\newblock A varifold perspective on the {$p$}-elastic energy of planar sets.
\newblock {\em J. Convex Anal. 27}, 3 (2020), 845--879.

\bibitem{ratz2006higher}
{\sc R{\"a}tz, A., and Voigt, A.}
\newblock Higher order regularization of anisotropic geometric evolution
  equations in three dimensions.
\newblock {\em Journal of Computational and Theoretical Nanoscience 3}, 4
  (2006), 560--564.

\bibitem{rindler2018calculus}
{\sc Rindler, F.}
\newblock {\em Calculus of variations}, vol.~5.
\newblock Springer, 2018.

\bibitem{RogSchat}
{\sc R{\"o}ger, M., and Sch{\"a}tzle, R.}
\newblock On a modified conjecture of {D}e {G}iorgi.
\newblock {\em Mathematische Zeitschrift 254\/} (2006), 675--714.

\bibitem{rupp2024global}
{\sc Rupp, F., and Scharrer, C.}
\newblock Global regularity of integral 2-varifolds with square integrable mean
  curvature.
\newblock {\em arXiv preprint arXiv:2404.12136\/} (2024).

\bibitem{Santilli}
{\sc Santilli, M.}
\newblock Normal bundle and {A}lmgren's geometric inequality for singular
  varieties of bounded mean curvature.
\newblock {\em Bull. Math. Sci. 10}, 1 (2020), 2050008, 24.

\bibitem{santilli2021second}
{\sc Santilli, M.}
\newblock Second order rectifiability of varifolds of bounded mean curvature.
\newblock {\em Calculus of Variations and Partial Differential Equations 60}, 2
  (2021), 81.

\bibitem{schatzle2004quadratic}
{\sc Sch{\"a}tzle, R.}
\newblock Quadratic tilt-excess decay and strong maximum principle for
  varifolds.
\newblock {\em Annali della Scuola Normale Superiore di Pisa-Classe di Scienze
  3}, 1 (2004), 171--231.

\bibitem{SchatzLower}
{\sc Sch{\"a}tzle, R.}
\newblock Lower semicontinuity of the {W}illmore functional for currents.
\newblock {\em Journal of Differential Geometry 81}, 2 (2009), 437--456.

\bibitem{Simon_GMT}
{\sc Simon, L.}
\newblock {\em Lectures on geometric measure theory}, vol.~3 of {\em
  Proceedings of the Centre for Mathematical Analysis, Australian National
  University}.
\newblock Australian National University, Centre for Mathematical Analysis,
  Canberra, 1983.

\bibitem{spencer2004asymptotic}
{\sc Spencer, B.~J.}
\newblock Asymptotic solutions for the equilibrium crystal shape with small
  corner energy regularization.
\newblock {\em Physical Review E 69}, 1 (2004), 011603.

\bibitem{Lowen}
{\sc Torabi, S., Lowengrub, J., Voigt, A., and Wise, S.}
\newblock A new phase-field model for strongly anisotropic systems.
\newblock {\em Proceedings of the Royal Society A: Mathematical, Physical and
  Engineering Sciences 465}, 2105 (2009), 1337--1359.

\bibitem{wise2007solving}
{\sc Wise, S., Kim, J., and Lowengrub, J.}
\newblock Solving the regularized, strongly anisotropic cahn--hilliard equation
  by an adaptive nonlinear multigrid method.
\newblock {\em Journal of Computational Physics 226}, 1 (2007), 414--446.

\end{thebibliography}
\end{document}